\documentclass{article}

\usepackage[T1]{fontenc}
\usepackage[latin1]{inputenc}
\usepackage{mathtools,amsthm,amssymb}
\usepackage{mathabx}
\usepackage{mathrsfs}
\usepackage{subcaption, graphicx}

\usepackage[comma, authoryear]{natbib}

\usepackage{xcolor}
\usepackage{authblk}

\usepackage{hyperref}

\usepackage[a4paper]{geometry}

\newtheorem{prop}{Proposition}[section]
\newtheorem{cor}[prop]{Corollary}

\newtheorem{lem}[prop]{Lemma}

\newtheorem{thm}[prop]{Theorem}

\theoremstyle{remark}
\newtheorem{remark}[prop]{Remark}

\newcommand{\dif}{\mathrm{d}}
\newcommand{\Cov}{\operatorname{Cov}}

\newcommand{\vertiii}[1]{{\left\vert\kern-0.25ex\left\vert\kern-0.25ex\left\vert #1
    \right\vert\kern-0.25ex\right\vert\kern-0.25ex\right\vert}}

\newcommand{\E}{\EE}
\newcommand{\Var}{\operatorname{Var}}

\newcommand{\EE}{\mathbb{E}}

\newcommand{\NN}{\mathbb{N}}
\newcommand{\PP}{\mathbb{P}}

\newcommand{\RR}{\mathbb{R}}

\newcommand{\bouncy}{\mathrm{B}}
\newcommand{\zigzag}{\mathrm{Z}}

\def \PP{\mathbb{P}}
\def \RR{\mathbb{R}}

\def \EE{\mathbb{E}}

\numberwithin{equation}{section}

\begin{document}

\title{High-dimensional scaling limits of piecewise deterministic sampling algorithms}

\author[$1$]{Joris Bierkens}
\author[$2$]{Kengo Kamatani}
\author[$3$]{Gareth O. Roberts}
\affil[$1$]{{\small Delft University of Technology}}
\affil[$2$]{{\small Osaka University}}
\affil[$3$]{{\small University of Warwick}}
\date{}

\maketitle

\begin{abstract}
Piecewise deterministic Markov processes are an important new tool in the design of Markov Chain Monte Carlo algorithms. Two examples of fundamental importance are the Bouncy Particle Sampler (BPS) and the Zig-Zag process (ZZ). In this paper scaling limits for both algorithms are determined. Here the dimensionality of the space tends towards infinity and the target distribution is the multivariate standard normal distribution. For several quantities of interest (angular momentum, first coordinate, and negative log-density) the scaling limits show qualitatively very different and rich behaviour. Based on these scaling limits the performance of the two algorithms in high dimensions can be compared. Although for angular momentum both processes require only a computational effort of $O(d)$ to obtain approximately independent samples, the computational effort for negative log-density and first coordinate differ: for these BPS requires $O(d^2)$ computational effort whereas ZZ requires $O(d)$. Finally we provide a criterion for the choice of the refreshment rate of BPS. 
\end{abstract}


\section{Introduction}

Piecewise deterministic Markov processes (PDMPs, \cite{Davis1984}
	) have turned out to be of substantial interest for Monte Carlo analysis, see, for example \cite{BouchardCote2017,pakman2016stochastic,BierkensFearnheadRoberts2016,vanetti2017piecewise}, which have particularly focused on potential for applications in Bayesian statistics, although their uses are far wider, see for example \cite{michel2014generalized,peters2012rejection} for applications in physics. 
However, there are still substantial gaps in our understanding of their theoretical properties.  Even results about the ergodicity of these methods (including irreducibility and exponential ergodicity problems) often involve intricate and complex problems \cite{MR2385873,Deligiannidis2017,BierkensZitt2017}.

The two main PDMP methodologies for Monte Carlo algorithms
are the Zig-Zag (\cite{BierkensFearnheadRoberts2016}) and the Bouncy Particle Sampler (BPS) (\cite{BouchardCote2017}), and we refer to these papers for applications of these methods.  Interesting hybrid strategies are certainly possible but are currently under-explored. 
The important practical question for Monte Carlo practitioners concerns which methodology should be chosen, with currently available empirical comparisons giving mixed results.

The focus of the present paper is on shedding some light on these questions by providing a high-dimensional analysis of these two classes of PDMPs. Our approach will identify weak limits of PDMP chains (suitably speeded up) as dimension goes to infinity. Such analyses are of interest in connection with computational cost estimation of Monte Carlo methods (see for example \cite{RGG,roberts2014complexity}).

\subsection{Piecewise deterministic Markov processes}

We shall consider two particular classes of PDMPs (Zig-Zag and BPS)  which have proved to be valuable for Monte Carlo sampling. Their
 constructions begin in the same way. We are interested in sampling from
 a target distribution which has density  $Z_d^{-1} \exp(-\Psi^d(\xi))$
 with respect to $d$-dimensional Lebesgue measure with  normalising constant 
 \begin{equation}\label{eq:normalizing_constant}
	Z_d=\int_{\mathbb{R}^d} \exp(-\Psi^d(\xi))~\dif \xi<\infty \ . 
\end{equation}
 Zig-Zag and BPS proceed by augmenting this space to include an independent velocity variable taking values uniformly on a prescribed space $ \Theta \subset \mathbb{R}^d$.  Both algorithms define Piecewise deterministic Markov dynamics which preserve this extended target distribution on the augmented state space $E^d=\mathbb{R}^d\times \Theta $. The difference between Zig-Zag and BPS lies in the choice of $\Theta$ and the dynamics for moving between velocities.
 
 For both algorithms we shall make use of independent standardised homogeneous Poisson measures, $N$ say, on $\mathbb{R}_+\times\mathbb{R}_+$, so that $\mathbb{E}[N(\dif t,\dif x)]=\dif t\dif x$. In our notations we will use a superscript $\zigzag$ to indicate the Zig-Zag process, and a superscript $\bouncy$ to refer to the Bouncy Particle Sampler.

 \subsubsection{Zig-Zag sampler} 
 \label{sec:ZZ}
 For the Zig-Zag sampler  the set of possible directions is given by 
$$
\Theta = \mathfrak{C}^{d-1} := \{-1,+1\}^d\ ,
$$
with $\chi_d$ denoting the uniform distribution on $\mathfrak{C}^{d-1}$, and constructs a Markov chain on  the state space
$E^{\zigzag,d}= \mathbb{R}^d\times  \mathfrak{C}^{d-1}$. Here the scaling of the velocities is chosen such that $\Theta$ is a subset of the unit sphere, to enable a more direct comparison with the Bouncy Particle Sampler later on.
%
%
Let $\lambda^{\zigzag,d} = (\lambda^{\zigzag,d}_1,\ldots,\lambda^{\zigzag,d}_d):E^{\zigzag,d}\rightarrow \mathbb{R}_+^d$. The Zig-Zag sampler 
with the jump rate $\lambda^{\zigzag,d}$
generates a Markov process $\{x_t^{\zigzag,d}=(\xi_t^{\zigzag,d},v_t^{\zigzag,d})\}_{t\ge 0}$ 
on $E^{\zigzag,d}$ such that
$$
\xi_t^{\zigzag,d}~=~\xi_0^{\zigzag,d}~+~\int_0^t v_s^{\zigzag,d}~\dif s, \quad  (t\ge 0), 
$$
and $v_t^{\zigzag,d}=(v_{1,t}^{\zigzag,d},\ldots,v_{d,t}^{\zigzag,d})$ is defined by 
$$
v_{i,t}^{\zigzag,d}~=~v_{i,0}^{\zigzag,d}-2\int_{(0,t]\times\mathbb{R}_+} v_{i,s-}^{\zigzag,d}~1_{\{z\le \lambda^{\zigzag,d}_i(x_{s-}^{\zigzag,d})\}}~N^i(\dif s,\dif z)\ (t\ge 0,~i~=~1,\ldots, d) 
$$
for independent Poisson processes $N^1, \ldots N^d$,
where $x_0^{\zigzag,d}=(\xi_0^{\zigzag,d},v_0^{\zigzag,d})$ is an $E^{\zigzag,d}$-valued random variable. 

\subsubsection{Bouncy Particle Sampler}\label{subsubsec:BPS}
 For the Bouncy Particle Sampler the set of possible directions is given by 
$$\Theta := \mathfrak{S}^{d-1}=\{x\in\mathbb{R}^d:\|x\|^2=1\}$$
with $\psi_d$ denoting the uniform distribution on $\mathfrak{S}^{d-1}$,
and constructs a Markov chain on  the state space
$E^{\bouncy,d}=  \mathbb{R}^d\times \mathfrak{S}^{d-1}$.
Let $\kappa^d:E^{\bouncy,d}\rightarrow \mathfrak{S}^{d-1}$ be a function and let $\lambda^{\bouncy,d}
: E^{\bouncy,d}\rightarrow \mathbb{R}_+$. 
Then BPS with the jump rate $\lambda^{\bouncy,d}$ and the refreshment rate $\rho^d>0$ generates a Markov process $\{x_t^{\bouncy,d}=(\xi_t^{\bouncy,d},v_t^{\bouncy,d})\}_{t\ge 0}$  defined by 
$$
\xi_t^{\bouncy,d}~=~\xi_0^{\bouncy,d}~+~\int_0^t v_s^{\bouncy,d}~\dif s, \quad  (t\ge 0), 
$$
and $v_t^{\bouncy,d}$ is defined by 
\begin{align*}
v_t^{\bouncy,d}~&=~v_0^{\bouncy,d}~+~\int_{(0,t]\times\mathbb{R}_+} (\kappa^d(x_{s-}^{\bouncy,d})-v^{\bouncy,d}_{s-})~1_{\{z\le \lambda^{\bouncy,d}(x_{s-}^{\bouncy,d})\}}~N(\dif s,\dif z)\\
&\quad~+~\int_{(0,t]\times\mathfrak{S}^{d-1}}(u-v_{s-}^{\bouncy,d})~R_d(\dif s,\dif u)
\end{align*}
where $R_d$ is a homogeneous random measure which is independent from $N$ with intensity measure 
$$
\mathbb{E}[R_d(\dif s,\dif u)]=\rho^d~\dif s~\psi_d(\dif u). 
$$
Without refreshment the Bouncy Particle Sampler may not be ergodic in general \cite{BouchardCote2017}. 
The refreshment rate using the random measure $R_d$ was referred to as restricted refreshment in \cite{BouchardCote2017}, and 
other choices were also considered in that paper.

\subsection{Finite dimensional properties}

In this section we briefly review finite dimensional properties of the piecewise deterministic processes. 
Here and elsewhere, we denote the $d$-dimensional Euclidean inner product by $\langle x,y\rangle =\sum_{i=1}^dx_iy_i$  and the norm by $\|x\|=(\langle x,x\rangle)^{1/2}$.

Let $F_i(v)$ be the function that switches the sign of the $i$-th element of $v\in\mathfrak{C}^{d-1}$.  
By Theorem II.2.42 of \cite{JS} and Proposition VII.1.7 of \cite{MR1725357}, 
the infinitesimal generator $L^{\zigzag,d}$ of the Markov process corresponding to the Zig-Zag sampler is 
defined by 
$$
(L^{\zigzag,d}\varphi)(\xi,v)=\sum_{i=1}^d~\frac{\partial \varphi}{\partial \xi_i}(\xi,v)~v_i~+~\sum_{i=1}^d~\lambda_i^{\zigzag,d}(\xi,v)~(\varphi(\xi,F_i(v))-\varphi(\xi,v))
$$
for $\varphi:E^{\zigzag,d}\rightarrow \mathbb{R}$ such that $\varphi(\cdot,v)\in C^1_0(\mathbb{R}^d)\ (v\in\mathfrak{C}^{d-1})$ where $C^1_0(\mathbb{R}^d)$ is the set of differentiable functions with compact support. 
Let $\Psi^d:\mathbb{R}^d\rightarrow\mathbb{R}_+$ be a smooth function with (\ref{eq:normalizing_constant}). 
Set $\lambda^{\zigzag,d}(x)$ so that $\lambda^{\zigzag,d}_i(\xi, v)-\lambda^{\zigzag,d}_i(\xi,F_i(v))=\partial_i\Psi^d(\xi)v_i$. 
As discussed in, for example, \cite{BierkensFearnheadRoberts2016,BierkensZitt2017}, the Markov process corresponding to the Zig-Zag sampler is $\Pi^{\zigzag,d}$-invariant.

%

The infinitesimal generator $L^{\bouncy,d}$ of the Markov process corresponding to the Bouncy Particle Sampler is 
defined by 
\begin{align*}
(L^{\bouncy,d}\varphi)(\xi,v)&=\left\langle\nabla_\xi \varphi(\xi,v),v\right\rangle~+~~\lambda^{\bouncy,d}(\xi,v)~(\varphi(\xi,\kappa^d(\xi,v))-\varphi(\xi,v))\\
&\quad+~\rho^d~\left(\int \varphi(\xi,u)\psi_d(\dif u)-\varphi(\xi,v)\right)	
\end{align*}
for continuous functions $\varphi:E^{\bouncy,d}\rightarrow\mathbb{R}$ satisfying
$\varphi(\cdot,v)\in C_0^1(\mathbb{R}^d)\ (v\in\mathfrak{S}^{d-1})$.  Here, 
$\nabla_\xi=(\partial/\partial \xi_i)_{i=1,\ldots, d}$ is the derivative operator and we will denote it by $\nabla$ when there is no ambiguity. 
We assume a constant refreshment rate, that is $\rho^d
\equiv \rho>0$, and $\kappa^d$ is a reflection function defined by 
\begin{equation}\label{eq:reflection_function}
	\kappa^d(\xi,v)=v-2\frac{\langle \nabla \Psi^d(\xi), v\rangle}{\|\nabla \Psi^d(\xi)\|^2}\nabla\Psi^d(\xi)
\end{equation}
and finally $\lambda^{\bouncy,d}(\xi, v)=\max\{\langle \nabla\Psi^d(\xi),v\rangle,0\}$. 
As discussed in, for example, \cite{BouchardCote2017,Deligiannidis2017} the Markov process corresponding to the Bouncy Particle Sampler is $\Pi^{\bouncy,d}$ invariant.

%

\subsection{Summary of the main results}

In Section \ref{sec:high-dim}, we study the asymptotic properties of piecewise deterministic processes.  This section summarises the main results in that section. 
For simplicity, all results in Section \ref{sec:high-dim} assume that the initial value of $\xi$ is generated from the target distribution, and the initial value of $v$ is generated from the uniform distribution on the direction space. 
To avoid technical difficulties, we only consider the standard normal case, that is, 
$$
\Psi^d(\xi)=\frac{\|\xi\|^2}{2}.
$$
In agreement with this assumption, the jump rate of the Zig-Zag sampler is 
$$
\lambda^{\zigzag,d}_i(\xi,v)=\max\{\xi_i v_i,0\}=(\xi_iv_i)^+, \quad i=1,\ldots, d,\ (\xi, v) \in E^{\zigzag,d},
$$
and the jump rate and the refreshment rate of the Bouncy Particle Sampler are
$$
\lambda^{\bouncy,d}(\xi,v)=\max\{\langle \xi ,v\rangle,0\} = \langle \xi, v \rangle^+,~\rho^d(\xi,v)=\rho>0,\quad (\xi, v) \in E^{\bouncy,d},
$$
and the reflection function satisfies (\ref{eq:reflection_function}). 
Analogous to \cite{RGG}, we focus on relevant finite dimensional summary statistics. 
The \emph{angular momentum} process, the \emph{negative log-target density process}  and the \emph{first coordinate process} are defined by 
$$
t\mapsto \left\langle \xi_t, \frac{v_t}{\|v_t\|}\right\rangle,\quad
t\mapsto d^{1/2}(d^{-1}\|\xi_t\|^2-1),\quad 
t\mapsto \xi_{1,t},
$$
respectively, for both the Zig-Zag sampler and the Bouncy particle sampler. As $d \rightarrow \infty$, the stationary distributions of these statistics converge to centered normal distributions (with variances 1, 2 and 1, respectively).
We compare the convergence rates of the Zig-Zag sampler (ZZ) and the Bouncy Particle Sampler (BPS) for these summary statistics. 
Table~\ref{tab:convergence-rates} summarises the results.   
 
\begin{table}[ht!]
\begin{center}
\begin{tabular}{ l  c  c  c}
  Method & Angular momentum & Negative log-density & $1$st Coordinate \\
  \hline			
  ZZ & $O(1)$ (Thm.~\ref{theo:zigzag_limit}) & $O(1)$ (Thm.~\ref{theo:zigzag_integral_limit}) & $O(1)$ (Thm.~\ref{theo:zigzag_first_component})\\
  BPS & $O(1)$ (Thm.~\ref{theo:bps_limit}) & $O(d)$ (Thm.~\ref{theo:bps_ll}) & $O(d)$ (Thm.~\ref{theo:bps_coordinate})\\
\end{tabular}
\caption{Size of continuous time intervals required to obtain approximately independent samples for the piecewise deterministic processes.}
\label{tab:convergence-rates}
\end{center}
\end{table}

The computational effort per unit time of the processes is proportional to the number of switches per unit time interval, multiplied by the computational effort per switch. For Zig-Zag and BPS, these are as given in Table~\ref{tab:computational-effort}.  
For the Zig-Zag sampler it is in principle possible to obtain higher efficiency ($O(1)$ computational effort per event instead of $O(d)$) but this depends on specialised  independence structure in the target distribution. In particular for the case of statistically independent targets as studied theoretically in most of this paper, the Zig-Zag can be implemented with the higher efficiency described in the top row of Table~\ref{tab:computational-effort}.
Moreover in situations in which there is a specific (sparse) conditional independence structure,  the Zig-Zag sampler can again be designed to benefit from this structure to give complexity according to
the top row of Table~\ref{tab:computational-effort}. On the other hand we do not see a way in which the generic BPS as described in this paper can utilise conditional independence. However it is worth noting that generalisations of Zig-Zag termed {\em Local BPS} by  \cite{BouchardCote2017}  and other variants as discussed in  \cite{peters2012rejection}
also share computational advantages from sparse conditional independence.
However in the context of a general partial correlation structure, implementation costs are an order of magnitude greater for the Zig-Zag (as is the case for relevant competitor algorithms such as MALA and HMC). Thus we give two complexities for Zig-Zag in Table~\ref{tab:computational-effort} which can be thought of as best and worst cases according to the above discussion.
\begin{table}[ht!]
\begin{center}
\begin{tabular}{ l  c  c  c}
  Method & $\sharp$ events/unit time & Comp. effort/event & Combined effort/unit time \\
  \hline			
  ZZ (with independence) & $O(d)$ (Cor.~\ref{cor:zigzag_switches}) & $O(1)$ & $O(d)$ \\
  ZZ (general case) & $O(d)$ (Cor.~\ref{cor:zigzag_switches}) & $O(d)$ & $O(d^2)$ \\
  BPS & $O(1)$ (Cor.~\ref{cor:bps_switches}) & $O(d)$ & $O(d)$\\
\end{tabular}
\caption{Computational effort of the piecewise deterministic processes.}
\label{tab:computational-effort}
\end{center}
\end{table}

In order to obtain the algorithmic complexity required to draw approximately independent samples, we should multiply the required continuous time scaling with the computational complexity per continuous time unit. By doing so, we obtain the algorithmic complexities of the ZZ and BPS as listed in Table~\ref{tab:combined}. 

\begin{table}[ht!]
\begin{center}
\begin{tabular}{ l  c  c  c}
  Method & Angular momentum & Negative log-density & $1$st Coordinate \\
  \hline			
  ZZ (with independence) & $O(d)$  &  $O(d)$  & $O(d)$ \\
    ZZ (general case) & $O(d^2)$  &  $O(d^2)$  & $O(d^2)$ \\
  BPS & $O(d)$  & $O(d^2)$ & $O(d^2)$ \\
\end{tabular}
\caption{Algorithmic complexity to obtain approximately independent samples. }
\label{tab:combined}
\end{center}
\end{table}

%
In terms of which algorithm, BPS or Zig-Zag should be implemented in any specific situation, the conclusions to the findings of Table~\ref{tab:combined} tentatively suggest that in the context of sparse conditional independence structure the Zig-Zag seems to have better complexity properties, but that for general target densities the methods have the same complexity. Of course these conclusions need to be treated with caution given the relatively specialised nature of the theory which underpins 
Table~\ref{tab:convergence-rates}.

Analogous to \cite{RGG}, we also obtain the optimal choice of the refreshment jump rate $\rho$. The limiting process of the negative log-target density of the BPS sampler is the Ornstein-Uhlenbeck process. The process attains the optimal convergence rate 
when the ratio of the expected number of refreshment jumps to that of all jumps is 
approximately $0.7812$ (see Figure \ref{fig:functions}). This result provides a practical criterion for selecting the refreshment rate. In Section \ref{sec:optimal}, we analyse this criterion for more general target probability distributions.

Asymptotic limit results illustrate some similarities and differences with the Metropolis-Hastings (MH) algorithm. Typically, high-dimensional limiting processes of MH algorithms are diffusions \cite{RGG,MR1888450}. In contrast,
the first two summary statistics processes of ZZ converge to non-Markovian Gaussian processes and the $1$st coordinate process of ZZ and the angular momentum process of BPS have pure jump process limits. 
At the same time, like MH algorithms, our results show that 
the piecewise deterministic processes can exhibit diffusive behaviour. In particular the latter two summary statistics processes for BPS have diffusion limits. 
Diffusion limits are known for PDMPs \cite{BouguetCloez2018,Fontbona2015}, but have to our knowledge not been established for dimension tending to infinity.

In this paper, we only consider the standard normal distribution except Section \ref{sec:optimal}. Experimental results of Section~\ref{sec:experiments} suggest that the obtained results remain valid for general distributions of product form.  For non-product strongly correlated distributions such as in \cite{arXiv:1412.6231} the convergence rates could be different. This remains a topic of active research. See also Section \ref{sec:discussion}. Also, throughout in this paper, we assume stationarity of the process. This assumption can be weakened with some extra work for moment calculus. However, their behaviours will be different from the current study if the initial distribution is far from the centre region of the target distribution. See \cite{MR2137324, MR3349007} for the case of Markov chain Monte Carlo methods. This also remains a topic of active research.

%

\section{High-dimensional properties}\label{sec:high-dim}

We analyse high-dimension properties of the Zig-Zag and  BPS samplers. 
Throughout in this paper, we assume strong stationarity of the Markov processes.  
Our first main objective is the analysis of 
the angular momentum processes
\begin{align*}
S_t^{\zigzag,d}&=\left\langle \xi_t^{\zigzag,d},\frac{v_t^{\zigzag,d}}{\|v_t^{\zigzag,d}\|}\right\rangle=
d^{-1/2}\left\langle\xi_t^{\zigzag,d},v_t^{\zigzag,d}\right\rangle\\
S_t^{\bouncy,d}&=\left\langle \xi_{t}^{\bouncy,d},\frac{v_{t}^{\bouncy,d}}{\|v_{t}^{\bouncy,d}\|}\right\rangle=\left\langle \xi_{t}^{\bouncy,d},v_{t}^{\bouncy,d}\right\rangle. 
\end{align*}
The behaviour of the 
angular momentum processes 
illustrates the dissimilarity of the Zig-Zag and BPS samplers. 

The angular momentum processes do not completely capture the asymptotic properties of the Markov processes. For the understanding of long-time properties, it is more natural to consider the behavior of the negative log-target density. Observe that there is an interesting connection between the angular momentum process and the negative log-target density processes: 
$$
\dif \|\xi_t^{\zigzag,d}\|^2 =2 d^{1/2}S_{t}^{\zigzag,d}~\dif t,\ 
\dif \| \xi_t^{\bouncy,d} \|^{2}= 2S_t^{\bouncy,d}~\dif t. 
$$
Additionally we will study the number of switches (jumps)  
$$
\sum_{0\le t\le T}1_{\{\Delta S_t^{\zigzag,d}\neq 0\}},\ 
\sum_{0\le t\le T}1_{\{\Delta S_t^{\bouncy,d}\neq 0\}}
$$
up to $T>0$, where  $\Delta X_t=X_t-X_{t-}$.  
Finally, we will check the convergence rates for the coordinate processes. 

\begin{remark}[Proof strategy]
In the high-dimensional MCMC literature, 
as in \cite{RGG}, the Trotter-Kato type approach is the most popular which uses convergence of generators to prove 
convergence of Markov processes.
Classical literature is \cite{MR838085}. 
In this paper, we closely follow the semimartingale characteristics approach taken in \cite{JS}, which is natural to the non-Markovian processes which arise in our analysis. 
See Section IX.2a of \cite{JS} for the connection between the two approaches. 
\end{remark}

\subsection{Asymptotic limit of the Zig-Zag sampler}\label{subsec:ss}

In this section, we study the asymptotic properties of the Zig-Zag sampler. All the proofs are postponed to Appendix \ref{asec:zz}. 
To state the first results, we introduce a stationary piecewise deterministic jump process 
\begin{align}\label{eq:limitzigzag}
\mathcal{T}_t
~& =\mathcal{T}_0+t - 2\int_{(0,t]\times\mathbb{R}_+} \mathcal{T}_{s-}~1_{\{z\le \mathcal{T}_{s-}\}}~N(\dif s,\dif z)
\end{align}
with $\mathcal{T}_0\sim \mathcal N(0,1)$. 
The process has the infinitesimal generator
\begin{equation}\label{eq:generator_t}
	Gf(x)=f'(x)+x^+(f(-x)-f(x)). 
\end{equation}
Therefore $\mathcal N(0,1)$ is the invariant distribution of $\mathcal{T}=(\mathcal{T}_t)_{t\ge 0}$ by Proposition 4.9.2 of \cite{MR838085}. In particular, 
$\mathcal{T}$ is a strictly stationary process. 
Set 
\begin{equation} \label{eq:covariance} K(s,t)=\mathbb{E}[\mathcal{T}_s\mathcal{T}_t].\end{equation}
This covariance kernel will play an important role in this work. Some properties are collected in Proposition \ref{prop:covariance}.

%
Our first result is the asymptotic limit of the angular momentum process $S^{\zigzag,d}=(S_t^{\zigzag,d})_{t\ge 0}$. 
The limiting process is a non-Markovian Gaussian process, unlike most of the scaling limit results related to Markov chain Monte Carlo methods. We also discuss sample path continuity. See \cite[Eq (2.2.8)]{Karatzas1991} or Section \ref{asec:zz_angle} for the definition of local $\alpha$-H\"older continuity. 

\begin{thm}\label{theo:zigzag_limit}
The process $S^{\zigzag,d}$
converges to $S^{\zigzag}=(S_t^{\zigzag})_{t\ge 0}$ in distribution in Skorohod topology where 
$S^{\zigzag}$ is the non-Markovian stationary Gaussian process with mean $0$ and covariance function $K(s,t)$. The Gaussian process is locally $\alpha$-H\"older continuous for any $\alpha \in (0,1/2)$, but it is not locally 
$\alpha$-H\"older continuous for any $\alpha\ge 1/2$. 
\end{thm}

The second result concerns the number of switches for the Zig-Zag process, indicating the computational cost of the process. The following results show that the process $S^{\zigzag,d}$, the number of switches 
per unit time is $O(d)$. 

\begin{cor}\label{cor:zigzag_switches}
The  number of switches of $S^{\zigzag,d}$ over a time interval $(0,T]$ scaled by $d^{-1}$ satisfies
$$
d^{-1}\sum_{0\le t\le T}1_{\{\Delta S_t^{\zigzag,d}\neq 0\}}\longrightarrow \frac{T}{\sqrt{2\pi}} \quad \text{in probability as $d \rightarrow \infty$}
$$
for any $T>0$. 
\end{cor}

The third result is the analysis of the negative log-target density process. As for the angular momentum process, the limiting process is a non-Markovian Gaussian process. We also discuss the sample path property. We call a process differentiable if there is a modification such that each path is differentiable almost surely. See Section \ref{asec:zz_coordinate} for the definition. 

\begin{thm}\label{theo:zigzag_integral_limit}
The negative log-density process
$$
Y_t^{\zigzag,d}:=\sqrt{d}\left(\frac{\|\xi_t^{\zigzag,d}\|^2}{d}-1\right)
$$
converges to a stationary Gaussian process $Y^{\zigzag}$ with mean $0$ and covariance function
$$
L(s,t)=2-2\int_{s}^t\int_{s}^t K(u,v)\, \dif u\, \dif v.
$$
Moreover the Gaussian process $Y^{\zigzag}$ is differentiable with respect to the time index $t$. 
\end{thm}

Finally we consider the first coordinates of $\xi$. Let 
$$
\pi_k(\xi)=\pi_k^d(\xi)=(\xi_1,\ldots, \xi_k) \quad  \text{for} \, \xi=(\xi_1,\ldots, \xi_d)\in\mathbb{R}^d,
$$
denote the operation of taking the first $k \in \{1,\ldots, d\}$ components of a $d$-dimensional vector. If $k>d$, then we set
$$
\pi_k(\xi)=\pi_k^d(\xi)=(\xi_1,\ldots, \xi_d, \overbrace{0,\ldots, 0}^{k-d}). 
$$
Let $\phi_k(x)$ be the density of the $k$-dimensional standard normal distribution $\mathcal N(0,I_k)$. 

\begin{thm}\label{theo:zigzag_first_component}
For any $k\in\mathbb{N}$ and $d \geq k$, the law of the process $Z_t^{\zigzag, d, k}:=\pi_k(\xi_{t}^{\zigzag,d})$
does not depend on $d$, and $Z^{\zigzag,d,k}$ is an ergodic process. 
In particular, for any $\mathcal N(0,I_k)$-integrable function $f:\mathbb{R}^k\rightarrow\mathbb{R}$, we have 
$$
\frac{1}{T}\int_0^T f (Z_t^{\zigzag,d,k})~\dif t\longrightarrow \int_{\mathbb{R}^k} f(x)\phi_k(x)\dif x\ \text{in\ probability as $T \rightarrow \infty$}.
$$

Finally we remark on the joint limit process of the angular momentum, the negative log-density and the $1$st coordinate processes. The joint process of the first two processes has a Gaussian limit by the central limit theorem of the processes. The diagonal components of the corresponding covariance kernel are $K(s,t)$ and $L(s,t)$. The off-diagonal component is the covariance of the angular momentum and the log negative-density processes. Since $\dif Y_t^{\zigzag,d}=2S_t^{\zigzag,d}$,  off-diagonal component of the corresponding covariance kernel is 
\[
M(s,t):=\mathbb{E}[Y_s^{\zigzag}S_t^{\zigzag}]=2^{-1}\frac{\partial  L(s,t)}{\partial t}=-2\int_s^tK(t,u)\dif u. 
\]
The $1$st coordinate process is asymptotically independent from other processes. 
\end{thm}

\subsection{Asymptotic limit of the Bouncy Particle Sampler}\label{subsec:bps}

In this section, we study the asymptotic properties of the Bouncy Particle Sampler. All the proofs are postponed to Appendix \ref{asec:bps}. 
The limiting process of the angular momentum is represented as 
\begin{equation}
\label{eq:limitBPS}
\begin{split}
S_t^{\bouncy}
~ &=~S_0^{\bouncy} + t - 2\int_{(0,t]\times\mathbb{R}_+} S^{\bouncy}_{s-}~1_{\{z\le S^{\bouncy}_{s-}\}}~N(\dif s,\dif z)\\
&\quad+\int_{(0,t]\times\mathbb{R}}(z-S^{\bouncy}_{s-})
~R(\dif s,\dif z)
\end{split}
\end{equation}
where $R$ is the random measure with the intensity measure  
$$
\mathbb{E}[R(\dif s,\dif z)]~=~\rho~\dif s~\phi(z) \dif z
$$
where $\phi$ denotes the $\mathcal N(0,1)$ density function. 
The process $S^{\bouncy}=(S^{\bouncy}_t)_{t\ge 0}$ has the infinitesimal generator
\begin{equation}\label{eq:BPSgenerator}
	Hf(x)=f'(x)+x^+(f(-x)-f(x))+\rho\left(\int_\mathbb{R} \phi(y)f(y)\, \dif y-f(x)\right). 
\end{equation}
The process is $\mathcal N(0,1)$-invariant by Proposition 4.9.2 of \cite{MR838085}. 

\begin{thm}\label{theo:bps_limit}
The process $S^{\bouncy,d}$
converges in law to $S^{\bouncy}$. 
\end{thm}

%
In fact, if $\rho=0$, then the law of $S^{\bouncy,d}$ is \emph{identical} to the law of $S^{\bouncy}$ for any $d \in \NN$. Indeed,
say $g(\xi,v) = f (\langle \xi,v\rangle)$ for $(\xi, v) \in E^{\bouncy,d}$. Then (without refreshment), for $s = \langle \xi,v \rangle$,
\begin{align*} (L^{\bouncy,d} g)(\xi,v) & = \langle v, \nabla_\xi g(\xi,v) \rangle + \langle \xi, v\rangle^+ (g(\xi,-v) - g(\xi,v)) \\
 & = f'(\langle \xi,v \rangle) \langle v, v \rangle + \langle \xi,v \rangle^+ (f(-\langle \xi, v \rangle) - f(\langle \xi, v \rangle)) \\
 & = f'(s) + (s)^+ (f(-s) - f(s)),
\end{align*}
which establishes that if $\rho = 0$, then $S^{\bouncy,d}$ is a Markov process with generator $H$.

%

\begin{cor}\label{cor:bps_switches}
The expected number of switches of $S^{\bouncy, d}$ over a time interval $(0,T]$  does not depend on $d$ and is given by
$$
\mathbb{E}\left[\sum_{0\le t\le T}1_{\{\Delta S^{\bouncy,d}_t\neq 0\}}\right]= T\left(\frac{1}{\sqrt{2\pi}}+\rho\right).
$$
\end{cor}

Unlike the Zig-Zag sampler, the number of switches is random even in the limit $d\rightarrow\infty$. This is the reason why we consider expectation rather than the limit in Corollary \ref{cor:bps_switches}. 
Note that each switch changes all components of the direction $v$. 
On the other hand, the Zig-Zag sampler only changes one component in each switch. 


\begin{thm}\label{theo:bps_ll}
The normalised negative log-target density process 
$$
Y_t^{\bouncy,d}:=\sqrt{d}\left(\frac{\|\xi^{\bouncy,d}_{dt}\|^2}{d}-1\right)
$$
converges to the stationary Ornstein-Uhlenbeck process $Y^{\bouncy}$
such that 
$$
\dif Y_t^{\bouncy}=-\frac{\sigma(\rho)^2}{4}~Y_t^{\bouncy}~\dif t+\sigma(\rho)~\dif W_t
$$
where 
$$
\sigma(\rho)^2 := 8\int_0^\infty e^{-\rho s} K(s,0)\dif s
$$
with $K(s,0)$ defined in~\eqref{eq:covariance}, and where $(W_t)_{t\ge 0}$ is the one-dimensional standard Wiener process. 
\end{thm}

The speed of convergence of the negative log-target density process is determined by 
$\sigma(\rho)$, and the speed is optimised when $\sigma(\rho)$ assumes its maximum. 

\begin{prop}\label{prop:limit_of_sigma}
The continuous function $\sigma(\rho)^2$ satisfies 
$$
\lim_{\rho\rightarrow+0}\sigma(\rho)^2=8\int_0^\infty K(s,0)\dif s=0,\ 
\lim_{\rho\rightarrow+\infty}\sigma(\rho)^2=0. 
$$
In particular, there exists $\rho^*\in (0,\infty)$ such that $\sigma(\rho^*)^2=\sup_{\rho\in (0,\infty)}\sigma(\rho)^2$. 
\end{prop}

The covariance function $K(t,0)$ and the diffusion coefficients $\sigma(\rho)^2$ do not admit simple  expressions. These functions can be written as infinite sums of convolutions, and numerical evaluation is difficult. On the other hand, simple Monte Carlo calculations yield good estimates of these functions (Figure \ref{fig:functions}). The Monte Carlo estimates also provide that the optimal choice of $\rho^*$ is around $1.424$. The ratio of the expected number of the refreshment jumps to that of overall jumps is 
\begin{equation}\label{eq:optimal_rho}
	\frac{\rho^*}{\frac{1}{\sqrt{2\pi}}+\rho^*}\approx 0.7812.  
\end{equation}
Note that the choice of $\rho$ is not scale invariant, that is, 
if we apply the target distribution with the negative log-density $\Psi^d(\xi)=\|\xi\|^2/(2\gamma^2)$, the optimal choice depends on $\gamma>0$. However the above jump ratio does not depend on the scale so we can use the $78.12\%$ rule as a criterion for the choice of the refreshment rate $\rho>0$.

%
%
%
%

\begin{figure}
\includegraphics[width=0.45\textwidth,bb=0 0 432 216]{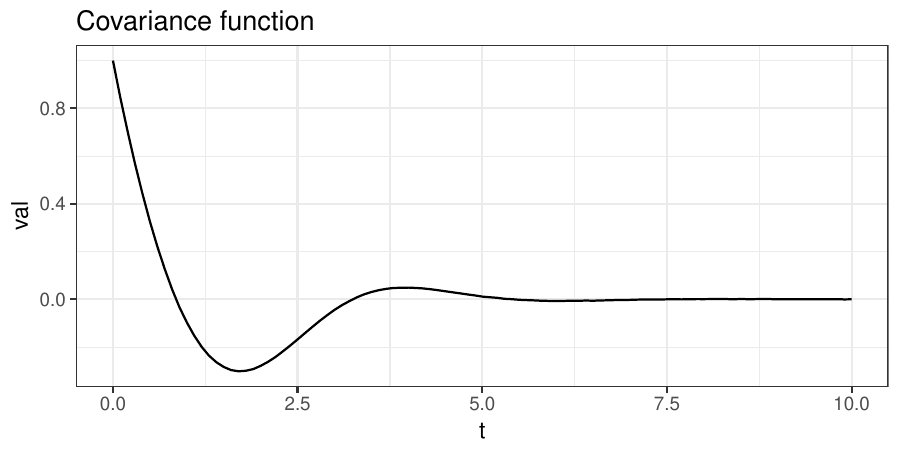}
\includegraphics[width=0.45\textwidth,bb=0 0 432 216]{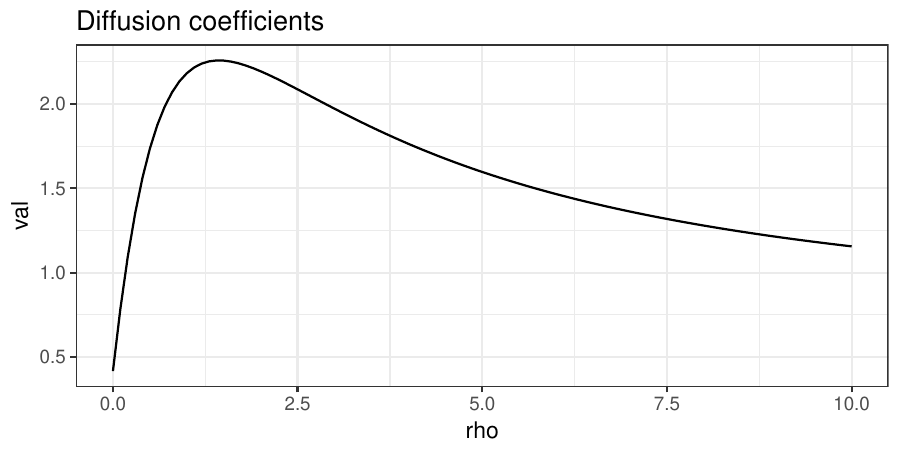}
 \caption{Monte Carlo estimated values of the covariance function $K(t,0)$ (left) and the diffusion coefficient $\sigma(\rho)^2$ (right). }
 \label{fig:functions}
\end{figure}

Finally, we consider the coordinate process convergence for the bouncy particle sampler. 

\begin{thm}\label{theo:bps_coordinate}
For any $k\in\mathbb{N}$, the process $Z^{\bouncy,d,k}=(Z_t^{\bouncy,d,k})_{t\ge 0}$ defined by $Z_t^{\bouncy,d,k}:=\pi_k(\xi_{dt}^{\bouncy,d})$
converges to the stationary Ornstein-Uhlenbeck process $Z^{\bouncy,k}$ satisfying the SDE
$$
\dif Z_t^{\bouncy,k}=-\rho^{-1}Z_t^{\bouncy,k}\dif t +\sqrt{2\rho^{-1}}~\dif W_t^k
$$
for $k\in\mathbb{N}$ where $W^k$ is the $k$-dimensional standard Wiener process. 
In particular, any bounded continuous function $f:\mathbb{R}^k\rightarrow\mathbb{R}$, we have 
\begin{equation}\label{eq:BPS_LLN}
\frac{1}{T}\int_0^T f(Z_t^{\bouncy,d,k})\dif t\longrightarrow_{d,T\rightarrow\infty} \int_{\mathbb{R}^k} f(x)\phi_k(x)\dif x\ \mathrm{in\ probability}.
\end{equation}
\end{thm}

Note that (\ref{eq:BPS_LLN}) is a double limit. In other words, for all $\varepsilon >0$ and $\gamma > 0$, there is a $K > 0$ such that for all $d > K$, $T > K$, 
\[ \PP \left( \left| \frac{1}{T}\int_0^T f(Z_t^{\bouncy,d,k})\dif t- \int_{\mathbb{R}^k} f(x)\phi_k(x)\dif x\right| > \gamma \right)<\varepsilon.\] This means that the limits with respect to $d$ and $T$ can be freely interchanged.
The robustness of the result in terms of the choice of $d$ and $T$ is important for Markov chain Monte Carlo analysis since the practitioner may use $T=d^2$ or $T=d^{10}$, or even $T=\sqrt{d}$. 

Finally we remark on the joint convergence of the angular momentum, the negative log-density and the $1$st component processes as in the Zig-Zag sampler case. Unlike the Zig-Zag sampler case, the angular momentum do not share the  time scaling with other processes. Therefore, only useful joint process is the combination of the negative log-density and the $1$st component processes, and these two processes are asymptotically independent.

\subsection{Optimal choice of the refreshment ratio for the bouncy particle sampler}\label{sec:optimal}

In Section \ref{subsec:bps} we discussed the optimal choice of the refreshment ratio which maximises the diffusion coefficient. However, it is also possible to estimate the coefficient directly. 
Let $0=\sigma_0<\sigma_1<\ldots<\sigma_N$ be the refreshment times until $T>0$. 
The diffusion coefficient can be estimated by
$$
\hat{\sigma}^2_N(\rho):=4\rho~N^{-1}\sum_{n=1}^N(\Psi^d(\xi_{\sigma_n}^{\bouncy,d})-\Psi^d(\xi_{\sigma_{n-1}}^{\bouncy,d}))^2
$$
and we may treat it as an efficiency criterion of the bouncy particle sampler. 

For non-Gaussian, non-i.i.d. case, the meaning of the coefficient is unclear. However, we may still treat it as a criterion since if the value is large, we expect that the process moves relatively well. So we want to know the property of the coefficient other than the standard Gaussian case. 

For the general case, we still assume stationarity for the process, and assume the following. Let $\Psi^d$ be a thrice differentiable function, and assume the Lipschitz type condition
\begin{align}
\label{eq:Lipschitz_bound_delta}&\|\nabla\Psi^d(x)-\nabla\Psi^d(y)\|\le l(\|x-y\|), 
\end{align}
where $l:\mathbb{R}_+\rightarrow\mathbb{R}_+$ is a non-decreasing function. We use notation
\[
\nabla^2\Psi^d(\xi)[u,v]=\sum_{i=1}^d\frac{\partial^2\Psi^d(\xi)}{\partial \xi_i\partial \xi_j}u_iv_j,\ 
\nabla^3\Psi^d(\xi)[u,v,w]=\sum_{i=1}^d\frac{\partial^3\Psi^d(\xi)}{\partial \xi_i\partial \xi_j\partial \xi_k}u_iv_jw_k
\]
and $\nabla^2\Psi^d(\xi)[u^{\otimes 2}]=\nabla^2\Psi^d(\xi)[u,u]$, 
and  $\nabla^3\Psi^d(\xi)[u^{\otimes 3}]=\nabla^3\Psi^d(\xi)[u,u,u]$. 
We also assume consistency conditions
\begin{align}
\label{eq:Consistency_delta}&
\mathbb{E}\left[\left|\frac{\|\nabla\Psi^d(\xi_0^{\bouncy,d})\|^2}{d}-H\right|\right]~\longrightarrow_{d\rightarrow\infty}~0,\\ 
\label{eq:Consistency_nabla}&
\mathbb{E}\left[\left|\nabla^2\Psi^d(\xi_0^{\bouncy,d})[(v_0^{\bouncy,d})^{\otimes 2}]-H\right|\right]~\longrightarrow_{d\rightarrow\infty}~0
\end{align}
for $H>0$. The following non-explosive condition is also assumed: 
\begin{align}
\label{eq:Uniform_bound_nabla}
\sup_{\xi\in\mathbb{R}^d}\sup_{u\in\mathfrak{S}_{d-1}}\left|\nabla^2\Psi^d(\xi)[u^{\otimes 2}]\right|<C,\ \sup_{\xi\in\mathbb{R}^d}\sup_{u\in\mathfrak{S}_{d-1}}\left|\nabla^3\Psi^d(\xi)[u^{\otimes 3}]\right|<C	
\end{align}
for some $C>0$ and $\nabla^2\Psi^d(\xi)$ is positive definite matrix for each $\xi\in\mathbb{R}^d$. 

For i.i.d. scenario, we have an expression 
\[
\Psi^d(\xi)=\sum_{i=1}^d\psi(\xi_i)\ (\xi=(\xi_1,\ldots, \xi_d)
\]
for negative log density $\psi$ of some one-dimensional probability measure. Then \eqref{eq:Consistency_delta} is implied by 
\[
Z^{-1}\int_{-\infty}^\infty\psi'(\xi)^2e^{-\psi(\xi)}\dif\xi:=H<\infty
\]
by the law of large numbers. 
Also \eqref{eq:Consistency_nabla} is implied by the integration by parts formula
\[
Z^{-1}\int_{-\infty}^\infty\psi''(\xi)e^{-\psi(\xi)}\dif\xi=H
\]
by the law of large numbers together with the fact that $d(v_0^{\bouncy, d})^2$ follows the Beta distribution with parameters  $1/2$ and $(d-1)/2$. Therefore \eqref{eq:Consistency_delta} and \eqref{eq:Consistency_nabla} are usual regularity conditions for i.i.d. scenario. The last condition \eqref{eq:Uniform_bound_nabla} is also a usual regularity condition which is implied by $\sup_{\xi\in\mathbb{R}}|\psi''(\xi)|<\infty$ and $\sup_{\xi\in\mathbb{R}}|\psi'''(\xi)|<\infty$. 

In the following proposition, we denote $S_t^{\bouncy}(\rho)$
for the process defined in \eqref{eq:limitBPS} to specify the value of the refreshment rate. 

\begin{prop}\label{prop:general_phi}
Under the above assumption, the stochastic process $S^{\bouncy,d}=(S_t^{\bouncy,d})_{t\ge 0}$ defined by $S_t^{\bouncy,d}:=\langle\nabla\Psi^d(\xi_t^{\bouncy,d}),v_t^{\bouncy,d}\rangle$ converges to another stochastic process $(H^{1/2}S_{H^{1/2}t}(H^{-1/2}\rho))_{t\ge 0}$. In particular, 
$$
4\rho \mathbb{E}\left[(\Psi^d(\xi_{\sigma_n}^{\bouncy,d})-\Psi^d(\xi_{\sigma_{n-1}}^{\bouncy,d}))^2\right]~\longrightarrow_{d\rightarrow\infty} ~ H^{1/2}\sigma^2(H^{-1/2}~\rho). 
$$
\end{prop}

Proof of this proposition is in Section \ref{sec:non-gauss}. 
Since the limit of $S_H^{\bouncy,d}$ is the time-scale change of $S^{\bouncy}$, we can still use the $78.12\%$ rule (\ref{eq:optimal_rho}) for the choice of $\rho$. More precisely, the expected number of all jumps and that of the refreshment jumps up to time $T$ are 
\[
TH^{1/2}\left(\frac{1}{\sqrt{2\pi}}+H^{-1/2}\rho\right),\ 
T\rho
\]
with respectively. 
Therefore the fraction of the number of refreshment jumps is
\[
\frac{H^{-1/2}\rho}{\frac{1}{\sqrt{2\pi}}+H^{-1/2}\rho}. 
\]
On the other hand, $H^{1/2}\sigma^2(H^{-1/2}\rho)$ is maximised when $H^{-1/2}\rho=\rho^*$. Therefore we will have the same ratio $78.12\%$ as before when $H^{1/2}\sigma^2(H^{-1/2} \rho)$ is maximised. 

\subsection{Some properties of the limiting processes}\label{subsec:ergodiclimit}

In this section, we study the ergodic properties of the limiting processes $S^{\bouncy}$ and $\mathcal T$. All the proofs are postponed to Appendix \ref{asec:ergodicity}. First we show the existence and uniqueness of the solution of (\ref{eq:limitzigzag}) and (\ref{eq:limitBPS}) starting from 
$\mathcal{T}_0=x\in\mathbb{R}$ and $S_0^{\bouncy}=x\in\mathbb{R}$.  

By Proposition II.1.14 of \cite{JS} (see also III.1.24), there are stopping times $0<\tau_1<\tau_2<\cdots$ with $\mathcal{F}_{\tau_n}$-measurable random variables $Z_n\ (n\ge 1)$ such that
$$
N(\dif t,\dif z)=\sum_{n\ge 1}1_{\{\tau_n<\infty\}}\delta_{(\tau_n,Z_n)}(\dif t,\dif z). 
$$
As in Theorem IV.9.1 of \cite{MR1011252}, we can construct a unique solution in the time interval $[0,\tau_1]$ for (\ref{eq:limitzigzag}) by
\begin{align*}
	\mathcal{T}_t=
	\left\{
	\begin{array}{ll}
		x+t&0\le t <\tau_1\\
		\mathcal{T}_{\tau_1-}+1_{\{Z_n\le \mathcal{T}_{\tau_1-}\}}(-2\mathcal{T}_{\tau_1-}) &t=\tau_1. 
	\end{array}
	\right. 
\end{align*}
Similarly, we can construct a unique solution in the time interval $[0,\tau_n]$ for any $n\in\mathbb{N}$, and hence $\mathcal{T}_t$ is determined globally. 

For (\ref{eq:limitBPS}), in the same way, there are stopping times $0<\sigma_1<\sigma_2<\cdots$ 
with $\mathcal{F}_{\sigma_n}$-measurable random variables $W_n\ (n\ge 1)$ such that 
$\mathcal{L}(W_n|\mathcal{F}_{\sigma_n-})=\mathcal N(0,1)$ and 
$$
R(\dif t,\dif x)=\sum_{n\ge 1}1_{\{\sigma_n<\infty\}}\delta_{(\sigma_n,W_n)}(\dif t,\dif x). 
$$
Then we can construct the unique solution in time interval $[0,\sigma_1]$ by 
\begin{align*}
	S_t^{\bouncy}=
	\left\{
	\begin{array}{ll}
		\mathcal{T}_t&0\le t<\sigma_1\\
		W_1 &t=\sigma_1. 
	\end{array}
	\right. 
\end{align*}
Then, the process $(S_t^{\bouncy})_{t\in [\sigma_1,\sigma_2]}$ proceeds according to (\ref{eq:limitzigzag}) up to time $\sigma_2-$ starting from $S_{\sigma_1}^{\bouncy}=W_1$ in the same way as above. By iterating this procedure, we can construct a unique solution in time interval $[0,\sigma_n]$ for any $n\in\mathbb{N}$, and hence $S_t^{\bouncy}$ is determined globally. 

Let $\psi$ be a $\sigma$-finite measure on a measurable space $(E,\mathcal{E})$. Then 
a continuous time Markov process $X_t$ is said to be ($\psi$-)irreducible if 
$$
\psi(A)>0~\Longrightarrow~\mathbb{E}_x[\eta_A]>0\ (\forall x\in E)
$$
where $\eta_A$ is the occupation time defined by 
$$
\eta_A=\int_0^\infty 1_{\{X_t\in A\}}\dif t. 
$$
A simple sufficient condition for $\psi$-irreducibility is 
$$
\psi(A)>0~\Longrightarrow~P_t(x,A)=:\mathbb{P}_x(X_t\in A)>0\ \quad (\forall x\in E, t\ge T)
$$
for some $T>0$ which is also a sufficient condition for aperiodicity of the Markov process. 
A measurable set $C\in\mathcal{E}$ is said to be small if there exists $t>0$, $\epsilon>0$
and a probability measure $\nu$ such that 
$$
P_t(x,A)\ge \epsilon~\nu(A)\ \quad (\forall x\in C,\ \forall A\in\mathcal{E}). 
$$
This Markov process is said to be $V$-uniformly ergodic if 
there exists a probability measure $\Pi$, a constant $\gamma\in (0,1), C>0$ and $V:E\rightarrow [1,\infty)$ such that 
$$
\|P_t(x,\cdot)-\Pi\|_V\le CV(x)\gamma^t, 
$$ 
where 
$$
\|\nu\|_V:=\sup_{\substack{f: E \rightarrow \mathbb{R}\\  |f(x)|\le V(x)}}\left|\int_Ef(x)\nu(\dif x)\right|. 
$$ 
A simple Foster-Lyapunov type drift condition was established by \cite{MR1379163}. By using their results the following can be proved. 

\begin{thm}\label{theo:ergodicity_tildeS}
	The Markov process $S^{\bouncy}$ is irreducible, aperiodic and any compact set is a small set. Moreover, it is $V$-uniformly ergodic for $V(x)=1+x^2$. 
\end{thm}


\begin{thm}\label{theo:ergodicity_t}
The Markov process $\mathcal{T}$ is irreducible, aperiodic and any compact set is a small set. Moreover, it is $V$-uniformly ergodic for some $e^{|x|}\le V(x)\le 2e^{|x|}$. 
\end{thm}

%
Since the process $\mathcal{T}$ only changes the sign of the process in each jump time, 
by It\^{o}'s formula, 
it satisfies 
\begin{equation}\label{eq:square_t}
	\mathcal{T}_t^2-\mathcal{T}_0^2=2\int_0^t\mathcal{T}_s\dif s,\quad \text{and} \quad  
	|\mathcal{T}_t|-|\mathcal{T}_0|=\int_0^t\mathrm{sgn}(\mathcal{T}_s)\dif s,\ 
\end{equation}
where $\mathrm{sgn}(x)$ is the sign of $x\in\mathbb{R}$. 
Moreover, the following result summarizes some properties of the covariance kernel of $\mathcal T$. See also Figure \ref{fig:functions}. 

\begin{prop}\label{prop:covariance}
The covariance function $K(s,t)=\mathbb{E}[\mathcal{T}_s\mathcal{T}_t]$
of $\mathcal{T}$ satisfies 
\begin{equation}
\label{eq:integral_k}
	\int_0^\infty K(s,0)\dif s=0
\end{equation}
and 
\begin{equation}\label{eq:derivative_k}
	\partial_t K(t,0)|_{t=0}=- 4 \phi(0) = -2\sqrt{\frac{2}{\pi}},\ 
	\partial_t^2 K(t,0)|_{t=0}=1.  
\end{equation}
\end{prop}

\section{Experimental results}
\label{sec:experiments}

In order to investigate the dependence of our results on the distributional assumptions we will carry out computer experiments with respect to four different $d$-dimensional target distributions:
\begin{itemize}
	\item[(i)] The standard normal distribution
	\item[(ii)] A correlated Gaussian distribution, for which $\Var(\xi_i) = 1$ and $\operatorname{Cov} (\xi_i, \xi_j) = \rho$ (for $i \neq j$) where we take $\rho = 0.9$.
	\item[(iii)] $(\xi_1, \dots, \xi_d)$ are i.i.d. Student distributed with $\nu = 4$ degrees of freedom.
	\item[(iv)] $(\xi_1, \dots, \xi_d)$ is a $d$-dimensional spherically symmetric Student distribution with $\nu = 4$ degrees of freedom (see \cite{boisbunon2012class}).
\end{itemize}

For these four distributions we run both the Zig-Zag sampler and the Bouncy Particle Sampler with a refresh rate of 1.4.
In all cases the Zig-Zag process with speeds $v^{\mathrm Z} \in \{-1,+1\}^d$ is run on a fixed continuous time interval $[0,T]$ where $T = 100$. The Bouncy Particle Sampler with speeds $v^{\mathrm B} \in \mathfrak S^{d-1} = \{ v \in \RR^d : \| v \| = 1\}$ is run on a continuous time interval $[0,d \times T]$, which for the purpose of this section is equivalent to a BPS at speed increased by a factor $d$ run on a time interval $[0,T]$. These combinations of velocities and interval length are such that the processes with respect to the observables `first coordinate' and `log density' converge in distribution to their limiting processes as specified in this paper, at least for the standard normal distribution. All processes are started from a random sample from their respective stationary distributions.

In the experiments, for a given trajectory $(\xi(t))_{t \geq 0}$, we define the standardized error with respect to an observable $h$ as
\[ E_h = \frac{\frac 1 T \int_0^T h(\xi(s)) \, \dif s - \pi^d(h)}{\sqrt{\Var_{\pi^d}(h)}},\]
where $\pi^d$ represents the probability distribution with unnormalized negative log density $\Psi^d$.
In all experiments we know the exact value of $\pi^d(h)$ and $\Var_{\pi^d}(h)$ which are specified in Section~\ref{sec:exactvalues}. The continuous time integral representing the ergodic average (given the piecewise deterministic trajectory $(\xi(t))_{0 \leq t \leq T}$) is evaluated analytically, as discussed in Section~\ref{sec:ergodicaverage}. In the box plots below the standardized squared error is displayed for increasing dimension, based on 1000 experiments.

As to be expected from the theory developed in this paper the distribution of the standardized squared error for the standard normal distribution  (Figure~\ref{fig:standardnormal}) is stable with respect to increase in dimension. BPS seems to be more robust in the presence of correlations (Figure~\ref{fig:correlatedgaussian}), in particular with respect to the first coordinate.
In the case of a factorized heavy tailed distribution (Figure~\ref{fig:iidstudent}) we see that the behaviour of both Zig-Zag and BPS is very robust. Finally in the case of a spherically symmetric example (Figure~\ref{fig:symmstudent})  we see similar behaviour for the different samplers with a non-constant dependence on dimension.

\begin{figure}[ht!]
	{\centering  
		\begin{subfigure}{0.49 \textwidth}
			\includegraphics[width=\textwidth]{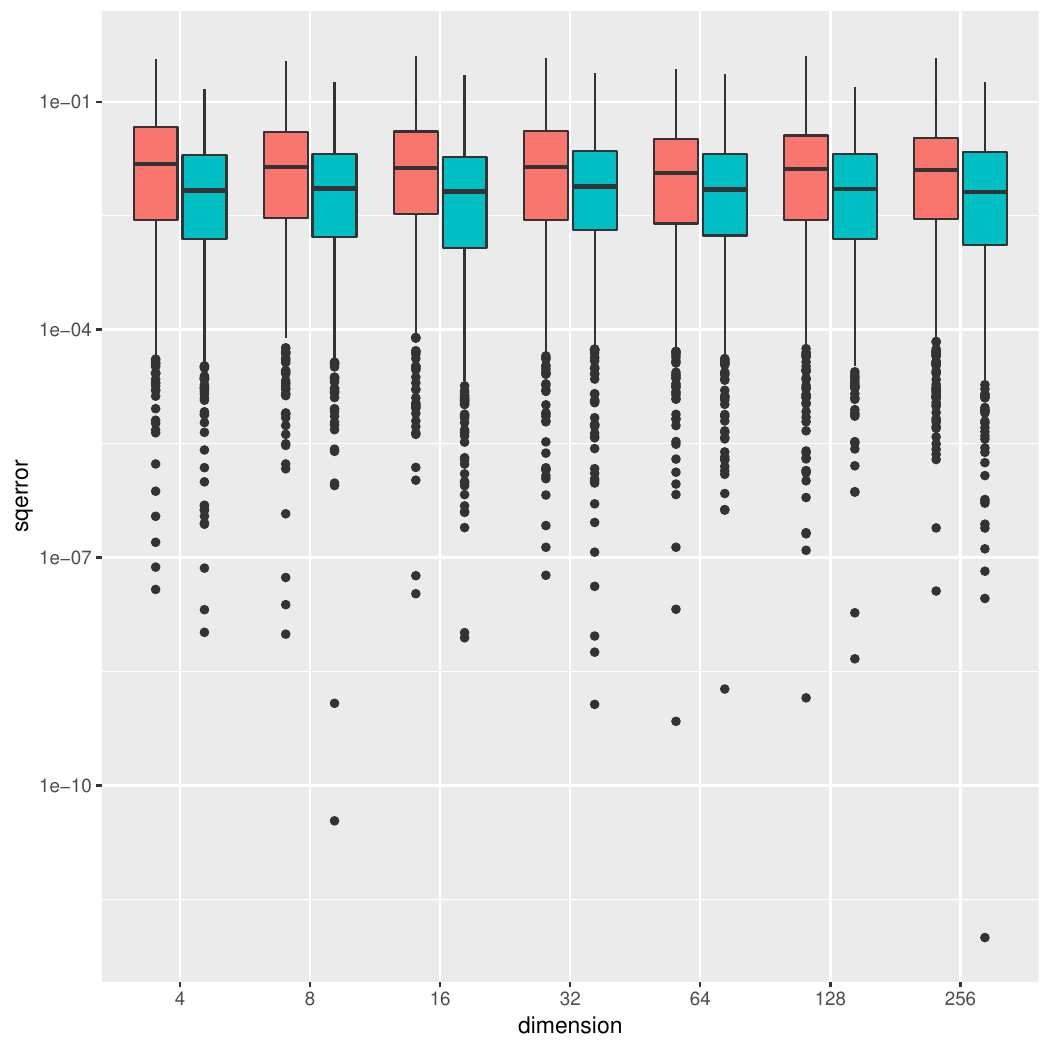}
			\caption{First coordinate}
		\end{subfigure}
		\begin{subfigure}{0.49 \textwidth}
			\includegraphics[width=\textwidth]{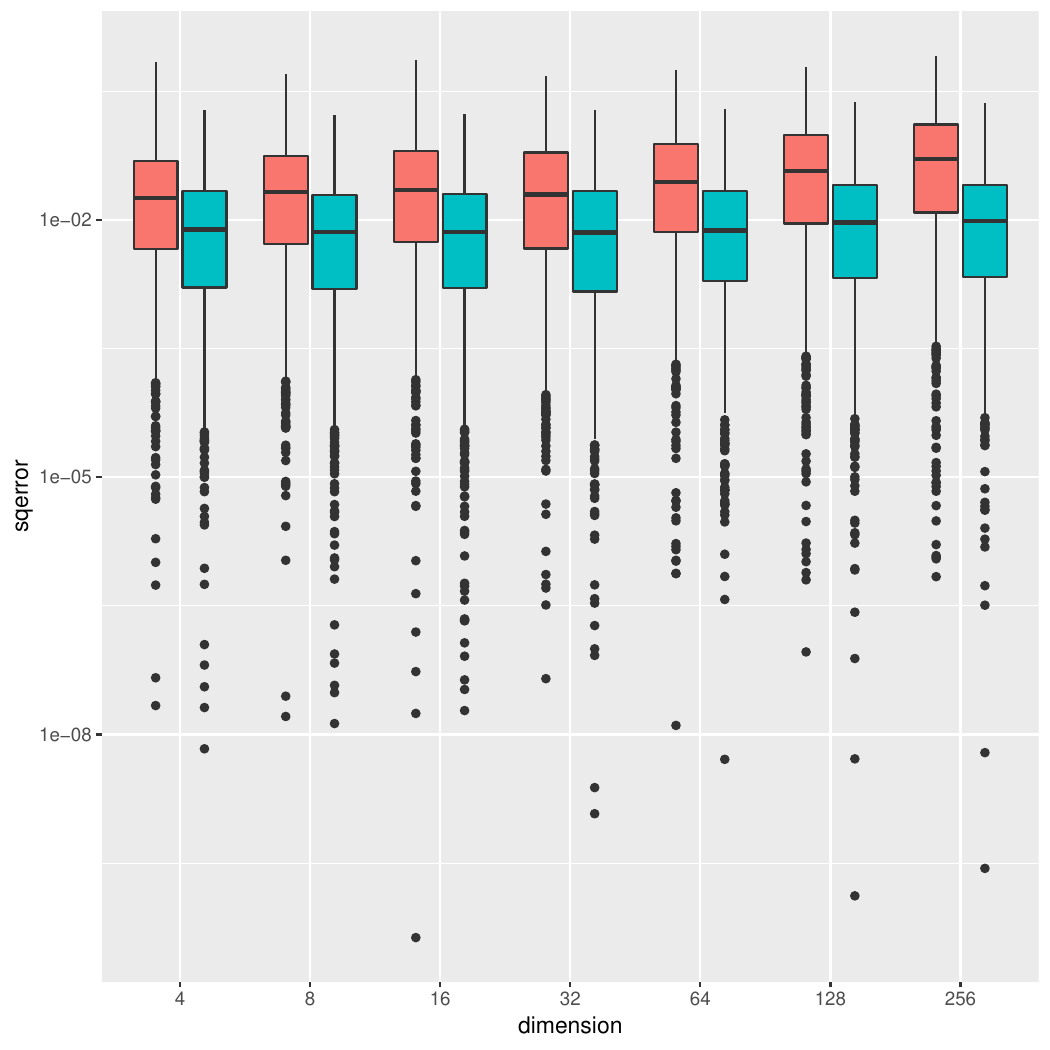}
			\caption{Log density}
		\end{subfigure}
		\caption{Standardized squared errors standard normal distribution. ZZ is cyan, BPS is red.}
		\label{fig:standardnormal}
	}
\end{figure}

\begin{figure}[ht!]
	{\centering  
		\begin{subfigure}{0.49 \textwidth}
			\includegraphics[width=\textwidth]{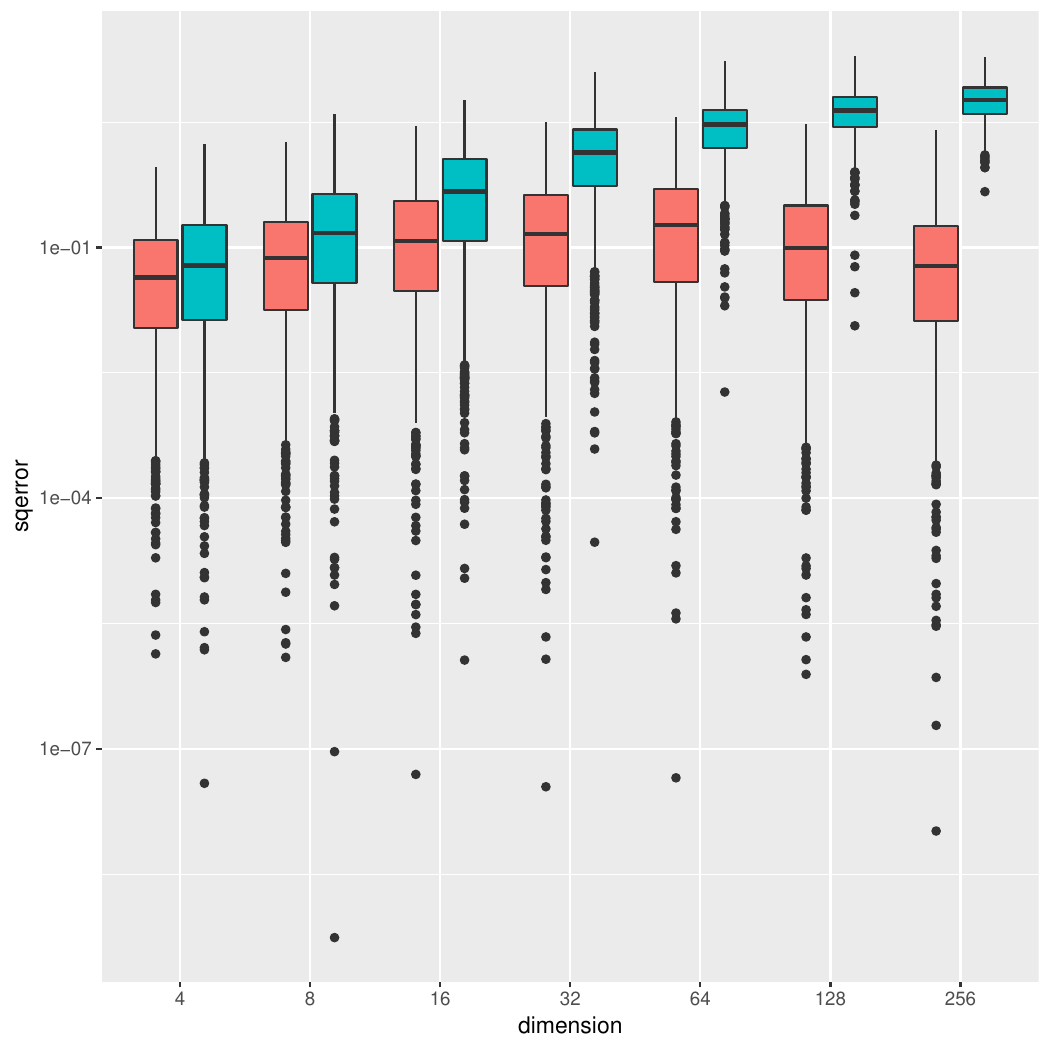}
			\caption{First coordinate}
		\end{subfigure}
		\begin{subfigure}{0.49 \textwidth}
			\includegraphics[width=\textwidth]{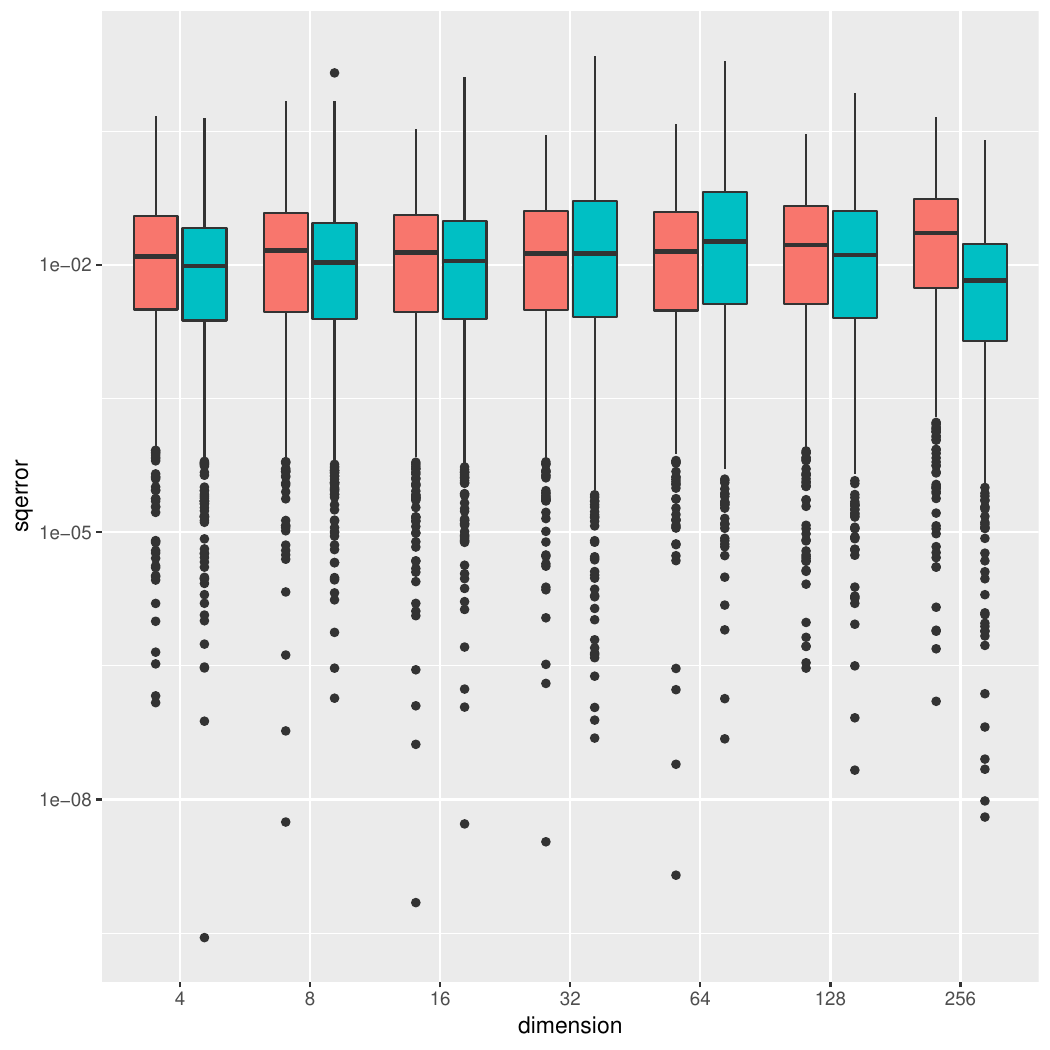}
			\caption{Log density}
		\end{subfigure}
		\caption{Standardized squared errors correlated Gaussian distribution, $\Var(\xi_i) = 1$, $\Cov(\xi_i,\xi_j) = \rho = 0.9$, $i \neq j$. ZZ is cyan, BPS is red.}
		\label{fig:correlatedgaussian}
	}
\end{figure}

\begin{figure}[ht!]
	{\centering  
		\begin{subfigure}{0.49 \textwidth}
			\includegraphics[width=\textwidth]{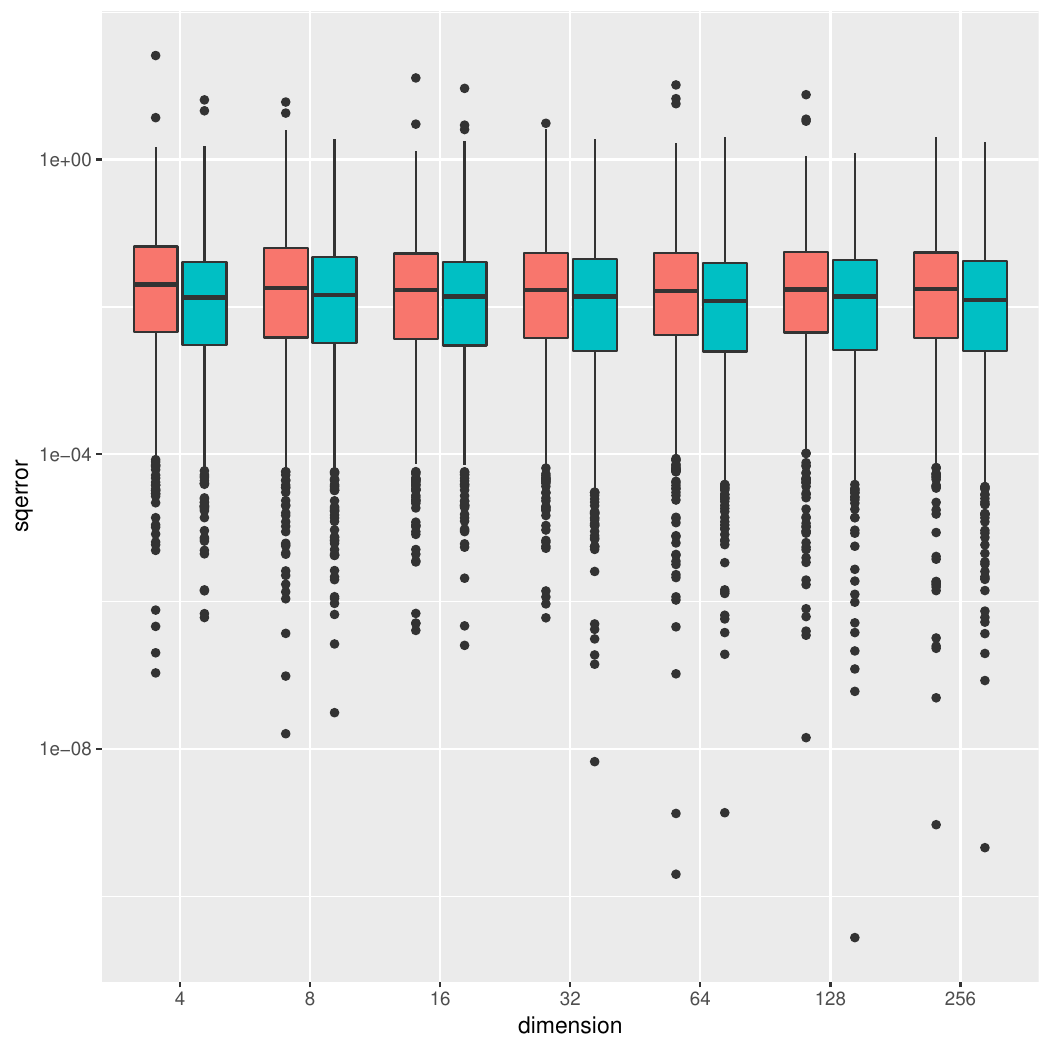}
			\caption{First coordinate}
		\end{subfigure}
		\begin{subfigure}{0.49 \textwidth}
			\includegraphics[width=\textwidth]{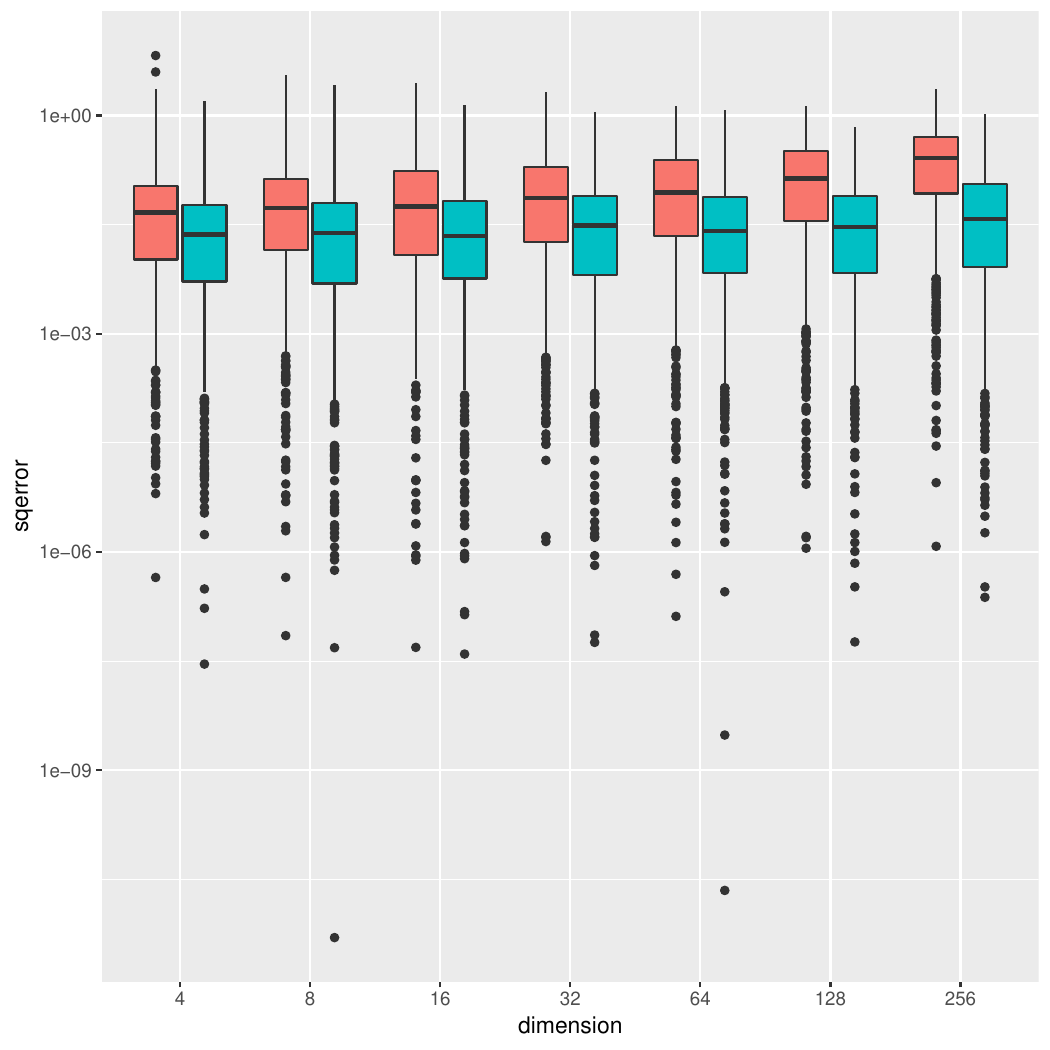}
			\caption{Log density}
		\end{subfigure}
		\caption{Standardized squared errors IID Student distribution with $\nu= 4$ degrees of freedom. ZZ is cyan, BPS is red.}
		\label{fig:iidstudent}
	}
\end{figure}

\begin{figure}[ht!]
	{\centering  
		\begin{subfigure}{0.49 \textwidth}
			\includegraphics[width=\textwidth]{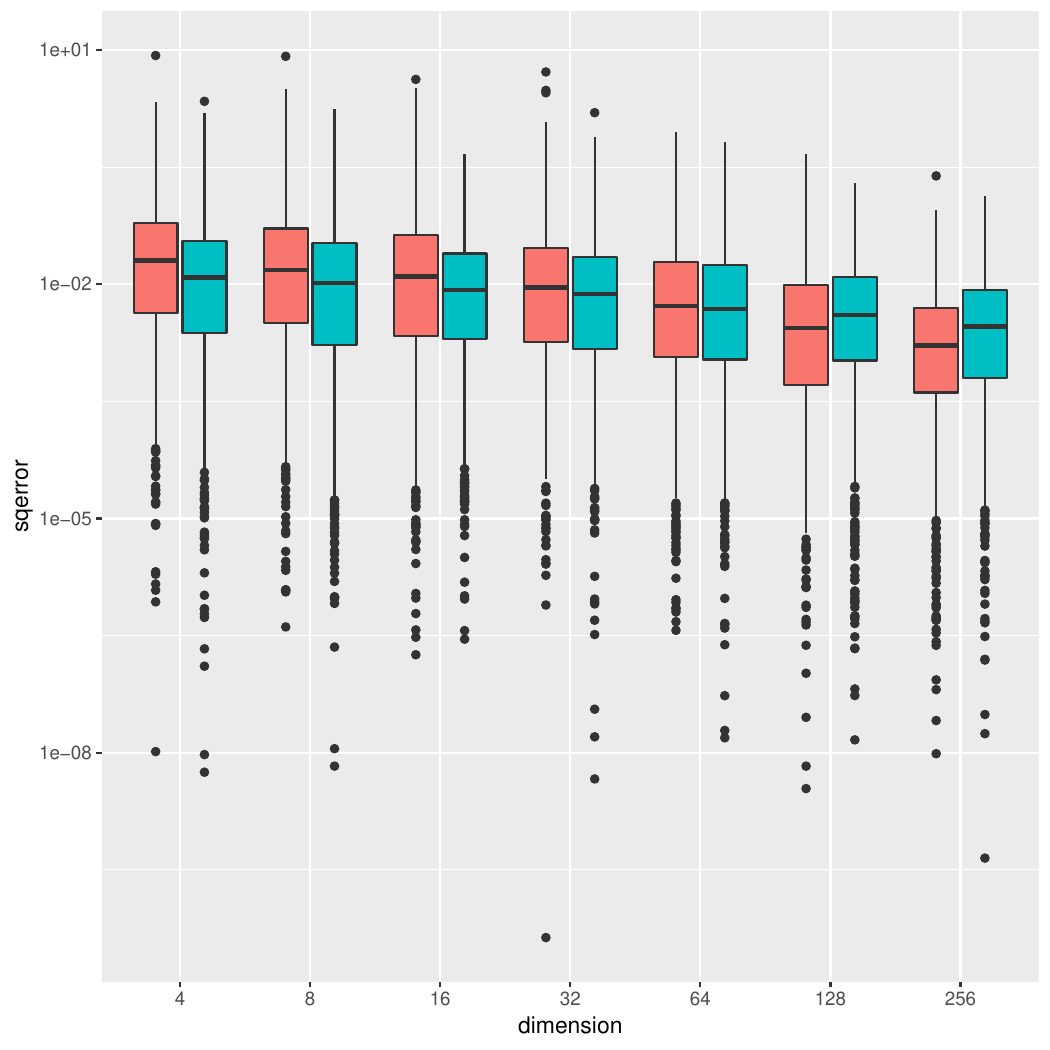}
			\caption{First coordinate}
		\end{subfigure}
		\begin{subfigure}{0.49 \textwidth}
			\includegraphics[width=\textwidth]{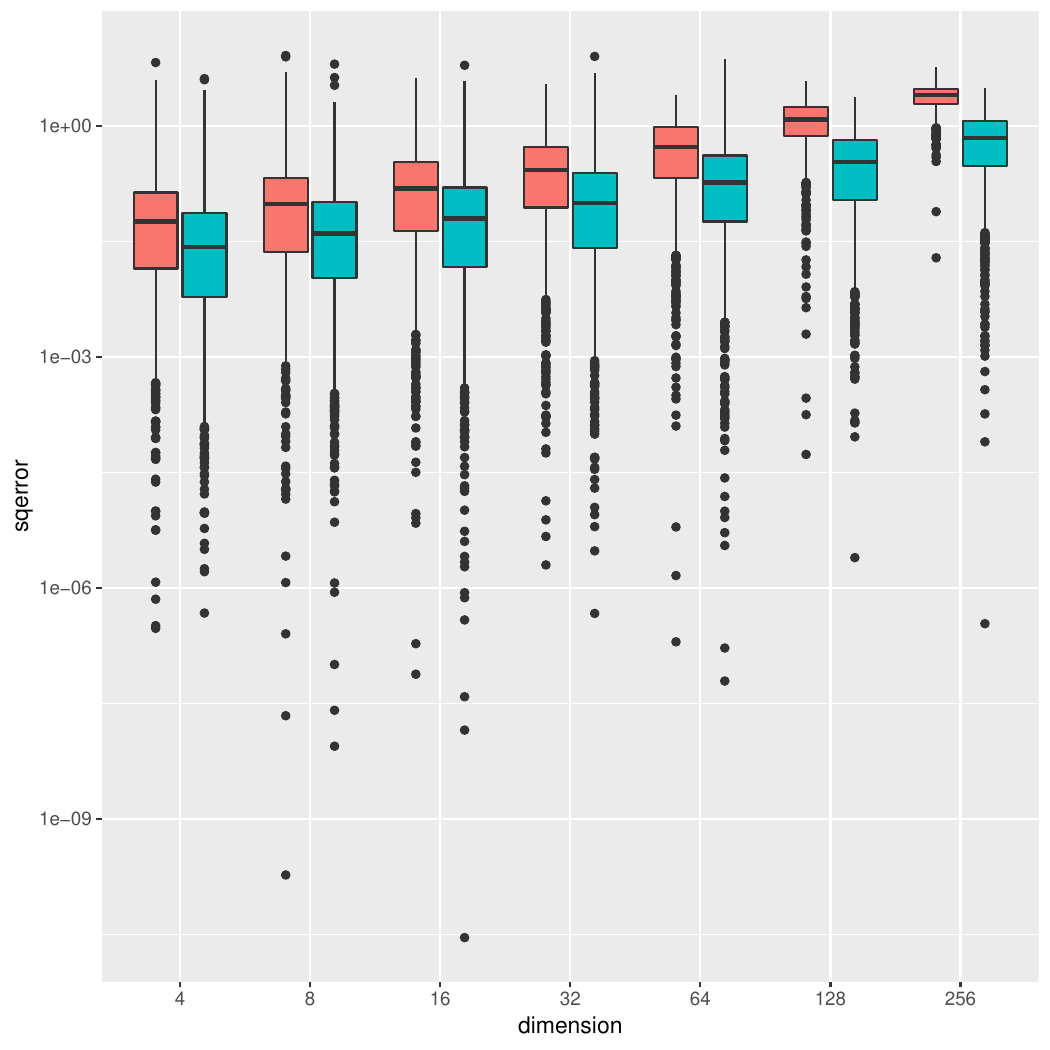}
			\caption{Log density}
		\end{subfigure}
		\caption{Standardized squared errors spherically symmetric Student distribution with $\nu= 4$ degrees of freedom. ZZ is cyan, BPS is red.}
		\label{fig:symmstudent}
	}
\end{figure}

\section{Discussion}\label{sec:discussion}
%
In this paper we considered the high-dimensional asymptotic analysis of ZZ and BPS. The target probability distribution is assumed to be the standard normal distribution. This assumption is indeed restrictive, but the results can be extended to more general target distributions. For the ZZ sampler, it is straight forward to generalise it to a target distribution with a product form 
$\prod_{i=1}^d F(\dif \xi_i)$ where $F$ is a probability measure on $\mathbb{R}$. For BPS sampler, we proved convergence of the angular momentum for a general target distribution in Section \ref{sec:optimal}. 
The generalisation for other results for BPS sampler is a future research goal. 

Recently, the convergence rates of BPS and ZZ have also been studied by \cite{2018arXiv180808592A} and \cite{2018arXiv180804299D}. 
In the former article, they studied $L^2$-exponential convergence rates of Markov semigroups corresponding to the PDMPs under fairly general assumptions. 
The BPS convergence rate $O(d)$ considered here is in agreement with their results after noticing that they assumed $\sqrt{d}~\mathfrak{S}_{d-1}$ as the direction space. On the other hand, the convergence rate obtained by \cite{2018arXiv180804299D} is $O(d^{1/2})$, which is different from ours since this work studied a different scaling limit regime. In Theorem \ref{theo:bps_coordinate}, we obtained the Ornstein-Uhlenbeck process limit for the $1$st coordinate process. This scaling limit regime does not have optimal choice of $\rho$ since it can be accelerated by taking $\rho\downarrow 0$. They studied another scaling limit regime by taking $\rho=O(d^{-1/2})$, and proved that the $1$st coordinate process, together with the velocity, converges to a randomised Hamiltonian Monte Carlo process. In their regime, the negative log density process will be degenerate by Theorem \ref{theo:bps_ll}, and so we did not use diminishing refreshment rate.

\appendix

\section{The convergence of the Zig-Zag sampler}\label{asec:zz}

\subsection{Proof of Theorem \ref{theo:zigzag_limit}}\label{asec:zz_angle}

Let $S^\zigzag=(S^{\zigzag}_t)_{t\ge 0}$ be a Gaussian process with mean $0$ and covariance
$
K(s,t)=\mathbb{E}[\mathcal{T}_s\mathcal{T}_t]
$
where $\mathcal{T}$ is defined in (\ref{eq:limitzigzag}). 
Firstly, we prove that the Gaussian process $S^{\zigzag}$ is not a Markov process, although $\mathcal{T}$ is a Markov process. 

\begin{lem}\label{lem:non_markov}
The stationary Gaussian process $S^{\zigzag}$ is not a Markov process.	
\end{lem}

\begin{proof}
By Theorem V.8.1 of \cite{Doob} together with the continuity of $t\mapsto K(t,0)$, if $S^{\zigzag}$ is a Markov process, then 
	$$
	K(t,0)=e^{-ct}
	$$
	for some $c\in\mathbb{R}$. Therefore, the first and the second derivatives of $K(t,0)$ at $t=0$ are $-c$ and $c^2$ with respectively. However, this is impossible by derivatives calculated in Proposition \ref{prop:covariance}. Thus the process $S^{\zigzag}$ is not a Markov process. 
\end{proof}

Next we prove convergence of $S^{\zigzag,d}$. 
We denote the space of continuous and c\`adl\`ag functions on $[0,\infty)$ by $\mathbb C[0,\infty)$ and $\mathbb D[0,\infty)$, respectively. 
A sequence of $\mathbb{D}[0,\infty)$-valued processes $X^d=(X_t^d)_{t\ge 0}$ is called $\mathbb{C}$-tight if it is tight and any limit point is in $\mathbb{C}[0,\infty)$ with probability $1$. By Corollary VI.3.33 of \cite{JS}, if $X^d$ and $Y^d$ are $\mathbb{C}$-tight, then $(X_t^d+Y_t^d)_{t\ge 0}$ is $\mathbb{C}$-tight. On the other hand, the sum of tight sequence of processes is not tight in general. 

\begin{lem}\label{lem:sz_s}
The process $S^{\zigzag,d}$
converges to $S^{\zigzag}$. 
\end{lem}

\begin{proof}
Observe that $S_t^{\zigzag,d}=d^{-1/2}\sum_{i=1}^d \mathcal{T}_{i,t}^d$ where 
\begin{equation}\label{eq:iid_decomposition}
	\mathcal{T}_{i,t}^d=\xi_{i,t}^{\zigzag,d}~v_{i,t}^{\zigzag,d}. 
\end{equation}
By construction, 
$(\mathcal{T}_{i,t}^d)_{t\ge 0}\ (i=1,\ldots, d)$ are independent processes and have the same law as that of $\mathcal{T}$. 
By using the fact, we prove tightness of the sequence of processes $(S_t^{\zigzag,d})_{t\in (0,T]}$ for each $T>0$. 
By (\ref{eq:square_t}), we have
$$
\sup_{0\le t,u\le T}|\mathcal{T}_u-\mathcal{T}_t|\le 2\sup_{0\le t\le T}|\mathcal{T}_t|\le 2|\mathcal{T}_0|+2T. 
$$
Observe that any moments of the right-hand side of the above inequality exist since $\mathcal{T}_0\sim \mathcal{N}(0,1)$. 
By using this bound, for the Poisson random measure $N(\dif t,\dif z)$, we have 
\begin{align*}
	\sum_{t<v\le u}1_{\{\Delta \mathcal{T}_v\neq 0\}}
	=\int_{(t,u]\times\mathbb{R}_+}1_{\{z\le\mathcal{T}_{s-}\}}N(\dif s,\dif z)\le N(A)
\end{align*}
where $A=(t,u]\times (0, |\mathcal{T}_0|+T]$ since $|\mathcal{T}_{t}|\le |\mathcal{T}_0|+t$ by \eqref{eq:square_t}. 
Let $\lambda=|u-t|(|\mathcal{T}_0|+T)$. Then we have
$$
\mathbb{P}(N(A)\ge 1)=1-e^{-\lambda}\le \lambda,\ 
\mathbb{P}(N(A)\ge 2)=1-e^{-\lambda}-\lambda e^{-\lambda}\le \frac{\lambda^2}{2}. 
$$
Since if there is no jump, $\mathcal{T}_t$ has the deterministic move, and we have
$$
\sum_{t< v\le u}1_{\{\Delta \mathcal{T}_v\neq 0\}}=0~\Longrightarrow~\mathcal{T}_u-\mathcal{T}_s=u-s. 
$$
Hence if $t\le u\le T$, we have
\begin{align*}
\mathbb{E}\left[(\mathcal{T}_u-\mathcal{T}_t)^2\right]&=
\mathbb{E}\left[(\mathcal{T}_u-\mathcal{T}_t)^2,\sum_{t< v\le u}1_{\{\Delta \mathcal{T}_v\neq 0\}}=0\right]\\
&\quad+
\mathbb{E}\left[(\mathcal{T}_u-\mathcal{T}_t)^2,\sum_{t< v\le u}1_{\{\Delta \mathcal{T}_v\neq 0\}}\ge 1\right]\\
&\le |u-t|^2
+\mathbb{E}[(2|\mathcal{T}_0|+2T)^2\times\lambda]\\
&\le C|u-t|
\end{align*}
for some $C>0$. On the other hand, if $s\le t\le u\le T$
\begin{align*}
\sum_{s< v\le u}1_{\{\Delta \mathcal{T}_v\neq 0\}}=1~
\Longrightarrow~
\sum_{s< v\le t}1_{\{\Delta \mathcal{T}_v\neq 0\}}=0\ \mathrm{or}\ 
\sum_{t< v\le u}1_{\{\Delta \mathcal{T}_v\neq 0\}}=0
\end{align*}
and hence 
\begin{align*}
\sum_{s< v\le u}1_{\{\Delta \mathcal{T}_v\neq 0\}}=1~
\Longrightarrow~
(\mathcal{T}_u-\mathcal{T}_t)^2(\mathcal{T}_t-\mathcal{T}_s)^2\le |u-s|^2(2|\mathcal{T}_0|+2T)^2. 
\end{align*}
Therefore, 
\begin{align*}
\mathbb{E}\left[(\mathcal{T}_u-\mathcal{T}_t)^2(\mathcal{T}_t-\mathcal{T}_s)^2\right]
&=
\mathbb{E}\left[(\mathcal{T}_u-\mathcal{T}_t)^2(\mathcal{T}_t-\mathcal{T}_s)^2,\sum_{t< v\le u}1_{\{\Delta \mathcal{T}_v\neq 0\}}= 0\right]\\
&\quad +
\mathbb{E}\left[(\mathcal{T}_u-\mathcal{T}_t)^2(\mathcal{T}_t-\mathcal{T}_s)^2,\sum_{t< v\le u}1_{\{\Delta \mathcal{T}_v\neq 0\}}= 1\right]\\
&\quad +
\mathbb{E}\left[(\mathcal{T}_u-\mathcal{T}_t)^2(\mathcal{T}_t-\mathcal{T}_s)^2,\sum_{t< v\le u}1_{\{\Delta \mathcal{T}_v\neq 0\}}\ge 2\right]\\
& \le 
|u-s|^4+|u-s|^2\mathbb{E}[(2|\mathcal{T}_0|+2T)^2]\\
&\quad +
\mathbb{E}[(2|\mathcal{T}_0|+2T)^4\times\frac{\lambda^2}{2}]\\
& \le C|u-s|^2
\end{align*}
for some $C>0$. 
These inequalities imply the conditions (i, ii) in Theorem 2 of \cite{Hahn1978}. Therefore, by Theorem 2 of \cite{Hahn1978}, we have central limit theorems for the sum of the copies of $(\mathcal{T}_t)_{t\in (0,T]}$. In particular, $(S_t^{\zigzag,d})_{t\in (0,T]}$ is tight. 

On the other hand, for any $0\le t_1<t_2<\cdots <t_k$, 
any $k$-dimensional random variable 
$$
(S_{t_1}^{\zigzag,d},\ldots, S_{t_k}^{\zigzag,d})=d^{-1/2}\sum_{i=1}^d(\mathcal{T}_{i,t_1}^d,\ldots, \mathcal{T}_{i,t_k}^d)
$$
converges to a normal distribution by the central limit theorem since the random variables 
$(\mathcal{T}_{i,t_1}^d,\ldots, \mathcal{T}_{i,t_k}^d)$ $(i=1,\ldots, d)$ are independent and have the same law as that of 
$(\mathcal{T}_{t_1},\ldots, \mathcal{T}_{t_k})$. 
Hence $(S_t^{\zigzag,d})_{t\in (0,T]}$ converges to $(S_t^{\zigzag})_{t\in (0,T]}$ 
by Lemma VI.3.19 of \cite{JS}. Then the convergence of $(S^{\zigzag,d}_t)_{t\ge 0}$ to 
$(S^{\zigzag}_t)_{t\ge 0}$
also follows from 
Theorem 16.7 of \cite{MR1700749}. 
\end{proof}

We call a $\mathbb{D}[0,\infty)$-valued processes $X=(X_t)_{t\ge 0}$ \emph{locally $\alpha$-H\"older continuous} if there is a $\mathbb{C}[0,\infty)$-valued process $\tilde{X}$ with the same law as that of $X$ such that there exists $\delta>0, h(\omega)>0$ and 
\begin{align*}
	\mathbb{P}\left(\omega\in\Omega:\sup_{|u-v|\le h(\omega), 0\le u,v\le T}\frac{|\tilde{X}_u(\omega)-\tilde{X}_v(\omega)|}{|u-v|^\alpha}\le\delta \right)=1
\end{align*}
for any $T>0$. 

\begin{lem}\label{lem:holder}
	$S^{\zigzag}$ is locally $\alpha$-H\"older continuous for $\alpha\in (0,1/2)$ but it is not locally $\alpha$-H\"older continuous for any $\alpha\ge 1/2$. 
\end{lem}

\begin{proof}
The mean zero Gaussian process $S^{\zigzag}$ satisfies 
$S_t^{\zigzag}\sim \mathcal N(0,1)$
and 
$S_t^{\zigzag}-S_0^{\zigzag}\sim\mathcal N(0,\sigma(t)^2)$ where
$$
\sigma(t)^2
:=\mathbb{E}[(S_t^{\zigzag}-S_0^{\zigzag})^2]=
\mathbb{E}[(S_t^{\zigzag})^2]+\mathbb{E}[(S_0^{\zigzag})^2]-2
\mathbb{E}[S_t^{\zigzag}S_0^{\zigzag}]=
2-2K(t,0). 
$$
Observe that $\sigma^2(0)=0$. By Proposition \ref{prop:covariance} we have 
\begin{align*}
	\frac{\sigma(t)^2}{t}&=\frac{\sigma(t)^2-\sigma^2(0)}{t}\\
	&=-2\frac{K(t,0)-K(0,0)}{t}\longrightarrow_{t\rightarrow 0}-2\partial_tK(t,0)|_{t=0}=8\phi(0), 
\end{align*}
and in particular, for sufficiently small $h>0$, 
we have a local bound
$ct\le |\sigma(t)^2|\le Ct\ (0\le t\le h)$ for some $c, C>0$. 
On the other hand, since we have $\sigma(t)^2\le 2 \E[(S_t^\zigzag)^2 + (S_0^\zigzag)^2] = 4$, 
 there is a global bound
$|\sigma(t)^2|\le C t\ (t\ge 0)$ for some  constant $C>0$.  
Therefore, the $(2n)$-th moment of  $S_t^{\zigzag}-S_0^{\zigzag}\sim\mathcal N(0,\sigma(t)^2)$ is 
$$
\mathbb{E}\left[|S_t^{\zigzag}-S_0^{\zigzag}|^{2n}\right]=(2n-1)!!~|\sigma(t)^2|^n\le C|t|^n
$$
for some $C>0$ for any $n\in\mathbb{N}$. 
Thus, local $\alpha$-H\"older continuity for any $\alpha\in (0,1/2)$ follows from Kolmogorov-\v{C}entsov's theorem (Theorem 2.2.8 of \cite{Karatzas1991}). 

On the other hand, by Proposition \ref{prop:covariance}, the second derivative of $K(t,0)$ around $t=0$ is positive and hence $\sigma(t)^2$ is concave around $t=0$. Therefore, by Slepian's lemma (Theorem 7.2.10 of \cite{MR2250510}), we have
\begin{equation}\label{eq:slepian}
\lim_{t\rightarrow 0}\sup_{|u-v|\le t, 0\le u,v\le 1}\frac{|S_u^{\zigzag}(\omega)-S_v^{\zigzag}(\omega)|}{\sqrt{2\sigma^2(u-v)\log(1/|u-v|)}}\ge 1
\end{equation}
almost surely. 
If $S^{\zigzag}$ is locally $1/2$-H\"older continuous, then there exists a process 
$\tilde{S}^{\zigzag}$, with the same law as $S^{\zigzag}$, such that 
for $t\ge 0$, and for some $\delta>0$, 
$$
\frac{|\tilde{S}_{t+h}^{\zigzag}(\omega)-\tilde{S}_t^{\zigzag}(\omega)|}{\sqrt{2\sigma^2(h)\log(1/|h|)}}
\le \delta \frac{|h|^{1/2}}{\sqrt{2c|h|\log(1/|h|)}}
$$
for sufficiently small $h$. 
The right hand side converges to $0$ which contradicts \eqref{eq:slepian}. Thus 
$\tilde{S}^{\zigzag}$
and  
$S^{\zigzag}$
 can not be locally $1/2$-H\"older continuous with probability $1$. 
\end{proof}

\begin{proof}[Proof of Theorem \ref{theo:zigzag_limit}]
The claim follows by Lemmas \ref{lem:non_markov}-\ref{lem:holder}. 
\end{proof}

\subsection{Proof of Corollary \ref{cor:zigzag_switches}}

\begin{proof}[Proof of Corollary \ref{cor:zigzag_switches}]
The convergence of the switching rate comes from the law of large numbers. 
Observe that 
$$
d^{-1}\sum_{0\le t\le T}1_{\{\Delta S_t^{\zigzag,d}\neq 0\}}=
d^{-1}\sum_{i=1}^d\sum_{0\le t\le T}1_{\{\Delta \mathcal{T}_{i,t}^d\neq 0\}}
$$
where $(\mathcal{T}_{i,t}^d)_{t\ge 0}\ (i=1,\ldots, d)$ are independent copies of (\ref{eq:limitzigzag}). Therefore, by the law of large numbers, we have
\begin{align*}
	d^{-1}\sum_{0\le t\le T}1_{\{\Delta S_t^d\neq 0\}}
\longrightarrow_{d\rightarrow\infty}\mathbb{E}\left[\sum_{0\le t\le T}1_{\{\Delta \mathcal{T}_t\neq 0\}}\right]&=\mathbb{E}\left[\int_{(0,T]\times\mathbb{R}_+} 1_{\{z\le \mathcal{T}_t\}}\dif z\dif t\right]\\
&=\int_0^T\mathbb{E}[\mathcal{T}_t^+]\dif t=\frac{T}{\sqrt{2\pi}}
\end{align*}
by $\mathcal{T}_t\sim \mathcal N(0,1)$. 
\end{proof}

\subsection{Proof of Theorem \ref{theo:zigzag_integral_limit}}\label{asec:zz_coordinate}

We call a $\mathbb{D}[0,\infty)$-valued processes $X$ differentiable with respect to the time index $t$ if there is a $\mathbb{C}[0,\infty)$-valued process $\tilde{X}$ with the same law as that of $X$ and 
another  $\mathbb{C}[0,\infty)$-valued process $(\partial \tilde{X}_t(\omega))_{t\ge 0}$
on the same probability space as that of $\tilde{X}$ such that
\begin{align*}
	\mathbb{P}\left(\omega\in\Omega:\lim_{h\rightarrow 0}\frac{\tilde{X}_{t+h}(\omega)-\tilde{X}_t(\omega)}{h}=\partial \tilde{X}_t(\omega),\ \forall t\in (0,T) \right)=1
\end{align*}
for any $T>0$. 

\begin{proof}[Proof of  Theorem \ref{theo:zigzag_integral_limit}]
The map 
$(\alpha_t)_{t\ge 0}\mapsto (\int_0^t \alpha_{s}\dif s)_{t\ge 0}$
from $\mathbb{D}[0,\infty)$ to $\mathbb{C}[0,\infty)$ is continuous.
Also, by Theorem~\ref{theo:zigzag_limit}, the sequence $S^{\zigzag,d}$ converges in law to $S^{\zigzag}$. 
Therefore, the sequence of processes $(Y_t^{\zigzag,d}-Y_0^{\zigzag,d})_{t\ge 0}\ (d\in\mathbb{N})$ is $\mathbb{C}$-tight since 
\begin{equation}
\label{eq:representation_zz_coordinate}
\begin{split}
	Y_t^{\zigzag,d}-Y_0^{\zigzag,d}&=d^{-1/2}(\|\xi_{t}^{\zigzag,d}\|^2-\|\xi_{0}^{\zigzag,d}\|^2)\\
	&=2\int_0^{t}S_{u}^{\zigzag,d}\dif u~\longrightarrow_{d\rightarrow\infty}~2\int_0^{t}S_{u}^{\zigzag}\dif u
\end{split}
\end{equation}
in distribution in Skorohod topology.
Also, $\xi_0^{\zigzag,d}\sim \mathcal{N}_d(0, I_d)$ and we have
$$
Y_0^{\zigzag,d}=\sqrt{d}\left(\frac{\|\xi_0^{\zigzag,d}\|^2}{d}-1\right)
\Longrightarrow_{d\rightarrow\infty} \mathcal N(0,2).
$$
Thus $(Y_t^{\zigzag,d})_{t\ge 0}=((Y_t^{\zigzag,d}-Y_0^{\zigzag,d})+Y_0^{\zigzag,d})_{t\ge 0}$ is $\mathbb{C}$-tight. On the other hand, by the finite dimensional central limit theorem, 
$(Y_{t_1}^{\zigzag,d},Y_{t_2}^{\zigzag,d},\ldots, Y_{t_k}^{\zigzag,d})$ converges in distribution to some normal distribution for any $k\in\mathbb{N}$ and any $t_1<\ldots<t_k$, since 
\begin{align*}
	(Y_{t_1}^{\zigzag,d},Y_{t_2}^{\zigzag,d},\ldots, Y_{t_k}^{\zigzag,d})
&=\sqrt{d}\left(\frac{\|\xi_{t_1}^{\zigzag,d}\|^2}{d}-1,\ldots, \frac{\|\xi_{t_k}^{\zigzag,d}\|^2}{d}-1\right)\\
&=\sqrt{d}^{-1}\sum_{i=1}^d\left(\|\xi_{i,t_1}^{\zigzag,d}\|^2-1,\ldots, \|\xi_{i,t_k}^{\zigzag,d}\|^2-1\right)\\
&=:\sqrt{d}^{-1}\sum_{i=1}^d U_i^d
\end{align*}
and $U_i^d\ (i=1,\ldots, d,\ d\in\mathbb{N})$ are mean $0$ and independent and identically distributed since every component of $\xi^{Z,d}$ is an independent Zig-Zag process due to the decoupling of the switching rate.
Thus by Lemma VI.3.19 of \cite{JS}, $Y^{\zigzag,d}$ converges to a Gaussian process, which will be denoted by	 $Y^{\zigzag}$ with a covariance function denoted by $L(s,t)$. 
Since the covariance function of $Y^{\zigzag,d}$ and $Y^\zigzag$ are the same, 
and $\dif Y_t^{\zigzag,d}=2S_t^{\zigzag,d}\dif t$, we have
\begin{align*}
	L(s,t)&=\mathbb{E}[Y_s^{\zigzag,d}Y_t^{\zigzag,d}]\\
	&=\frac{1}{2}\left(\mathbb{E}[(Y_s^{\zigzag,d})^2]+\mathbb{E}[(Y_t^{\zigzag,d})^2]-\mathbb{E}[(Y_s^{\zigzag,d}-Y_t^{\zigzag,d})^2]\right)\\
	&=\frac{1}{2}\left(4-4\mathbb{E}\left[\left\{\int_{s}^tS_u^{\zigzag,d}\dif u\right\}^2\right]\right)\\
	&=2-2\int_{s}^t\int_{s}^t\mathbb{E}[S_u^{\zigzag,d}S_v^{\zigzag,d}]\dif u\dif v. 
\end{align*}
Furthermore, since the covariance function of $S^{\zigzag,d}$ and $\mathcal{T}$ are the same, 
we have 
\begin{align*}
	L(s,t)&=2-2\int_{s}^t\int_{s}^t\mathbb{E}[\mathcal{T}_u\mathcal{T}_v]\dif u\dif v=
	2-2\int_{s}^t\int_{s}^tK(u,v)\dif u\dif v. 
\end{align*}
Finally,
 since  $(Y_t^{\zigzag}-Y_0^{\zigzag})_{t\ge 0}$ and $(2\int_0^tS_u^{\zigzag}\dif u)_{t\ge 0}$ have the same law by (\ref{eq:representation_zz_coordinate}) and the latter process is differentiable, the process 
$Y^{\zigzag}$ has a differentiable version. 
\end{proof}

\subsection{Proof of Theorem \ref{theo:zigzag_first_component}}\label{sec:proof_zigzag_first_component}

\begin{proof}[Proof of Theorem \ref{theo:zigzag_first_component}]
Let $(\xi_t^{\zigzag})_{t\ge 0}$ be the process 
such that  $\xi_0^{\zigzag}\sim \mathcal N(0,1)$ and $v_0^{\zigzag}$ are independent 
and $\mathbb{P}(v_0^{\zigzag}=+1)=\mathbb{P}(v_0^{\zigzag}=-1)=1/2$ and 
$$
\xi_t^{\zigzag}~=~\xi_0^{\zigzag}~+~\int_0^tv_s^{\zigzag}~\dif s\ (t\ge 0), 
$$
and 
$$
v_t^{\zigzag}~=~v_0^{\zigzag}-2\int_{(0,t]\times\mathbb{R}_+} v_{s-}^{\zigzag}~1_{\{z\le \xi_{s-}^{\zigzag})\}}~N(\dif s,\dif z)\ (t\ge 0) 
$$
where $N(\dif t ,\dif x)$ is the homogeneous Poisson measure with the intensity measure $\dif t~\dif x$.
The process $(\xi_t^{\zigzag})_{t\ge 0}$ was studied extensively by \cite{MR3694318}. In particular, it is ergodic by Proposition 2.2 of \cite{MR3694318}. Therefore, for $k\in\mathbb{N}$, if $(\xi_{i,t}^{\zigzag})_{t\ge 0}\ (i=1,\ldots, k)$ are independent copies of $(\xi_t^{\zigzag})_{t\ge 0}$, we have
\begin{equation}\label{eq:zig_zag:lln}
\frac{1}{T}\int_0^T f(\xi_{1,t}^{\zigzag},\ldots, \xi_{k,t}^{\zigzag})\dif t\longrightarrow_{T\rightarrow\infty}
\int_{\mathbb{R}^k} f(x)\phi_k(x)\dif x	
\end{equation}
almost surely, where 
$f:\mathbb{R}^k\rightarrow\mathbb{R}$ is a $N_k(0,I_k)$-integrable function.

On the other hand, the processes $(\xi_{k,t}^{\zigzag,d})_{t\ge 0}\ (k\in\{1,\ldots, d\}, d\in\mathbb{N})$
are independent and identically distributed with the same law as that of $(\xi_t^{\zigzag})_{t\ge 0}$.   
Since 
$$
\frac{1}{T}\int_0^{T} f(\pi_k(\xi_t^{\zigzag,d}))\dif t
=\frac{1}{T}\int_0^{T} f(\xi_{1,t}^{\zigzag,d},\ldots, \xi_{k,t}^{\zigzag,d})\dif t
$$
has the same law as that of the left-hand side of (\ref{eq:zig_zag:lln}), the claim follows. 
\end{proof}

\section{The convergence of the Bouncy Particle Sampler}\label{asec:bps}

\subsection{Some preliminary results}

\subsubsection{Some remarks on semimartingale characteristics and majoration hypothesis}\label{subsubsec:semimartingale}

As commented at the end of Section \ref{subsec:ergodiclimit}, we use Martingale problem approach 
to show scaling limit results instead of the classical Trotter-Kato approach. For this approach, we need some knowledge on semimartingale theory. 
A nice introduction to semimartingale theory can be found in Chapters I and II of \cite{JS}. Our notation will generally follow this reference. 
A semimartingale $X=(X_t)_{t\ge 0}$, is called locally square-integrable if it has the canonical decomposition 
$$
X_t=X_0+M_t+B'_t,  \quad t\ge 0,
$$ 
such that 
$M=(M_t)_{t\ge 0}$ is locally square-integrable local martingale, and $B'=(B'_t)_{t\ge 0}$ is predictable process with finite variation (see Definition II.2.27).
We consider the convergence of a sequence of semimartingales. We prove the convergence by using the so-called characteristics $(B', C,\nu)$ and the modified second characteristic $\tilde{C}'$. We briefly explain these characteristics for locally square-integrable semimartingale. 
Note that as in Section IX.3b.2, for a locally square-integrable semimartingale, we can treat the characteristics without truncation function $h(x)$ in Definition II.2.16. 

The first characteristic $B'$ was already introduced as above. 
We denote $\mu^X$ for the random measure associated to the jumps of $X$, that is, 
$$
\mu^X(\omega;\dif t,\dif x)=\sum_{s>0}1_{\{\Delta X_s(\omega)\neq 0\}}\delta_{(s,\Delta X_s(\omega))}(\dif t,\dif x). 
$$
The third characteristic $\nu(\omega;\dif t,\dif x)$ is the intensity measure of the random measure $\mu^X$, and $\tilde{C}'=(\tilde{C}'_t)_{t\ge 0}$ is the predictable quadratic variation of $M$. 
The second characteristic $C$ is the  predictable quadratic variation of the continuous part of $X$,  but in this section, $C\equiv 0$ since the processes $S^{\bouncy}$ and $S^{\bouncy,d}$ do not have continuous martingale parts. 
 
For example, the Markov process $S^{\bouncy}$ defined in (\ref{eq:limitBPS}) has the following decomposition. By the definition for the stochastic integral with respect to random measures (Section II.1d), the square integrable martingale part is
\begin{align*}
M_t=M_t(S^{\bouncy})
~& =- 2\int_{(0,t]\times\mathbb{R}_+} S^{\bouncy}_{s-}~1_{\{z\le S^{\bouncy}_{s-}\}}~\left\{N(\dif s,\dif z)-\dif s\, \dif z\right\}\\
&\quad +\int_{(0,t]\times\mathbb{R}}(z-S^{\bouncy}_{s-})~\left\{R(\dif s,\dif z)-\rho~\dif s\, \phi(z)\, \dif z\right\}.
\end{align*} 
The predictable process part is 
\begin{align}\label{eq:bps_angle_b}
B_t'=B_t'(S^{\bouncy})
~=~t- 2\int_0^t \left\{(S^{\bouncy}_{s})^+\right\}^2\dif s-\rho\int_0^{t}S^{\bouncy}_{s}\dif s, 
\end{align} 
which is the sum of the deterministic part $t$ and the intensity measure of the random measure part.  
By Theorem II.1.33, the predictable quadratic variation of $M$ is 
\begin{align*}
\tilde{C}'_t(S^{\bouncy})& :=~4\int_0^t\left\{(S^{\bouncy}_{s})^+\right\}^3\dif s+\rho\int_0^t\left(1+(S^{\bouncy}_s)^2\right)\dif s. 
\end{align*}
The random measure $\mu=\mu^{S^{\bouncy}}$ is defined by the integral form
\begin{align*}
	g(x)*\mu_t&:=\int g(x)\mu_t(\dif x)\\
	&:=\int_{(0,t]\times\mathbb{R}_+} g(-2 S^{\bouncy}_{s-})1_{\{z\le S^{\bouncy}_{s-}\}}N(\dif s,\dif z)\\
	&\quad  + \int_{(0,t]\times\mathbb{R}} \left(g(z-S^{\bouncy}_{s-})\right)R(\dif s,\dif z)
\end{align*}
where $g:\mathbb{R}\rightarrow [0,\infty)$ is a continuous bounded function. 
The random measure $\nu(\omega;\dif t,\dif x)$ is its compensator which is defined by
\begin{equation}
	\label{eq:bps_angle_nu}
\begin{split}
	g(x)*\nu_t&:=\int g(x)\nu_t(\dif x)\\
	&:=\int_0^t g(-2 S^{\bouncy}_s)(S^{\bouncy}_s)^+\dif s + \rho\int_0^t\int_{\mathbb{R}} \left(g(z-S^{\bouncy}_{s})\right)\dif s\, \phi(z)\dif z. 
\end{split}
\end{equation}
By this decomposition $S^{\bouncy}$ is also a homogeneous jump process in the sense of Section III.2c, where 
$b(x)=1-2(x^+)^2-\rho x$, $c(x)\equiv 0$, and $K(x,\dif y)=(x^+)\delta_{\{-2x\}}(\dif y)+\rho~\phi(y-x)\dif y$. 

On the other hand, the process  $S^{\bouncy,d}$ is not a Markov process, and has the expression 
\begin{equation}
\begin{split}
S^{\bouncy,d}_t&=~S^{\bouncy,d}_0+t-2\int_{(0,t]\times\mathbb{R}_+} S^{\bouncy,d}_{s-} 1_{\{z\le S^{\bouncy,d}_{s-}\}}N(\dif s,\dif z)\\
&\quad+\int_{(0,t]\times\mathfrak{S}^{d-1}}\left(\langle \xi_s^{\bouncy,d},u\rangle-S^{\bouncy,d}_{s-}\right)R_d(\dif s,\dif u), \label{eq:tildestd}
\end{split}
\end{equation}
by It\^{o}'s formula. 
We denote $(B'^d, C^d,\nu^d)$ and $\tilde{C}'^d$ for the characteristics and modified second characteristic of $S^{\bouncy,d}$. As in the above example, we have
\begin{align*}
B'^d_t&:=t-2\int_0^t\left\{(S^{\bouncy,d}_s)^+\right\}^2\dif s-\rho\int_0^tS^{\bouncy,d}_s\dif s,\\ 
\tilde{C}'^d_t&:=4\int_0^t\left\{(S^{\bouncy,d}_s)^+\right\}^3\dif s+\rho\int_0^t\left(\frac{\|\xi_{s}^{\bouncy,d}\|^2}{d}+(S^{\bouncy,d}_s)^2\right)\dif s, 
\end{align*}
and 
\begin{align*}
	g(x)*\nu^d_t:=\int g(x)\nu^d_t(\dif x)&:=\int_0^t g(-2S^{\bouncy,d}_s)(S^{\bouncy,d}_s)^+\dif s\\
	&\quad + \rho\int_0^t\int_{\mathfrak{S}^{d-1}} g\left(\langle\xi_{s}^{\bouncy,d},u\rangle-S^{\bouncy,d}_s\right)\dif s \psi_d(\dif u)
\end{align*}
for a continuous bounded function $g$. 

Finally, we introduce strong majorisation property which is important to prove tightness of the sequence of processes. For two increasing processes $X=(X_t)_{t\ge 0}, Y=(Y_t)_{t\ge 0}$, $X$ \emph{strongly majorises} $Y$ if $X-Y=(X_t-Y_t)_{t\ge 0}$
is an increasing process, that is, almost all paths of $X_t(\omega)-Y_t(\omega)$ is increasing; see \cite[Definition VI.3.34]{JS}. 
We denote $Y\prec X$ if $X$ strongly majorises $Y$.

\subsubsection{Some remark on spherically symmetric distribution}\label{sec:remark_spherical}

Some of the characteristics of semimartingales $S^{\bouncy,d}$ and $Y^{\bouncy,d}$ are written by the expectation of $U^d$ which will be defined in (\ref{eq:ud}), and $U^d$ will be approximated by a Gaussian random variable. We will quantify this approximation error by the result in {\cite{MR898502}.

As mentioned above, we need to show that 
\begin{equation}\label{eq:ud}
	U^d:=d^{1/2}\langle e, v\rangle,
\end{equation}
where $v\sim\psi_d$ 
and $e$ is a unit vector, converges to the standard normal distribution 
and we need to quantify the approximation error. 
The distribution is extensively studied by \cite{MR898502}.  
For example, since $|\langle e,v\rangle|^2$ follows the Beta distribution with parameters $1/2$ and $(d-1)/2$, we have
\begin{equation}\label{eq:beta_moment}
\mathbb{E}\left[|U^d|^\alpha\right]=\frac{d^{\alpha/2}B(\frac{\alpha+1}{2},\frac{d-1}{2})}{B(\frac{1}{2},\frac{d-1}{2})}\longrightarrow_{d\rightarrow\infty}  \frac{\Gamma(\frac{\alpha+1}{2})}{\Gamma(\frac{1}{2})}2^{\alpha/2}
\end{equation}
for $\alpha>-1$, where we used Stirling's approximation.
Moreover, 
\begin{equation}\label{eq:tv_convergence}
	\|\mathcal{L}(U^d )-\mathcal N(0,1)\|_{\mathrm{TV}}=O(1/d)
\end{equation}
for $\|\nu\|_{\mathrm{TV}}=\sup|\int h(x)\nu(\dif x)|$ where the supremum is evaluated over those measurable function $h(x)$ bounded above by $1$. 
Since the expectations in the semimartingale characteristics are not bounded functions, we need the following proposition to quantify the approximation error. 

%

\begin{prop}\label{prop:diaconis}
For any $\epsilon>0, k\in\mathbb{N}$ and $W\sim \mathcal{N}(0,1)$, 
\begin{equation}\nonumber
	\sup_{|h(x)|\le (1+|x|)^k}\left|\mathbb{E}[h(U^d)]-\mathbb{E}[h(W)]\right|=O(d^{\epsilon-1}). 
\end{equation}
\end{prop}

\begin{proof}
Without loss of generality, we can assume $\epsilon\in (0,1/2)$. 
Let $|h(x)|\le (1+|x|)^k$. 
To apply (\ref{eq:tv_convergence}), we consider a bounded modification 
\begin{equation*}
h_a=h(x)~1_{\{|h(x)|\le a\}}
\end{equation*}
for $a>0$. Then
\begin{align*}
\left|\mathbb{E}[h_{d^{\epsilon}}(U^d)]
-
\mathbb{E}[h_{d^{\epsilon}}(W)]
\right|\le 
d^{\epsilon} \|\mathcal{L}(U^d )-\mathcal N(0,1)\|_{\mathrm{TV}}=O(d^{\epsilon-1}). 
\end{align*}
By Markov's inequality, the error due to the modification of $h(U^d)$ is 
\begin{align*}
\left|\mathbb{E}[h_{d^{\epsilon}}(U^d)]
-
\mathbb{E}[h(U^d)]
\right|& \le 
\mathbb{E}\left[|h(U^d)|, |h(U^d)|>d^{\epsilon}\right]\\
& \le \mathbb{E}\left[|h(U^d)|\left\{\frac{|h(U^d)|}{d^{\epsilon}}\right\}^{(1-\epsilon)/\epsilon}\right]\\
& \le d^{\epsilon - 1}\mathbb{E}\left[(1+|U^d|)^{k(1+(1-\epsilon)/\epsilon)}\right]=O(d^{\epsilon -1})
\end{align*}
by (\ref{eq:beta_moment}). Similarly, the error 
due to the modification of $h(W)$ is dominated by 
\begin{align*}
\left|\mathbb{E}[h_{d^{\epsilon}}(W)]
-
\mathbb{E}[h(W)]
\right|& \le 
\mathbb{E}\left[|h(W)|, |h(W)|>d^{\epsilon}\right]\\
& \le \mathbb{E}\left[|h(W)|\left\{\frac{|h(W)|}{d^{\epsilon}}\right\}^{(1-\epsilon)/\epsilon}\right]\\
& \le d^{\epsilon-1}\mathbb{E}_y\left[(1+|W|)^{k(1+(1-\epsilon)/\epsilon)}\right]=O(d^{\epsilon-1}). 
\end{align*}
Hence the claim follows by the triangle inequality. 
\end{proof}

\subsubsection{Remark on Stein's method}

We will use a martingale problem approach for the convergence of stochastic processes and hence we will show the convergence of 
characteristics of semimartingales. In order to prove the convergence of characteristics, 
we will use Stein's identity and Stein's method.  

Thanks to the results in Section \ref{sec:remark_spherical}, the semimartingale characteristics
are, essentially, written by expectations with respect to normal distributions. 
For calculation involving Gaussian random variables, Stein's identity is useful:
\begin{equation}\label{eq:stein}
	\mathbb{E}[Wf(W)]=\mathbb{E}[f'(W)]
\end{equation}
 where $W\sim \mathcal N(0,1)$ and $f$ is sufficiently smooth. 

\textbf{Stein's equation} (\ref{eq:stein}) characterises the standard normal distribution: $W\sim \mathcal N(0,1)$ if and only if (\ref{eq:stein}) is satisfied for every differentiable function $f$ with $\mathbb{E}|f'(W)|<\infty$. Moreover,  by using \textbf{Stein's method}, the deviation from $\mathcal N(0,1)$ is bounded by 
the deviation from  Stein's equation. The usefulness of the Stein's method is illustrated in the monographs
\citet{Stein} and \citet{MR2962301}. 
In this paper, we will use the following result due to Proposition 3.2.2 of \cite{MR2962301}. 

\begin{lem}\label{lem:stein}
For any $h:\mathbb{R}\rightarrow\mathbb{R}$ such that $\mathbb{E}[|h(W)|]<\infty$ for $W\sim\mathcal{N}(0,1)$,  there is 
the unique solution $f:\mathbb{R}\rightarrow\mathbb{R}$ of the ordinary differential equation (called Stein's equation)
\begin{equation}\label{Steineq0}
Lf(s):=f'(s)-sf(s)=h(s)-\mathbb{E}[h(W)] 
\end{equation}
such that $\lim_{x\rightarrow\pm\infty}\phi(x)f(x)=0$. 
\end{lem}

There are many important properties of the solution of Stein's equation. We remark here the integration-by-parts formula
\begin{equation}\label{eq:ibp}
	\int (Lf)(x)g(x)\phi(x)\dif x=-\int f(x)g'(x)\phi(x)\dif x
\end{equation}
for smooth functions $f, g$. Also, we would like to remark the following lemma  which provides a sufficient condition for $\mathcal N(0,1)$-integrability of Stein's solution. 
For $\beta>0$, let 
$$
\vertiii{f}_\beta=\sup_{x\in\mathbb{R}}e^{-\beta|x|}|f(x)|. 
$$
If $\vertiii{f}_\beta<\infty$, $f$ is $\mathcal{N}(0,1)$-integrable. 

\begin{lem}\label{lem:stein_bound}
For $\beta>0$, there exists $C_\beta<\infty$ such that 
for any $h:\mathbb{R}\rightarrow\mathbb{R}$, such that 
$\mathbb{E}[h(W)]=0$ for $W\sim \mathcal{N}(0,1)$, we have
\begin{align}\label{eq:stein_bound}
	\vertiii{f}_\beta\le C_\beta\vertiii{h}_\beta, 
\end{align}
where $f$ is the solution to~\eqref{Steineq0} such that $\lim_{x\rightarrow\pm\infty}\phi(x)f(x)=0$. 
\end{lem}

\begin{proof}
Without loss of generality, we can assume $\vertiii{h}_\beta<\infty$. By equation (3.23) of \cite{MR2962301}, Stein's solution is given by
\begin{align*}
	f(x)=\phi(x)^{-1}\int_{-\infty}^x~h(y)~\phi(y)\dif y
	=-\phi(x)^{-1}\int_{x}^\infty~h(y)~\phi(y)\dif y. 
\end{align*}
Therefore, if $x\ge 0$, we have
\begin{align}
		e^{-\beta x}|f(x)|
	&=(e^{\beta x}\phi(x))^{-1}\left|\int_{x}^\infty~h(y)~\phi(y)\dif y\right|\nonumber\\
	& \le (e^{\beta x}\phi(x))^{-1}\int_{x}^\infty~\left|h(y)\right|~\phi(y)\dif y\nonumber\\
	& \le ~\vertiii{h}_\beta~(e^{\beta x}\phi(x))^{-1}\int_{x}^\infty e^{\beta y} \phi(y)\dif y. \label{eq:upper_bound_ft}
\end{align}
With a similar calculation for $x\le 0$, we obtain the inequality (\ref{eq:stein_bound}) with the constant
\begin{align*}
	C_\beta=\sup_{x\ge 0}c_\beta(x),\ c_\beta(x):=\left(e^{\beta x}\phi(x)\right)^{-1}\int_{x}^\infty e^{\beta y}\phi(y)\dif y. 
\end{align*}
Observe that $e^{\beta x}\phi(x)=e^{\beta^2/2}\phi(x-\beta)$. 
Also, if $y\ge 1$, we have $\phi(y)\le y\phi(y)$ and hence 
	$\Phi(-x)\le \phi(x)$ by integrating $y\in [x,\infty)$. Therefore, 
	if $x\ge \beta+1$, 
\begin{align*}
	c_\beta(x)
	& = \phi(x-\beta)^{-1}\int_{x}^\infty \phi(y-\beta)\dif y\\
	& = \phi(x-\beta)^{-1}\Phi(-(x-\beta))\le 1. 
\end{align*}	
Also, $x\mapsto c_\beta(x)$ is continuous, and hence bounded on 
$[0,\beta+1]$. 
Hence $C_\beta<\infty$ and the claim follows. 
\end{proof}

%
%

\subsection{Proof of Theorem \ref{theo:bps_limit}}

\begin{proof}[Proof of Theorem \ref{theo:bps_limit}]
We apply \cite[Theorem IX.3.48]{JS} to $S^{\bouncy,d}$
with stopping time
$$
\tau_a(S^{\bouncy})=\inf\{t>0: |S^{\bouncy}_t|\ge a\ \mathrm{or}\ |S^{\bouncy}_{t-}|\ge a\}
$$
for $a>0$. Let $\tau_a^d=\tau_a(S^{\bouncy,d})$. Firstly we prove the local strong majoration hypothesis (i) of Theorem IX.3.48. 
By the expression of the predictable process $B'$ in (\ref{eq:bps_angle_b}), the total variation process (see Section I.3a) of $B'$ up to the stopping time $\tau_a$ is 
\begin{align*}
\mathrm{Var}(B')^{\tau_a}_t&=\int_0^{t\wedge \tau_a}\left|1-2\left\{(S^{\bouncy}_{s})^+\right\}^2-\rho~S^{\bouncy}_s\right|\dif s. 
\end{align*}
By construction of $\nu$ in (\ref{eq:bps_angle_nu}), we have
\begin{align*}
\left\{ |x|^2*\nu\right\}^{\tau_a}_t=\int|x|^2\nu_{t\wedge \tau_a}(\dif x)&=\int_0^{t\wedge \tau_a}\left\{4\left\{(S^{\bouncy}_{s})^+\right\}^3+\rho(1+(S^{\bouncy}_s)^2)\right\}\dif s.
\end{align*}
Hence 
$$
\mathrm{Var}(B')^{\tau_a}\prec F_1(a),\ 
\left\{ (|x|^2)*\nu\right\}^{\tau_a}\prec F_2(a)
$$
where 
$$
F_1(a)_t=t~(1+2a^2+\rho~a),\ 
F_2(a)_t=t~( 4a^3+\rho(1+a^2)). 
$$
Note that $C\equiv 0$. 
Thus (i) of Theorem IX.3.48 follows, since $\mathrm{Var}(B')^{\tau_a}$ and $\left\{ (|x|^2)*\nu\right\}^{\tau_a}$ are strongly majorised by $F(a)=F_1(a)+F_2(a)$.

Secondly we prove (ii)-(v) of Theorem IX.3.48. If we take $b>2a$, then 
\begin{align*}
\left\{ |x|^21_{\{|x|>b\}}*\nu\right\}^{t \wedge \tau_a} & = 
\rho\int_0^{t\wedge\tau_a}\int_\mathbb{R}|z-S_s^{\bouncy}|^21_{\{|z-S_s^{\bouncy}|>b\}}\phi(z)\dif z\dif s
\\
& \le \rho\int_0^t\int_\mathbb{R}(|z|+a)^21_{\{|z|+a>b\}}\phi(z)\dif z\dif s\longrightarrow_{b\rightarrow+\infty}~0
\end{align*}
which proves (ii) of Theorem IX.3.48. The existence and uniqueness of the  solution of (\ref{eq:limitBPS}) is proved in Section \ref{subsec:ergodiclimit}. Therefore, existence and uniqueness of the corresponding martingale problem follows from Theorem 2.3 of \cite{MR2789081} together with the fact that $\mathbb{P}(\tau_\infty=\infty)=1$. Thus local uniqueness condition (iii) of Theorem IX.3.48 comes from Lemma IX 4.4. Continuity condition (iv) is obvious. 
Since we assume stationarity, both $S^{\bouncy,d}_0$ and $S^{\bouncy}_0$ follows the standard normal distribution. Thus (v) of Theorem IX.3.48 follows. 

Finally we check the condition (vi) of Theorem IX.3.48. 
Recall that, by construction, 
\begin{equation}
\label{eq:raduis_estimate}
	\|\xi_t^{\bouncy,d}-\xi_0^{\bouncy,d}\|\le t~\Longrightarrow~
\sup_{0\le t\le T}\left|\frac{\|\xi_t^{\bouncy,d}\|^2}{d}-1\right|=o_\mathbb{P}(1)
\end{equation}
for any $0\le t\le T$. Thus for any $0\le s\le t$, 
$$
B_s'^d-B_s'(S^{\bouncy,d})=0,\ 
|\tilde{C}_s'^d-\tilde{C}'_s(S^{\bouncy,d})|\le \rho\int_0^t\left|~\frac{\|\xi_{s}^{\bouncy,d}\|^2}{d}-1~\right|\dif s=o_\mathbb{P}(1), 
$$
and hence the conditions [Sup-$\beta'_{\mathrm{loc}}$]  and $[\gamma'_{\mathrm{loc}}$-D] of (vi) are satisfied. 
For Condition IX.3.49 of (vi), 
let $g_b(x)=x^2 1_{\{|x|>b\}}$ for $b>2a$. Then 
\begin{align*}
	g_b * \nu_{t\wedge \tau_a^d}^d & =\rho\int_0^{t \wedge \tau_a^d} \int_{\mathfrak{S}^{d-1}}g_b(\langle\xi_s^{\bouncy,d},u\rangle-S^{\bouncy,d}_s)\psi_d(\dif u)\dif s \\
	& \leq \rho \int_0^t \int_{\mathfrak{S}^{d-1}} g_b(|\langle\xi_s^{\bouncy,d},u\rangle|+a)  \psi_d(\dif u)\dif s.
\end{align*}
By stationarity together with the fact that 
$\mathcal{L}(\langle\xi_0^{\bouncy,d},u\rangle)=\mathcal{L}(S^{\bouncy,d}_0)$, we have
\begin{align*}
	\mathbb{P}\left( g_b * \nu_{t\wedge \tau_a^d}^d>\epsilon\right)
&\le
\epsilon^{-1}\mathbb{E}\left[ \rho \int_0^t \int_{\mathfrak{S}^{d-1}} g_b(|\langle\xi_s^{\bouncy,d},u\rangle|+a) \psi_d(\dif u)\dif s\right]\\
&= \epsilon^{-1}t~\rho~\mathbb{E}\left[g_b(|S^{\bouncy,d}_0|+a)\right]. 
\end{align*}
Therefore, by taking the lim sup as $d \rightarrow \infty$ of the expectation on the right-hand side of the above inequality gives
$$
\limsup_{d \rightarrow \infty} \mathbb{P}\left( g_b * \nu_{t\wedge \tau_a^d}^d>\epsilon\right) \leq  \epsilon^{-1}~t~\rho~\mathbb{E}[g_b(|S^{\bouncy}_0|+a)]\longrightarrow_{b\rightarrow\infty} 0
$$
by $S_0^{\bouncy}\sim \mathcal{N}(0,1)$ which establishes Condition IX.3.49 of (vi). Finally, we check [$\delta_{\mathrm{loc}}$-D] of (iv). 
%
By construction for any bounded, continuous function $g$, we have
\begin{align*}
\epsilon_t^d:& =~g(x)*\nu_t^d-(g(x)*\nu_t)\circ S^{\bouncy,d}\\
& =~
\rho\int_0^t\int_{\mathfrak{S}^{d-1}}g(\langle \xi_s^{\bouncy,d},u\rangle-S^{\bouncy,d}_s)\psi_d(\dif u)\dif s-
\rho\int_0^t\int_\mathbb{R} g(z-S^{\bouncy,d}_s)\phi(z)\dif z\dif s. 
\end{align*}
Therefore, by stationarity of the process, we have
\begin{align*}
&\mathbb{E}\left[\sup_{0\le s\le t}|\epsilon_s^d|\right]\\
&\le t~\rho~\mathbb{E}\left[\left|\int_{\mathfrak{S}^{d-1}}g(\langle \xi_0^{\bouncy,d},u\rangle-S^{\bouncy,d}_0)\psi_d(\dif u)-\int_\mathbb{R} g(z-S^{\bouncy,d}_0)\phi(z)\dif z\right|\right]\\
&\le t~\rho~\|g\|_\infty~\mathbb{E}\left[\left\|\mathcal{L}_0(\langle \xi_0^{\bouncy,d},u\rangle)-\mathcal{N}(0,1)\right\|_{\mathrm{TV}}\right], 
\end{align*}
where $u\sim\psi_d$ and $\mathcal{L}_0(X)$ is the conditional distribution of $X$ given $\xi_0^{\bouncy,d}$ and $ v_0^{\bouncy, d}$, and 
$\|g\|_\infty=\sup_{x\in\mathbb{R}}|g(x)|$. By the property of the spherically symmetric distribution $\psi_d$,  we have
\begin{equation}\label{eq:law_of_s}
\begin{split}
\mathcal{L}_0(\langle \xi_{0}^{\bouncy,d},u\rangle )
&=\mathcal{L}_0\left(\frac{\|\xi_0^{\bouncy,d}\|}{d^{1/2}}d^{-1/2}\left\langle \frac{\xi_0^{\bouncy,d}}{\|\xi_0^{\bouncy,d}\|},u\right\rangle \right)=\mathcal{L}_0(\alpha^d~U^d),\\ 
(\alpha^d)^2&:=\frac{\|\xi_0^{\bouncy,d}\|^2}{d}. 
\end{split}
\end{equation}
Therefore the total variation distance in the above expectation is 
\begin{align*}
\left\|\mathcal{L}_0(\alpha^d U^d)-\mathcal{N}(0,1)\right\|_{\mathrm{TV}}
&\le 
\left\|\mathcal{L}_0(\alpha^d U^d)-\mathcal{N}(0,(\alpha^d)^2)\right\|_{\mathrm{TV}}\\
&\quad +
\left\|\mathcal{N}(0,(\alpha^d)^2)-\mathcal{N}(0,1)\right\|_{\mathrm{TV}}. 
\end{align*}
The first term in the right-hand side equals to (\ref{eq:tv_convergence}) which converges to $0$, and the second term is dominated by 
$$
2\left|1-(\alpha^d)^2\right|~\longrightarrow_{d\rightarrow\infty}~0\ \mathrm{in}\ \mathbb{P}
$$
by Proposition 3.6.1 of \cite{MR2962301}. This proves [$\delta_{\mathrm{loc}}$-D]. 
Thus, the condition (iv) of Theorem IX.3.48 of \cite{JS} is proved. 
Hence the claim follows. 
\end{proof}

\subsection{Proof for Corollary \ref{cor:bps_switches}}

\begin{proof}[Proof for Corollary \ref{cor:bps_switches}]
By the expression (\ref{eq:tildestd}), 
the expected number of switches of $S^{\bouncy,d}$ per unit time  is 
\begin{align*}
\mathbb{E}\left[\sum_{0\le t\le T}1_{\{\Delta S^{\bouncy,d}_t\neq 0\}}\right]
	&=~
\mathbb{E}\left[\int_{(0,T]\times\mathbb{R}_+} 1_{\{z\le S^{\bouncy,d}_{s-}\}}N(\dif s,\dif z)
+ R_d((0,T]\times\mathbb{R})\right]\\
&=~\mathbb{E}\left[\int_0^T(S^{\bouncy,d}_s)^+\dif s
+\rho~T\right]\\
&=~T~\mathbb{E}\left[(S^{\bouncy,d}_0)^++\rho\right]\\
&=~T~\left\{\int_\mathbb{R}x^+\phi(x)\dif x + \rho\right\}=T\left(\frac{1}{\sqrt{2\pi}}+\rho\right).
\end{align*}
\end{proof}

\subsection{Proof for Theorem \ref{theo:bps_ll}}\label{sec:bps_diffusion}

Thanks to the memoryless property of the exponential distribution, we can assume that a refreshment jump occurs at $t=0$ since it does not affect the law of $(\xi_t^{\bouncy,d},v_t^{\bouncy,d})_{t\ge 0}$. 
By Proposition II.1.14 of \cite{JS}, we can construct a probability space so that there are stopping times $0=\sigma_0<\sigma_1<\sigma_2<\cdots$ 
with $\mathcal{F}_{\sigma_n}$-measurable random variables $W_n^d\ (n\ge 1)$ such that 
\begin{equation}\label{eq:r_representation}
R_d(\dif t,\dif x)=\sum_{n\ge 1}1_{\{\sigma_n<\infty\}}\delta_{(\sigma_n,W_n^d)}(\dif t,\dif x),
\end{equation}
where $\mathbb{P}(W_n^d\in A|\mathcal{F}_{\sigma_n-})=\psi_d(A)$. 
%

The proof strategy of Theorem \ref{theo:bps_ll} is as follows. 
The first step is to show the convergence of $Y^{\bouncy, d}$ at refreshment times $(\sigma_n)_{n\ge 0}$. For that purpose, we consider a pure step Markov process  $\overline{Y}^{\bouncy,d}$ defined by
$$
\overline{Y}_t^{\bouncy,d}:= \sum_{n\ge 0}Y_{\sigma_n/d}^{\bouncy,d}~1_{\left[\frac{\sigma_n}{d},\frac{\sigma_{n+1}}{d}\right)}(t)
= \sum_{n\ge 0}d^{1/2}\left(\frac{\|\xi_{\sigma_n}^{\bouncy,d}\|^2}{d}-1\right)~1_{\left[\frac{\sigma_n}{d},\frac{\sigma_{n+1}}{d}\right)}(t).
$$
The pure step Markov process has a  simpler structure which is characterised by the so-called finite transition measure. 
Since $\sigma_j/d-\sigma_{j-1}/d$ follows the exponential distribution with mean $1/\rho d$, its finite transition measure $K^d(x,\dif y)$ is 
$$
\int_\mathbb{R} f(y)K^d(x,\dif y)=\rho~d~\mathbb{E}[f(Y^{\bouncy,d}_{\sigma_1/d}-Y^{\bouncy,d}_{0})|Y^{\bouncy,d}_0=x]
$$
in the sense of IX.4.19 of \cite{JS}. Then we will apply Theorem IX.4.21 of \cite{JS} to the Markov process $\overline{Y}^{\bouncy,d}$ in Lemma \ref{lem:bouncy_bar}. To apply the theorem, key step is the proof for the convergence of the semimartingale characteristics. For this step, Stein's techniques work efficiently. 
After the proof of Lemma \ref{lem:bouncy_bar}, finally we will show that the difference between 
$\overline{Y}^{\bouncy,d}$ and  $Y^{\bouncy, d}$ is ignorable. 

\begin{lem}\label{lem:bouncy_bar}
	The process $\overline{Y}^{\bouncy,d}$ converges in law to $Y^{\bouncy}$. 
\end{lem}

\begin{proof}
We can construct $(S^{\bouncy,d}_t)_{t\in [0,\sigma_1)}$ so that 
\begin{equation}\label{eq:law_equivalence}
	S_t^{\bouncy,d}=
		\mathcal{T}_t
		\ (0\le t<\sigma_1)
\end{equation}
where $\mathcal{T}$ follows (\ref{eq:limitzigzag}) with $\mathcal{T}_0=S_0^{\bouncy,d}=x$, and independent from the refreshment times $(\sigma_n)_{n\ge 0}$.  
We apply Theorem IX.4.21 of \cite{JS}. Since the limiting process is the Ornstein-Uhlenbeck process, hypothesis \cite[IX.4.3]{JS} is satisfied. By the central limit theorem, $\mathcal{L}(\overline{Y}_0^{\bouncy,d})$ converges to $\mathcal N(0,2)=\mathcal{L}(Y_0^\bouncy)$, and hence condition (iii) is also satisfied. Therefore, it is sufficient to prove conditions (i) and (ii). 

The condition (i) corresponds to the (locally uniformly in $y$) convergence of
\begin{align*}
	b'^d(y)&:=\rho~d~\mathbb{E}[Y_{\sigma_1/d}^{\bouncy,d}-Y_0^d|Y_0^{\bouncy,d}=y] \quad \text{and}\\ 
	\tilde{c}'^d(y)&:=\rho~d~\mathbb{E}[(Y_{\sigma_1/d}^{\bouncy,d}-Y_0^d)^2|Y_0^{\bouncy,d}=y]. 
\end{align*}
For simplicity, we will denote $\mathbb{E}[~\cdot~|Y_0^{\bouncy,d}=y]$
by $\mathbb{E}_{y}[ \cdot ]$. 
Firstly, we check the convergence of the drift coefficient $b'^d$. 
Since $\dif\|\xi_t^{\bouncy,d}\|^2=2 S_t^{\bouncy,d}\dif t$, we have
\begin{align*}
\|\xi_{\sigma_1}^{\bouncy,d}\|^2-\|\xi_{0}^{\bouncy,d}\|^2
	=2\int_0^{\sigma_1}S^{\bouncy,d}_t\dif t	=2\int_0^{\sigma_1}\mathcal{T}_t\dif t=2\int_0^\infty1_{\{t\le \sigma_1\}}\mathcal{T}_{t}\dif t. 
\end{align*}
Since $\sigma_1$ and $\mathcal{T}$ are independent, we can rewrite $b'^d(y)$ as 
\begin{align*}
	b'^d(y)	&=\rho~d^{1/2}\mathbb{E}_y\left[\|\xi_{\sigma_1}^{\bouncy,d}\|^2-\|\xi_{0}^{\bouncy,d}\|^2\right]\\
	&=2\rho~d^{1/2}\int_0^\infty \mathbb{P}_y(t\le\sigma_1)~\mathbb{E}_y\left[\mathbb{E}[\mathcal{T}_t|\mathcal{T}_0=S^{\bouncy,d}_0]\right]\dif t\\
	&=2\rho~d^{1/2}\int_0^{\infty}e^{-\rho t}\mathbb{E}_y\left[h_t(S^{\bouncy,d}_0)\right]\dif t
\end{align*}
where $h_t(x):=\mathbb{E}[\mathcal{T}_t|\mathcal{T}_0=x]$. Now we are going to approximate $S_0^{\bouncy,d}$ by a Gaussian random variable.   
For $\alpha>0$, by (\ref{eq:square_t}), we have
$$
|h_t(\alpha x)|= \left|\mathbb{E}[\mathcal{T}_t|\mathcal{T}_0=\alpha x]\right|\le 
\mathbb{E}\left[\left|\mathcal{T}_t\right||\mathcal{T}_0=\alpha x\right]\le |\alpha x|+t\le (|\alpha|+t)(1+|x|). 
$$
Conditioned on $y$, we show that the difference of the law of  $S_0^{\bouncy, d}$ 
and the normal distribution $\mathcal N(0,(\alpha^d)^2)$ is small, where  
$$
(\alpha^d)^2 :=\frac{\|\xi_0^{\bouncy,d}\|^2}{d}=1+d^{-1/2}y. 
$$
By the property of $\psi_d$,  we can rewrite the expectation of $S_0^{\bouncy, d}$ in terms of $U^d$ (see~\eqref{eq:ud}) since 
$\mathcal{L}_y(S^{\bouncy,d}_0)=\mathcal{L}_y(\alpha^d~U^d)$
as in (\ref{eq:law_of_s})
where $\mathcal{L}_y$ is the conditional distribution given $Y_0^{\bouncy,d}=y$. 
Therefore, we can apply Proposition \ref{prop:diaconis} with $k=1$ and $\epsilon\in (0,1/2)$ to $S_0^{\bouncy, d}$. We have 
\begin{align}\label{eq:hw_bound}
\left|\mathbb{E}_y[h_{t}(S^{\bouncy,d}_0)]
-
\mathbb{E}_y[h_t(\alpha^d W)]\right|
\le (|\alpha^d|+t)~O(d^{\epsilon-1}),
\end{align}
where $W\sim \mathcal{N}(0,1)$. 
Since $\alpha^d \rightarrow 1$ locally uniformly in $y$, we obtain that the drift coefficient is an expectation of Gaussian random variable with ignorable approximation error: 
\begin{align*}
	b'^d(y)
	&=2\rho~d^{1/2}\int_0^{\infty}e^{-\rho t}\mathbb{E}_y\left[h_t(\alpha^dW)\right]\dif t+O(d^{\epsilon-1/2}).  
\end{align*}
We are in a position to apply Stein's method. 
Let $f_t$ be the Stein's solution for $Lf_t=h_t$. 
Observe that $\mathbb{E}[h_t(W)]=\mathbb{E}[h_t(\mathcal{T}_0)]=\mathbb{E}[\mathcal{T}_t]=0$. 
By Lemma   \ref{lem:stein_bound}, $f_t$ and $f_t'=xf_t+h_t$ are $\mathcal{N}(0,1)$-integrable. Therefore, 
\begin{align*}
\mathbb{E}_y[h_t(\alpha^d W)]
&=	\mathbb{E}_y[f_t'(\alpha^d W)-\alpha^d W f_t(\alpha^d W)]\\
&=	\mathbb{E}_y[f_t'(\alpha^d W)-(\alpha^d)^2~ f_t'(\alpha^d W)]\\
&=	(1-(\alpha^d)^2)\mathbb{E}_y[f_t'(\alpha^d W)]\\
&=	-d^{-1/2}y~\mathbb{E}_y[f_t'(\alpha^d W)]
\end{align*}
where we used Stein's identity in the third line. Since $\alpha^d\longrightarrow_{d\rightarrow\infty} 1$ locally uniformly in $y$, by the dominated convergence theorem, we have
$$
b'^d(y)\longrightarrow_{d\rightarrow\infty} b'(y):=-2\rho y\int_0^\infty\int_\mathbb{R} e^{-\rho t}f'_t(x)\phi(x)\dif x\dif t. 
$$
To finish the calculation of the drift coefficient, we rewrite the expectation in the right hand side without using Stein's solution. 
By Stein's identity together with (\ref{eq:ibp}), 
\begin{align*}
	\int f_t'(x)\phi(x)\dif x&=
	\int xf_t(x)\phi(x)\dif x\\
	&=\int \left(\frac{x^2}{2}\right)'f_t(x)\phi(x)\dif x\\
	&=-\int \frac{x^2}{2}~h_t(x)\phi(x)\dif x\\
	&=-\mathbb{E}\left[\left(\frac{\mathcal{T}_0^2}{2}\right)~\mathbb{E}[\mathcal{T}_t|\mathcal{T}_0]\right]=-2^{-1}\mathbb{E}[\mathcal{T}_0^2\mathcal{T}_t].
\end{align*}
We used Stein's identity in the first line, and the integration by parts formula (\ref{eq:ibp}) with $f_t=Lh_t$ and $g(x)=x^2/2$ in the third line. 
We can rewrite this expectation as an integration with respect to the covariance function $K(s,t)$. 
By (\ref{eq:convergence_of_h}) with $k=2$, the right-hand side of the above equation equals 
\begin{align*}
2^{-1}\lim_{s\rightarrow\infty}
\mathbb{E}[(\mathcal{T}_s^2-\mathcal{T}_0^2)\mathcal{T}_t]
=\mathbb{E}\left[\int_0^\infty\mathcal{T}_s\mathcal{T}_t\dif s\right]
	=\int_0^\infty K(s,t)\dif s=\int_0^t K(s,0)\dif s
\end{align*}
where we used (\ref{eq:square_t}) in the first equation, and (\ref{eq:integral_k}) for the last equation. Therefore we obtain the expression of the drift coefficient:
\begin{align*}
	b'(y)
	=-2\rho y\int_0^\infty e^{-\rho t}\int_0^t K(s,0)\dif s \dif t
	=-2y\int_0^\infty e^{-\rho s} K(s,0)\dif s. 
\end{align*}

Secondly, we check convergence of the diffusion coefficient. By  $\dif \|\xi_t^{\bouncy,d}\|^2=2 S_t^{\bouncy,d}\dif t$,
\begin{align*}
	\tilde{c}'^d(y)&=\rho~\mathbb{E}_y\left[(\|\xi_{\sigma_1}^{\bouncy,d}\|^2-\|\xi_{0}^{\bouncy,d}\|^2)^2\right]\\
&=4\rho~\mathbb{E}_y\left[\left\{\int_0^{\sigma_1}S^{\bouncy,d}_t\dif t\right\}^2\right]=4\rho~\mathbb{E}_y\left[\left\{\int_0^{\sigma_1}\mathcal{T}_t\dif t\right\}^2\right].
\end{align*}
As in the drift coefficient case, since $\sigma_1$ and $\mathcal{T}_t$ are independent, we have
\begin{align*}
\mathbb{E}\left[\left.\left\{\int_0^{\sigma_1}\mathcal{T}_t\dif t\right\}^2\right|\mathcal{T}_0=S_0^{\bouncy,d}\right]
	&=
	\int_0^\infty\int_0^\infty\mathbb{E}\left[\left.1_{\{s, t\le\sigma_1\}}\mathcal{T}_t\mathcal{T}_s\right|\mathcal{T}_0=S_0^{\bouncy,d}\right]\dif t~\dif s\\
	&=
	\int_0^\infty\int_0^\infty e^{-\rho\max\{s,t\}}h_{s,t}(S_0^{\bouncy,d})\dif t~\dif s
\end{align*}
where $h_{s,t}(x)=\mathbb{E}\left[\left.\mathcal{T}_t\mathcal{T}_s\right|\mathcal{T}_0=x\right]$. 
Observe that if $t\ge s\ge 0$, by (\ref{eq:square_t}), we have
\begin{align*}
	|h_{s,t}(\alpha x)|=\left|\mathbb{E}\left[\left.\mathcal{T}_t\mathcal{T}_s\right|\mathcal{T}_0=\alpha x\right]\right|
	\le (|\alpha x|+t)(|\alpha x|+s)\le (|\alpha|+t)^2(1+|x|)^2. 
\end{align*}
Therefore by Proposition \ref{prop:diaconis} with $k=2$, we can approximate the expectation of $S_0^{\bouncy,d}$ by that of a Gaussian random variable: 
$$
|\mathbb{E}[h_{s,t}(S_0^{\bouncy, d})]-\mathbb{E}[h_{s,t}(\alpha^d \mathcal{T}_0)]|
\le (|\alpha^d|+\max\{s,t\})^2O(d^{\epsilon-1})
$$
for any $\epsilon\in (0,1)$. 
Therefore, we can conclude 
\begin{align*}
	\tilde{c}'^d(y)
	& = 4\rho~
	\int_0^\infty\int_0^\infty e^{-\rho\max\{s,t\}}\mathbb{E}_y\left[h_{s,t}(\alpha^d W)\right]\dif s\dif t+O(d^{\epsilon-1}). 
\end{align*}
Hence by the dominated convergence theorem, we have
\begin{align*}
	\tilde{c}'^d(y)\longrightarrow_{d\rightarrow\infty} 
4\rho~\int_0^\infty\int_0^\infty e^{-\rho\max\{s,t\}}K(s,t)\dif s \dif t=:\tilde{c}'(y), 
\end{align*}
since $\mathbb{E}[h_{s,t}(W)]=\mathbb{E}[h_{s,t}(\mathcal{T}_0)]=\mathbb{E}[\mathcal{T}_t\mathcal{T}_s]=K(s,t)$. 
By change of variable $(s,t)\mapsto(t-s,t)=:(u,t)$, we have
\begin{align*}
	\tilde{c}'(y)&=8\rho\int_{0<s\le t<\infty}e^{-\rho t}K(s,t)\,\dif s\, \dif t\\
	&=8\rho\int_0^\infty\int_u^\infty e^{-\rho t}K(u,0)\, \dif t\, \dif u\\
	&=8\int_0^\infty e^{-\rho u}K(u,0)\, \dif u. 
\end{align*}
Therefore, the condition (i) of Theorem IX.4.21 of \cite{JS} follows. 

Finally, we check condition (ii). By Markov property, for any $\epsilon>0$, we have
\begin{align*}
	\int_\mathbb{R} K^d(x,\dif y)|y|^21_{\{|y|>\epsilon\}}
& \le 
	\epsilon^{-2}	\int_\mathbb{R} K^d(x,\dif y)|y|^4=:\epsilon^{-2}\delta^d(y). 
\end{align*}
By construction of $K^d$, we can rewrite $\delta^d(y)$ as
\begin{align*}
\delta^d(y)&=\rho~d~\mathbb{E}[(Y_{\sigma_1/d}^{\bouncy,d}-Y_0^{\bouncy,d})^4|Y_0^{\bouncy,d}=y]. 
\end{align*}
By H\"older's inequality,
\begin{align*}
	\delta^d(y)&=\rho d^{-1}~\mathbb{E}_y\left[(\|\xi_{\sigma_1}^{\bouncy,d}\|^2-\|\xi_{0}^{\bouncy,d}\|^2)^4\right]\\
	&=16\rho d^{-1}~\mathbb{E}_y\left[\left\{\int_0^{\sigma_1}S^{\bouncy,d}_t\dif t\right\}^4\right]\\
	&=16\rho d^{-1}~\mathbb{E}_y\left[\left\{\int_0^{\sigma_1}\mathcal{T}_t\dif t\right\}^4\right]\\
	&\le16\rho d^{-1}~\mathbb{E}_y\left[\sigma_1^4(|S^{\bouncy,d}_0|+\sigma_1)^4\right]\\
	&=O(d^{-1})
\end{align*}
locally uniformly in $y$ where we used (\ref{eq:square_t}) in the inequality. 
Therefore, the condition (ii) follows. 
Thus, the claim follows by Theorem IX.4.21 of \cite{JS}. 
\end{proof}


\begin{proof}[Proof of Theorem \ref{theo:bps_ll}]
We showed that 
the process $\overline{Y}^{\bouncy,d}$ converges in law to $Y^{\bouncy}$. Therefore, 
by Lemma VI.3.31 of \cite{JS}, it is sufficient to show 
	$$
\epsilon^d_T:=\sup_{0\le t\le T}|Y_t^{\bouncy,d}-\overline{Y}_t^{\bouncy,d}|
\longrightarrow_{d\rightarrow \infty} ~0
	$$
in probability for any $T>0$. Let 
$$
A_T=(0,T]\times \mathbb{R} \quad \text{and} \quad \lambda_T=\rho T.
$$ 
Then $R_d(A_T)$ follows the Poisson distribution with mean $\lambda_T$. In particular,  $R_d(A_{dT})/d$ is tight. 
Since $R_d(A_T)$ is the number of the refreshment jumps until $T>0$,  we have
\begin{align*}
	\epsilon^d_T& \le\sup_{j\le R_d(A_{dT})}\sup_{\sigma_j\le dt< \sigma_{j+1}}|Y_t^{\bouncy,d}-Y_{\sigma_j/d}^{\bouncy,d}|
\end{align*}
On the other hand, for $\sigma_j\le dt< \sigma_{j+1}$, we have
\begin{align*}
|Y_t^{\bouncy,d}-Y_{\sigma_j/d}^{\bouncy,d}|&=d^{-1/2}|\|\xi_{td}^{\bouncy,d}\|^2-\|\xi_{\sigma_j}^{\bouncy,d}\|^2|\\
& \le 2d^{-1/2}\int_{\sigma_j}^{\sigma_{j+1}}|S_t^{\bouncy,d}|\dif t\\
& \le 2 d^{-1/2}\int_{\sigma_j}^{\sigma_{j+1}}(|S_{\sigma_j}^{\bouncy,d}|+t)~\dif t\\
& = 2 d^{-1/2}\left(|S_{\sigma_j}^{\bouncy,d}|(\sigma_{j+1}-\sigma_j)+\frac{1}{2}(\sigma_{j+1}-\sigma_j)^2\right), 
\end{align*}
where we used (\ref{eq:square_t}) in the third line. 
Therefore, for any $J\in\mathbb{N}$,
\begin{align*}
	\mathbb{P}(\epsilon^d_T>\epsilon)& \le \mathbb{P}(R_d(A_{dT})>dJ)\\
	&\quad+\mathbb{P}\left(2
	d^{-1/2}\sup_{j\le dJ}\left(|S_{\sigma_j}^{\bouncy,d}|(\sigma_{j+1}-\sigma_j)+\frac{1}{2}(\sigma_{j+1}-\sigma_j)^2\right)>\epsilon\right)\\
& \le \mathbb{P}(R_d(A_{dT})>dJ)+dJ\mathbb{P}\left(2
	d^{-1/2}\left(|S_0^{\bouncy,d}|\sigma_1+\frac{1}{2}\sigma_1^2\right)>\epsilon\right). 
\end{align*}
If we take $J\in\mathbb{N}$ large enough, the first probability in the right-hand side of the above inequality can be small. 
The second term converges to $0$ by Markov's inequality together with the fact that 
$S_0^{\bouncy,d}\sim \mathcal{N}(0,1)$
and $\sigma_1$ follows the exponential distribution with mean $1/\rho$. Hence the claim follows. 
\end{proof}

\subsection{Proof for Proposition \ref{prop:limit_of_sigma}}

\begin{proof}[Proof for Proposition \ref{prop:limit_of_sigma}]
By Proposition \ref{prop:covariance} together with Lebesgue's dominated convergence theorem, the claim is obvious. 
\end{proof}

\subsection{Proof for Theorem \ref{theo:bps_coordinate}}

Firstly, we prove that the process $Z^{\bouncy, d,k}$ can be approximated by a pure step Markov process.  Secondly, we show that this approximated process converges to an Ornstein-Uhlenbeck process  which completes the proof of Theorem \ref{theo:bps_coordinate}. 

\subsubsection{Approximation of the process}

Let 
$$
\overline{Z}_t^{\bouncy,d,k}:= \sum_{n\ge 0}Z_{\sigma_n/d}^{\bouncy,d,k}~1_{\left[\frac{\sigma_n}{d},\frac{\sigma_{n+1}}{d}\right)}(t), 
$$	
be the pure step version of $Z^{\bouncy, d,k}$. 
By construction, we have the following decomposition imitating the Doob-Meyer decomposition 
\begin{equation}\label{eq:doob_meyer}
\Delta\overline{Z}_{\sigma_{n+1}/d}^{\bouncy,d,k}:=\overline{Z}_{\sigma_{n+1}/d}^{\bouncy,d,k}-\overline{Z}_{\sigma_{n}/d}^{\bouncy,d,k}=\int_{\sigma_n}^{\sigma_{n+1}}\pi_k(v^{\bouncy,d}_t)\dif t=M_{n+1}^d+A_{n+1}^d 
\end{equation}
where 
$$
M_{n+1}^d=\int_{\sigma_n}^{\sigma_{n+1}}\pi_k(v_{\sigma_n}^{\bouncy, d})~\dif t,\ 
A_{n+1}^d=\int_{\sigma_n}^{\sigma_{n+1}}\pi_k(v_t^{\bouncy,d}-v_{\sigma_n}^{\bouncy, d})~\dif t. 
$$

Now we want to extract a predictable component from $A_{n+1}^d$. 
Let $(\mathcal{F}_t^d)_{t\ge 0}$ be the underlying filtration. For $N\in\mathbb{N}$, we show the following. 

\begin{lem}\label{lem:component-1}
\[
\mathbb{E}\left[\sum_{i=0}^{dN-1}\|A_{i+1}^d\|^2\right]\longrightarrow 0. 
\]
\end{lem}

\begin{proof}
By stationarity assumption, each $A_n^d$ has the same law. Therefore it is sufficient to show that $d~\mathbb{E}[\|A_1^d\|^2]\longrightarrow 0$. For the spherical symmetricity of the process $v^{\bouncy,d}$, we have
\begin{align*}
	d\mathbb{E}[\|A_1^d\|^2]
&=d\mathbb{E}\left[\left\|\int_0^{\sigma_1}\pi_k(v_t^{\bouncy,d}-v_0^{\bouncy,d})\dif t\right\|^2\right]=k\mathbb{E}\left[\left\|\int_0^{\sigma_1}v_t^{\bouncy,d}-v_0^{\bouncy,d}\dif t\right\|^2\right]. 
\end{align*}
Since the stopping time $\sigma_1$ is independent from $\mathcal{F}_{\sigma_1-}^d$, by the dominated convergence theorem, it is sufficient to show that 
$\mathbb{E}[\|v_t^{\bouncy,d}-v_0^{\bouncy,d}\|^2]\longrightarrow 0$ for any $t>0$, where $v_t^{\bouncy, d}$ follows the stochastic differential equation defined in Section \ref{subsubsec:BPS} without refreshment jumps. 
We have
\begin{align*}
v_t^{\bouncy,d}-v_0^{\bouncy,d}=\int_{(0,t]\times\mathbb{R}_+}\psi(u,z)N(\dif u,\dif z),\ \psi(t,z):=-2S_{t-}^{\bouncy,d}\frac{\xi_{t-}^{\bouncy, d}}{\|\xi_{t-}^{\bouncy,d}\|^2}1_{\{z\le S_{t-}^{\bouncy,d}\}}. 
\end{align*}
Therefore by Theorem II.1.33 of \cite{JS}, 
\begin{align*}
\mathbb{E}[\|v_t^{\bouncy, d}-v_0^{\bouncy, d}\|^2]&=\mathbb{E}\left[
\int_{(0,t]\times\mathbb{R}_+}\|\psi(u,z)\|^2\dif u\dif z+\left\|\int_{(0,t]\times\mathbb{R}_+}\psi(u,z)\dif u\dif z\right\|^2\right]. 
\end{align*}
By stationarity of the process together with Fubini's theorem, we have a bound
\begin{align*}
\mathbb{E}[\|v_t^{\bouncy, d}-v_0^{\bouncy, d}\|^2]
\le 4\mathbb{E}\left[t\{(S_0^{\bouncy,d})^+\}^3\frac{1}{\|\xi_0^{\bouncy,d}\|^2}+t^2\{(S_0^{\bouncy,d})^+\}^4\frac{1}{\|\xi_0^{\bouncy,d}\|^2}\right]. 
\end{align*}
The variable $S_0^{\bouncy,d}$ follows the standard normal distribution, and $\|\xi_0^{\bouncy,d}\|^{-2}$ follows the inverse of the chi-squared distribution with $d$ degrees of freedom which is on the order of $d^{-1}$ by Lemma 4.1 of \cite{arXiv:1412.6231}.  
Thus, the expectation in the above has on the order of $d^{-1}$ by the Cauchy-Schwarz inequality. 
Thus $\mathbb{E}[\|A_{i+1}^d\|^2]$ is on the order of $d^{-2}$ which proves the claim. 
\end{proof}

\begin{cor}\label{cor:component-1}
$$
\sup_{n=1,\ldots, dN}\left\| \sum_{i=0}^{n-1}A_{i+1}^d-\mathbb{E}[A_{i+1}^d|\mathcal{F}_{\sigma_i-}^d]\right\|~\longrightarrow~_{d\rightarrow\infty}~0
$$
in probability. 
\end{cor}

\begin{proof}
	Consider a filtration $(\mathcal{F}_{\sigma_n-}^d)_n$. A discrete process $X$ is L-dominated by $Y$ in the sense of I.3.29 of \cite{JS}, that is, 
$\mathbb{E}[|X_\tau|]\le \mathbb{E}[|Y_\tau|]$ for any bounded $(\mathcal{F}_{\sigma_n-}^d)_n$-stopping time $\tau$ where
\begin{align*}
X_n:=\left\| \sum_{i=0}^{n-1}A_{i+1}^d-\mathbb{E}[A_{i+1}^d|\mathcal{F}_{\sigma_i-}^d]\right\|^2,\ Y_n:= \sum_{i=0}^{n-1}\left\|A_{i+1}^d\right\|^2. 
\end{align*}
Then, by Lenglart's inequality (I.3.30 of \cite{JS}), we have
\[
\mathbb{P}(\sup_{n\le dN}X_n\ge \epsilon)\le \frac{\eta}{\epsilon}+\mathbb{P}(Y_{dN}\ge \eta)
\]
for $\epsilon, \eta>0$. Therefore, the convergence of $\sup_{n\le dN} X_n$ comes from Lemma \ref{lem:component-1}. 
\end{proof}

Now we show that a predictable component 
$\mathbb{E}[A_{n+1}^d|\mathcal{F}_{\sigma_n-}^d]$ has a simpler expression $\mathbb{E}[B_{n+1}^d|\mathcal{F}_{\sigma_n-}^d]$ where 
$$
B_{n+1}^d=-\int_{\sigma_n}^{\sigma_{n+1}}\int_0^{t}\frac{1}{d}\pi_k(\xi_{\sigma_n}^{\bouncy,d})~\dif s\dif t
$$
up to negligible term. Note that
$$
\mathbb{E}[B_{n+1}^d|\mathcal{F}_{\sigma_n-}^d]=-\frac{\rho^{-2}}{d}\pi_k(\xi_{\sigma_n}^{\bouncy,d}). 
$$

\begin{lem}\label{lem:component-2}
$$
\sup_{n=1,\ldots, dN}\left\| \sum_{i=0}^{n-1}\mathbb{E}[A_{i+1}^d|\mathcal{F}_{\sigma_i-}^d]-\mathbb{E}[B_{i+1}^d|\mathcal{F}_{\sigma_i-}^d]\right\|~\longrightarrow~_{d\rightarrow\infty}~0
$$
in probability. 
\end{lem}

\begin{proof}
By stationarity of the process together with the Cauchy-Schwarz inequality, it is sufficient to show 
$d~\mathbb{E}[\|\mathbb{E}[A_1^d-B_1^d|\mathcal{F}_{0-}^d]\|^2]^{1/2}\longrightarrow 0$. 
By spherical symmetricity of the processes, we have
\begin{align*}
\mathbb{E}[\|\mathbb{E}[A_1^d-B_1^d|\mathcal{F}_{0-}^d]\|^2]
&=\frac{k}{d}\mathbb{E}\left[\left.\left\|\mathbb{E}\left[\int_0^{\sigma_1}\int_{(0,t]\times\mathbb{R}_+}\psi(s,z)N(\dif s,\dif z)\dif t\right|\mathcal{F}_{0-}^d\right]
\right\|^2\right]\\
&=\frac{k}{d}\mathbb{E}\left[\left.\left\|\mathbb{E}\left[\int_0^{\sigma_1}\int_0^t\psi(s)\dif s\dif t\right|\mathcal{F}_{0-}^d\right]
\right\|^2\right]
\end{align*}
where
\begin{align*}
\psi(t,z)&=-2S_{t-}^{\bouncy,d}\frac{\xi_{t-}^{\bouncy,d}}{\|\xi_{t-}^{\bouncy,d}\|^2}
1_{\{z\le S_{t-}^{\bouncy,d}\}}+\frac{\xi_{t-}^{\bouncy, d}}{d},\\ 
\psi(t)&=-2\left\{(S_{t}^{\bouncy,d})^+\right\}^2\frac{\xi_{t}^{\bouncy,d}}{\|\xi_{t}^{\bouncy,d}\|^2}
+\frac{\xi_{t}^{\bouncy, d}}{d}
\end{align*}
By the Cauchy-Schwarz inequality together with dominated convergence theorem, it is sufficient to prove 
$d\mathbb{E}[\|\mathbb{E}[\int_0^t\psi(s)\dif s|\mathcal{F}_{0-}^{\bouncy, d}]\|^2]\longrightarrow 0$ where $\xi_t^{\bouncy, d}$ and $v_t^{\bouncy, d}$ follow the stochastic differential equation defined in Section \ref{subsubsec:BPS} without refreshment jumps. Let
\begin{align*}
	\psi_1(t)&=-2\left\{(S_{t}^{\bouncy,d})^+\right\}^2\left\{\frac{\xi_{t}^{\bouncy,d}}{\|\xi_{t}^{\bouncy,d}\|^2}-\frac{\xi_{t}^{\bouncy,d}}{d}\right\},\\
\psi_2(t)&=-\left\{2\left\{(S_{t}^{\bouncy,d})^+\right\}^2-1\right\}\frac{\xi_{t}^{\bouncy,d}-\xi_0^{\bouncy, d}}{d},\\ 
\psi_3(t)&=-\left\{2\left\{(S_{t}^{\bouncy,d})^+\right\}^2-1\right\}\frac{\xi_{0}^{\bouncy,d}}{d},\ 
\end{align*}
so that $\psi(t)=\psi_1(t)+\psi_2(t)+\psi_3(t)$. 
Convergence of $d\mathbb{E}[\|\psi_1(t)\|^2]=d\mathbb{E}[\|\psi_1(0)\|^2]$
 follows from the Cauchy-Schwartz inequality as in the proof of Lemma \ref{lem:component-1}. Convergence of $d\mathbb{E}[\|\psi_1(t)\|^2]$
also follows by the Cauchy-Schwartz inequality together with 
the uniform bound $\|\xi_t^{\bouncy, d}-\xi_0^{\bouncy, d}\|\le t$.  
 Therefore the proof will be completed if we show $d\mathbb{E}[\|\mathbb{E}[\int_0^t\psi_3(s)\dif s|\mathcal{F}_{0-}^{\bouncy, d}]\|^2]\longrightarrow 0$

By (\ref{eq:tildestd}), up to the refreshment time, we have
\[
S_t^{\bouncy,d}=S_0^{\bouncy,d}+t-2\int_{(0,t]\times\mathbb{R}}S_{s-}^{\bouncy,d}1_{\{z\le S_{s-}^{\bouncy, d}\}}N(\dif s,\dif z). 
\]
By this fact, 
\begin{align*}
	\mathbb{E}\left[\int_0^t\psi_3(s)\dif s|\mathcal{F}_{0-}^{\bouncy, d}\right]&=
	\mathbb{E}\left[\left.t-2\int_{(0,t]\times\mathbb{R}}S_{s-}^{\bouncy,d}1_{\{z\le S_{s-}^{\bouncy, d}\}}N(\dif s,\dif z)\right|\mathcal{F}_{0-}^d\right]\frac{\xi_0^{\bouncy,d}}{d}\\
&=
	\mathbb{E}\left[\left.S_t^{\bouncy, d}-S_0^{\bouncy, d}\right|\mathcal{F}_{0-}^d\right]\frac{\xi_0^{\bouncy,d}}{d}\\
&=
	\mathbb{E}\left[\left.h_t(S_0^{\bouncy, d})-S_0^{\bouncy, d}\right|\mathcal{F}_{0-}^d\right]\frac{\xi_0^{\bouncy,d}}{d}
\end{align*}
where $h_t(x)=\mathbb{E}[\mathcal{T}_t|\mathcal{T}_0=x]$.  
Let $\mathcal{L}_{0-}(X)$ be the  distribution of $X$ conditioned on $\mathcal{F}_{0-}^d$. Since the initial velocity 
is independent from the initial state, 
we have $\mathcal{L}_{0-}(S_0^{\bouncy, d})=\mathcal{L}_{0-}(\alpha^d U^d)$ 
as in (\ref{eq:law_of_s}) 
where $U^d$ is defined in (\ref{eq:ud}) and 
 $(\alpha^d)^2=\|\xi_{0}^{\bouncy, d}\|^2/d$. 
In particular, $\mathbb{E}[S_0^{\bouncy, d}|\mathcal{F}_{0-}^d]=0$. 
Moreover, by (\ref{eq:hw_bound}), we can substitute 
$h_t(S_0^{\bouncy,d})$ by $h_t(\alpha^dW)$
 where $W$ follows the standard normal distribution. Finally the claim follows by the dominated convergence theorem since 
 $\alpha^d\rightarrow 1$ and $\mathbb{E}[h_t(W)|\mathcal{F}_{0-}^d]=0$. Therefore, $d\mathbb{E}[\|\mathbb{E}[\int_0^t\psi_3(s)\dif s|\mathcal{F}_{0-}^{\bouncy, d}]\|^2]\longrightarrow 0$ which proves the claim. 
\end{proof}

\begin{lem}\label{lem:component-3}
$$
\sup_{n=1,\ldots, dN}\left\| \int_0^{\sigma_n/d}b'(\overline{Z}_t^{\bouncy,d,k})\dif t-\sum_{i=0}^{n-1}\mathbb{E}[B_{i+1}^d|\mathcal{F}_{\sigma_i-}^d]\right\|~\longrightarrow~_{d\rightarrow\infty}~0
$$
in probability, where $b'(x)=-\rho^{-1} x$. 
\end{lem}

\begin{proof}
Since the difference in the norm is 
\begin{align*}
	-\sum_{i=0}^{N-1}(\sigma_{i+1}-\sigma_i-\rho^{-1})\frac{\rho^{-1}}{d}\pi_k(\xi_{\sigma_i}^{\bouncy,d})
\end{align*}
and it is a martingale. Therefore the claim follows from Doob's inequality (I.1.43 of \cite{JS}). 
\end{proof}

Since $\overline{Z}^{\bouncy,d,k}$ is a pure step process, the semimartingale characteristics are entirely described by a random measure as described in Theorem II.3.11(b) of \cite{JS} (See also Proposition II.2.17). 
Therefore, we have the first and modified second characteristics as follows:
\begin{align*}
	B'^d_T&=\sum_{n:\sigma_n\le T}\mathbb{E}[\Delta\overline{Z}_{\sigma_{n}/d}^{\bouncy,d,k}|\mathcal{F}_{\sigma_{n-1}-}],\\ 
	\tilde{C}'^d_T&=\sum_{n:\sigma_n\le T}\mathbb{E}[(\Delta\overline{Z}_{\sigma_{n}/d}^{\bouncy,d,k})^{\otimes 2}|\mathcal{F}_{\sigma_{n-1}-}]-\mathbb{E}[(\Delta\overline{Z}_{\sigma_{n}/d}^{\bouncy,d,k})|\mathcal{F}_{\sigma_{n-1}-}]^{\otimes 2}. 
\end{align*}
Also the corresponding random measure is
\begin{align*}
g*\nu_T^d=\sum_{n:\sigma_n\le T}\mathbb{E}[g(\Delta\overline{Z}_{\sigma_{n}/d}^{\bouncy,d,k})|\mathcal{F}_{\sigma_{n-1}-}].
\end{align*}
for a bounded smooth function $g(x)$. 
Here, for a vector $v=(v_1,\ldots, v_k)\in\mathbb{R}^d$, 
$v^{\otimes 2}$ is a $k\times k$ matrix with $(i,j)$-th element $v_iv_j$. 

\begin{lem}
		The process $\overline{Z}^{\bouncy,d,k}$ converges in law to $Z^{\bouncy,k}$. 
\end{lem}

\begin{proof}
The first and the modified second characteristics of $Z^{\bouncy, k}$ are 
$$
B'_T=\int_0^Tb'(Z^{\bouncy, k}_t)\dif t,\ 
\tilde{C}'_T=2T\rho^{-1}. 
$$
	We apply Theorem IX.3.48 of \cite{JS}. Conditions (i-iv) is obvious since the limit is the Ornstein-Uhlenbeck process. The condition (v) is also clear since in this case, both $\eta^d$ and $\eta$ are the $k$-dimensional standard normal distribution. Therefore we only need to check four conditions in (vi). 

Firstly we can assume that the number of refreshment jumps until time $T$, $R_d((0,T]\times\mathbb{R})$ is smaller than $dN$ for some $N\in\mathbb{N}$ by the argument of the proof of Theorem \ref{theo:bps_ll}. 
Let $\nu^d$ be the random measure corresponding to $\overline{Z}^{\bouncy,d,k}$.  
For $g\in C_1(\mathbb{R})$ (See VII.2.7 of \cite{JS}), we can assume that $|g(x)|\le 1$ for any $x$ and $g(x)=0$ for $|x|<b$ for some $b>0$.  Then 
\begin{align*}
g*\nu_T^d\le \sum_{n:\sigma_n\le T}\mathbb{P}(\|\Delta\overline{Z}_{\sigma_{n}/d}^{\bouncy,d,k}\|>b|\mathcal{F}_{\sigma_{n-1}-}).
\end{align*}
Therefore, it is sufficient to prove
$$
\sum_{i=1}^{dN}\mathbb{P}(\|\Delta\overline{Z}_{\sigma_{n}/d}^{\bouncy,d,k}\|>b)~\longrightarrow_{d\rightarrow\infty}~0. 
$$
for [$\delta_{\mathrm{loc}}$-D]. This is also a sufficient condition for 3.49 of Theorem IX.3.48. By equation (\ref{eq:doob_meyer}), we have 
$$
\|\Delta\overline{Z}_{\sigma_{n}/d}^{\bouncy,d,k}\|\le \|M_n^d\|+\|A_n^d\|. 
$$
The convergence of $A_n^d$ part directly follows from Lemma \ref{lem:component-1} with Chevyshev's inequality, and the convergence of $M_n^d$ part follows from Markov's inequality 
together with the fact that the square of each component of $v_{\sigma_n}^{\bouncy,d}$ follows the Beta distribution with parameter $1/2$ and $(d-1)/2$. Condition [Sup-$\beta'_{\mathrm{loc}}$] follows by Corollary \ref{cor:component-1}-\ref{lem:component-3}. Finally we check [$\gamma_{\mathrm{loc}}'$-D]. 

By the decomposition of $\Delta\overline{Z}^{\bouncy,d,k}$, we have
\begin{align*}
	\tilde{C}_T'^d&=\sum_{n;\sigma_n\le T}\mathbb{E}[(M_n^d)^{\otimes 2}+M_n^d\otimes (A_n^d-\mathbb{E}[A_n^d|\mathcal{F}_{\sigma_{n-1}-}])\\
	&\quad + (A_n^d-\mathbb{E}[A_n^d|\mathcal{F}_{\sigma_{n-1}-}])\otimes M_n^d+(A_n^d-\mathbb{E}[A_n^d|\mathcal{F}_{\sigma_{n-1}-}])^{\otimes 2}|\mathcal{F}_{\sigma_{n-1}-}]. 
\end{align*}
The first term is 
\begin{align*}
\mathbb{E}[(M_n^d)^{\otimes 2}|\mathcal{F}_{\sigma_{n-1}-}]=	&
\mathbb{E}[(\sigma_{n}-\sigma_{n-1})^2\pi_k(v_{\sigma_{n-1}}^{\bouncy,d})^{\otimes 2}|\mathcal{F}_{\sigma_{n-1}-}]=2\rho^{-2}d^{-1}I_k. 
\end{align*}
From this fact together with Lemma \ref{lem:component-1}, the other term converges to $0$. By the same argument as Lemma \ref{lem:component-3}, the claim follows. 
\end{proof}

\begin{lem}
	The process $Z^{\bouncy,d,k}$ converges in law to $Z^{\bouncy,k}$.
\end{lem}

\begin{proof}
By Lemma VI.3.31 of \cite{JS}, it is sufficient to show 
	$$
\epsilon^d_T:=\sup_{0\le t\le T}\|Z_t^{\bouncy,d,k}-\overline{Z}_t^{\bouncy,d,k}\|
\longrightarrow_{d\rightarrow \infty} ~0
	$$
in probability. 
Let $A_t$ be as in the proof of Theorem \ref{theo:bps_ll}. Then we have
\begin{align*}
	\epsilon^d_T\le \sup_{0\le j\le R_d(A_{dT})}\sup_{\sigma_j\le t<\sigma_{j+1}}\|\pi_k(\xi_t^{\bouncy,d})-\pi_k(\xi_{\sigma_j}^{\bouncy,d})\|. 
\end{align*}
Therefore, for $J\in\mathbb{N}$, 
\begin{align*}
	\mathbb{P}(\epsilon^d_t>\epsilon)& \le \mathbb{P}(R_d(A_{dT})>dJ)\\
	&\quad+\mathbb{P}\left(\sup_{0\le j\le dJ}\sup_{\sigma_j\le t<\sigma_{j+1}}\|\pi_k(\xi_t^{\bouncy,d})-\pi_k(\xi_{\sigma_j}^{\bouncy,d})\|>\epsilon\right) \\
& \le \mathbb{P}(R_d(A_{dT})>dJ)+dJ\mathbb{P}\left(\sup_{0\le t<\sigma_1}\|\pi_k(\xi_t^{\bouncy,d})-\pi_k(\xi_{0}^{\bouncy,d})\|>\epsilon\right). 
\end{align*}
On the other hand, 
$$
\|\pi_k(\xi_s^{\bouncy,d})-\pi_k(\xi_0^{\bouncy ,d})\|\le \int_0^s\|\pi_k(v_u^{\bouncy,d})\|\dif u
$$
and the forth moment of the norm is on the order of $d^{-2}$. Thus by Markov's inequality, $\epsilon_T^d$ is negligible.  
\end{proof}

\subsubsection{Proof of Theorem \ref{theo:bps_coordinate}}

\begin{proof}[Proof of Theorem \ref{theo:bps_coordinate}]
Weak convergence of $Z^{\bouncy, d,k}$ has been proved. 
Therefore, the proof of Theorem \ref{theo:bps_coordinate} will be completed if we can show the law of large numbers (\ref{eq:BPS_LLN}). The proof is essentially the same as that of Lemma B.4 of \cite{arXiv:1412.6231}. 

Let $\|f\|_\infty=\sup_{x\in\mathbb{R}^k}|f(x)|$. Without loss of generality, we can assume 
$\int f(x)\phi_k(x)\dif x=0$. It is sufficient to show that 
$$
I_{d,T}:=\mathbb{E}\left[\left|\frac{1}{T}\int_0^Tf(Z_t^{\bouncy,d,k})\dif t\right|\right]\longrightarrow_{d,T\rightarrow\infty} 0. 
$$
Since the limiting process is the ergodic Ornstein-Uhlenbeck process, for any $\epsilon>0$ we can find $T_0>0$ so that 
$$
I_{T_0}=\mathbb{E}\left[\left|\frac{1}{T_0}\int_0^{T_0}f(Z_t^{\bouncy,k})\dif t\right|\right]<\epsilon
$$
by the law of large numbers. 
By dividing the interval $[0,T]$ into shorter intervals with length $T_0$, we have 
\begin{align*}
	I_{d,T}&=\mathbb{E}\left[\left|\frac{1}{T}\sum_{k=0}^{[T/T_0]-1}\int_{kT_0}^{(k+1)T_0}f(Z_t^{\bouncy,d,k})\dif t+\frac{1}{T}\int_{T_0[T/T_0]}^Tf(Z_t^{\bouncy,d,k})\dif t\right|\right]\\
& \le \frac{T_0}{T}\sum_{k=0}^{[T/T_0]-1}\mathbb{E}\left[\left|\frac{1}{T_0}\int_{kT_0}^{(k+1)T_0}f(Z_t^{\bouncy,d,k})\dif t\right|\right]\\
&\quad	+\frac{1}{T}\int_{T_0[T/T_0]}^T\mathbb{E}\left[\left|f(Z_t^{\bouncy,d,k})\right|\right]	\dif t. 
\end{align*}
Then by stationarity of the process $Z^{\bouncy,d,k}$ together with the weak convergence of $Z^{\bouncy,d,k}$, we have
\begin{align*}
	I_{d,T}& \le \frac{T_0}{T}\left[\frac{T}{T_0}\right]I_{d, T_0}+\frac{T-T_0[T/T_0]}{T}\|f\|_\infty
	\longrightarrow_{d,T\rightarrow\infty} I_{T_0}\le \epsilon,
\end{align*}
which completes the proof. 
\end{proof}

\section{Ergodic properties of the limiting processes}\label{asec:ergodicity}

\subsection{Proof of Theorem \ref{theo:ergodicity_tildeS}}

\begin{proof}[Proof of Theorem \ref{theo:ergodicity_tildeS}]
Construct $\mathcal{T}$ and $S^{\bouncy}$ as in Section \ref{subsec:ergodiclimit}.
We also set $\sigma_0=0$ and $W_0=S_0^\bouncy$. 
Firstly, we prove irreducibility and  aperiodicity  of the Markov process. For $K>0$, 
let $\nu_K$ be the Lebesgue measure restricted to(0 $[-K,K]$. 
Consider an event
$$
B_T=\left\{\omega\in\Omega:R((0,T]\times\mathbb{R})=1, N(C_T)=0\right\}
$$
where $C_T=(0,T]\times[0,|x|+|W_1|+T]$. 
On the event, since $R((0,T]\times\mathbb{R})=1$ there is a single refreshment jump $\sigma_1$ until $T>0$. Recall that 
in each interval $[\sigma_i,\sigma_{i+1})$, the process $S^{\bouncy}$ has the same behavior as that of $\mathcal{T}$ with $\mathcal{T}_{\sigma_i}=W_i$. 
Therefore, by (\ref{eq:square_t}), we have
$$
\omega\in B_T~\Longrightarrow~ |S_t^{\bouncy}|\leq \left\{
\begin{array}{ll}
|x|+ t &\mathrm{if}\ t<\sigma_1\\
W_1+ t &\mathrm{if}\ \sigma_1\le t\le T\\	
\end{array}
\right.
\Longrightarrow
\sup_{t\le T}|S_t^{\bouncy}|\le |x|+|W_1|+T.
$$
Therefore, on the event $B_T$, the process $S^\bouncy$ only jumps at the refreshment time $t=\sigma_1$ until $T$, since the number of jumps due to $N$ up to time $T$ is 
$$
\int_{(0,T]\times\mathbb{R}_+} 1_{\{z\le S_{s-}^{\bouncy}\}}N(\dif s,\dif z)
\le N(C_T)=0. 
$$
Therefore, except for the refreshment jump time $\sigma_1$, $S^\bouncy$ moves deterministically, and hence 
$$
\omega\in B_T~\Longrightarrow~ S_t^{\bouncy}=\left\{
\begin{array}{ll}
x+t&\mathrm{if}\ t<\sigma_1\\
W_1+(t-\sigma_1)&\mathrm{if}\ \sigma_1\le t\le T. 
\end{array}
\right.
$$
Now we calculate the probability of the event $B_T$. 
Since $R_d$ and $N$ are independent
\begin{align*}
	\mathbb{P}_x(B_T)& =
	\mathbb{P}_x(R((0,T]\times\mathbb{R})=1)\times \mathbb{P}(N(C_T)=0)\\
	& =	
	\left\{\rho T~e^{-\rho T}\right\}\times \left\{\int_\mathbb{R}e^{-(|x|+|y|+T)T}\phi(y)\dif y\right\}
	\\
	& \ge 
	\left\{\rho T~e^{-\rho T}\right\}\times\left\{
c_T~e^{-(K+T)T}\right\}
\end{align*}
where $c_T=\int \exp(-T|y|)\phi(y)\dif y$. 
On the other hand, 
\begin{align*}P_T(x, A)
	& \ge\mathbb{P}_x(S^{\bouncy}_T\in A, B_T)\\
	& =\mathbb{P}_x(W_1+(T-\sigma_1)\in A, B_T)\\
	& =\mathbb{E}_x\left[\int_A\phi(y-(T-\sigma_1))\dif y, B_T\right]\\
	& \ge \inf_{0 \leq s \leq T} \int_{A \cap K} \phi(y - (T-s)) \, d y ~\mathbb{P}_x(B_T) \\
	& \ge \kappa_T \nu_K(A) ~\mathbb{P}_x(B_T),
\end{align*}
where $\kappa_T =  \inf_{0 \leq s \leq T} \inf_{y \in K} \phi(y-(T-s))$. 
By these estimates, we obtain 
$$
P_T(x,A)\ge \kappa_T \left\{\rho T e^{-\rho T}\right\}\times\left\{
c_T~e^{-(K+T)T}\right\}\nu_K(A)\quad  \text{for} \quad x \in [-K,K]. 
$$
Thus, the Markov process is $\nu_K$-irreducible and aperiodic, and any compact set is a small set. 

Secondly, we prove $V$-uniform ergodicity. We need to check 
\begin{equation}\label{eq:drift_criterion}
	HV(x)\le -\gamma V(x)+b1_C
\end{equation}
for some $\gamma, b>0$, a small set $C$, and a drift function $V:\mathbb{R}\rightarrow [1,\infty)$ where $H$ is defined in (\ref{eq:BPSgenerator}). However, by taking $V(x)=1+x^2$, we have
$$
\frac{HV(x)}{V(x)}=\frac{2x+\rho(1-x^2)}{1+x^2}~\longrightarrow_{|x|\rightarrow\infty}~-\rho. 
$$
Thus, the drift condition is satisfied for $C=[-R,R]$ and $\gamma=\rho/2$ when $R$ is sufficiently large. Thus $V$-uniform ergodicity follows by Theorem 5.2 of \cite{MR1379163}. 
\end{proof}

\subsection{Proof of Theorem \ref{theo:ergodicity_t}}

\begin{proof}[Proof of Theorem \ref{theo:ergodicity_t}]
Let $K > 0$ and consider $x \in [-K, K]$. Let $T = 2 K + 1$, and define
$$
B_T= \left\{\omega\in\Omega: N(\omega;C_T)=N(\omega;D_T)=1\right\}
$$
where 
$$
C_T=(0,T]\times [0,|x|+T],\ D_T=[(1-x)^+,T]\times [0,1].
$$ 
On the event $B_T$, 
the number of jumps until time $T$ is 
$$
\int_{(0,T]\times\mathbb{R}_+} 1_{\{z\le\mathcal{T}_{s-}\}}N(\dif s,\dif z)
\le N(C_T)=1
$$
since $|\mathcal{T}_t|\le |x|+T\ (0\le t\le T)$ by (\ref{eq:square_t}). 
Thus the number of jumps is at most $1$. 
On the other hand, if there is no jump,  
then $\mathcal{T}_t=x+t\ (0\le t\le T)$. 
However, since $(1-x)^+\le t\Longrightarrow 1\le x+t= \mathcal{T}_{t-}$ we have
$$
\int_{(0,T]\times\mathbb{R}_+} 1_{\{z\le\mathcal{T}_{s-}\}}N(\dif s,\dif z)
\ge N(D_T)=1. 
$$
Therefore, there is a single jump until time $T$. 
Then, on the event $B_T$, we have 
\begin{align*}
\mathcal{T}_t=\left\{
\begin{array}{ll}
	x+t&\mathrm{if}~ t<\tau_1\\
	-(x+\tau_1^{(1)})+(t-\tau_1^{(1)})&\mathrm{if}~\tau_1\le t\le T,  
\end{array}
\right.		
\end{align*}
and hence 
\begin{align*}
		P_T(x,A)\ge&~\mathbb{P}_x(\mathcal{T}_T\in A, B_T)\\
		&=~\mathbb{P}_x(-(x+\tau_1^{(1)})+(T-\tau_1^{(1)})\in A, B_T)\\
		&=~\mathbb{P}_x\left(-(x+\tau_1^{(1)})+(T-\tau_1^{(1)})\in A|B_T\right)\times\mathbb{P}_x(B_T). 
\end{align*} 
We have
\begin{align*}
&\mathbb{P}_x\left(-(x+\tau_1^{(1)})+(T-\tau_1^{(1)})\in A|B_T\right)\\
&=~\int_{(1-x)^+}^T1_A(-(x+s)+(T-s))\frac{\dif s}{T-(1-x)^+}\\
	&\ge~T^{-1}\mathrm{Leb}(A\cap [-x-T,T-x-2(1-x)^+])
\end{align*}
where $\mathrm{Leb}$ is the Lebesgue measure. On the other hand, 
\begin{align*}
	\mathbb{P}_x(B_T)&=~\mathbb{P}_x(N(D_T)=1)\times \mathbb{P}(N(C_T\cap D_T^c)=0)\\
	&=~(T-(1-x)^+)e^{-(T-(1-x)^+)}\times e^{-(T(|x|+T)-(T-(1-x)^+))}\\
	&=:~c(T,x). 
\end{align*}
Since $c(T,x)>0\ (x\in\mathbb{R})$, the Markov process is $\mathrm{Leb}$-irreducible
and aperiodic
since we have $P_T(x,A)>0$  by taking $T>0$ sufficiently large. 
Also, by  
$c_T:=\inf_{x\in [-K,K]}c(T,x)>0$ we have
$$
P_T(x,A)\ge c_T~T^{-1}~\mathrm{Leb}(A\cap[K-T,T-K-2(1+K)^+])\ (x\in [-K,K]). 
$$
Thus any compact set is a small set. 

Finally, we prove $V$-uniform ergodicity. We need to check the drift criterion
(\ref{eq:drift_criterion})
for  $\gamma>0$, a small set $C$ and $V:E\rightarrow [1,\infty)$ and $G$ defined in (\ref{eq:generator_t}) in place of $H$. 
Construct a continuously differentiable function $V:E\rightarrow [1,\infty)$ so that 
\begin{equation}\label{eq:drift_for_T}
V(x)=
\left\{
\begin{array}{ll}
2\exp(x)&x>4\\
\exp(-x)&x\le 0. 
\end{array}
\right.
\end{equation}
Then $GV(x)=(2-x)e^x\le -V(x)$ for $x>4$, and $GV(x)=-V(x)$ for $x<0$. 
Thus the drift condition holds with $V(x)$, $C=[0,4]$ and $\gamma=1$. Thus the claim follows by Theorem 5.2 of \cite{MR1379163}. 
\end{proof}

\subsection{Proof of Proposition \ref{prop:covariance}}

By $V$-uniform ergodicity of the Markov process $\mathcal{T}$, for $s\le t$ and $k\in\mathbb{N}$, we have
\begin{equation}\label{eq:convergence_of_h}
	\left|\mathbb{E}[\mathcal{T}_t^k|\mathcal{T}_s=x]-\int y^k\phi(y)\dif y\right|\le C_k\gamma^{t-s}V(x)
\end{equation}
for some $C_k>0$, $\gamma\in (0,1)$ and hence the covariance function has exponential decay property 
$$
|K(s,t)|=\left|\mathbb{E}\left[\mathcal{T}_s\mathbb{E}\left[\mathcal{T}_t|\mathcal{T}_s\right]\right]\right|\le C_1\gamma^{t-s}\mathbb{E}[|\mathcal{T}_s|V(\mathcal{T}_s)]=C\gamma^{t-s}
$$
for some $C>0$ since the marginal distribution of $\mathcal T$ is the standard normal distribution and using the explicit form of $V$ given by~\eqref{eq:drift_for_T}.


\begin{proof}[Proof of Proposition \ref{prop:covariance}]
By (\ref{eq:convergence_of_h}) with $k=2$, we have
\begin{align*}
	0&=\lim_{t\rightarrow\infty}\mathbb{E}[(\mathcal{T}_t^2-\mathcal{T}_0^2)\mathcal{T}_0]\\
	&=
	\lim_{t\rightarrow\infty}\mathbb{E}\left[\left(\int_0^t2\mathcal{T}_s\dif s\right)\mathcal{T}_0\right]=
2\lim_{t\rightarrow\infty}\int_0^tK(s,0)\dif s=
2\int_0^\infty K(s,0)\dif s. 
\end{align*}
Hence we have (\ref{eq:integral_k}).

Next we calculate the derivatives of $
	K(t):=K(t,0)$. By It\^{o}'s formula together with the Lebesgue convergence theorem, we have
	\begin{align*}
		h^{-1}(K(t+h)-K(t))
		&=h^{-1}\mathbb{E}[(\mathcal{T}_{t+h}-\mathcal{T}_t)\mathcal{T}_0]\\
		&=h^{-1}\mathbb{E}\left[\int_0^h(1-2(\mathcal T_{t+s}^+)^2)\dif s~\mathcal{T}_0\right]\\
		&\longrightarrow_{h\rightarrow 0}\mathbb{E}\left[(1-2(\mathcal T_t^+)^2)\mathcal{T}_0\right]. 
	\end{align*}
	The first derivative at $t=0$ is 
	\begin{align*}
		K'(0)=\mathbb{E}\left[(1-2(\mathcal T_0^+)^2)\mathcal{T}_0\right]=-2\int_0^\infty x^3\phi(x)\dif x=-2 \sqrt{\frac{2}{\pi}}. 
	\end{align*}
	Similarly, the second derivative at $t=0$ is 
	\begin{align*}
		h^{-1}(K'(h)-K'(0))&=
		h^{-1}\mathbb{E}\left[\left\{(1-2(T_h^+)^2)\mathcal{T}_0\right\}-\left\{(1-2(T_0^+)^2)\mathcal{T}_0\right\}\right]\\
		&=-h^{-1}\mathbb{E}\left[\int_0^h(4\mathcal{T}_t^+-2(\mathcal{T}_t^+)^3)\dif s~\mathcal{T}_0\right]\\
		&\longrightarrow_{h\rightarrow 0}-\mathbb{E}\left[4(\mathcal T_0^+)^2-2(\mathcal T_0^+)^4\right]\\
		&=-\mathbb{E}\left[2T_0^2-T_0^4\right]=1. 
	\end{align*}
\end{proof}


\section{Non-Gaussian results}\label{sec:non-gauss}

First we show that the process $S^{\bouncy,d}$ converges to $S_{H,t}^\bouncy:=H^{1/2}S_{H^{1/2}t}^\bouncy(H^{-1/2}\rho)$. 
Let $B'^0(\rho), \tilde{C}'^0(\rho)$ and $\nu^0(\rho)$ be the first, modified second and third characteristics of $S^\bouncy(\rho)$ (See Section \ref{subsubsec:semimartingale}). Then the first and modified second characteristics of the process
$S_H^\bouncy$ is given by 
\begin{align*}
	B'_T=H^{1/2}B'^0_{H^{1/2}T}(H^{-1/2}\rho),\ 
	\tilde{C}'_T=H \tilde{C}'^0_{H^{1/2}T}(H^{-1/2}\rho)
\end{align*}
and the third characteristic is given by  
\begin{align*}
	g*\nu_T=g(H^{1/2}\ \cdot\ )*\nu_{H^{1/2}T}^0(H^{-1/2}\rho). 
\end{align*}
Therefore, by the change of variable formula, we have
\begin{align*}
B'_T&=HT-\int_0^T\{(S_{H,t}^\bouncy)^+\}^2\dif t-\rho\int_0^TS_{H,t}^\bouncy\dif t\\
\tilde{C}'_T&=4\int_0^T\{(S_{H,t}^\bouncy)^+\}^3\dif t+\rho\int_0^T\left(H+(S_{H,t}^\bouncy)^2\right)\dif t
\end{align*}
and 
\begin{align*}
g*\nu_T=\int_0^Tg(-2S_{H,t}^\bouncy)(S_{H,t}^\bouncy)^+\dif t+\rho\int_0^T\int_{\mathbb{R}}\left(g(u-S_{H,t}^\bouncy)\right)\phi_H(u)\dif u\dif t
\end{align*}
where $\phi_H$ is the probability density  function of $\mathcal{N}(0, H)$. On the other hand, the process $S^{\bouncy,d}$ satisfies
\begin{align*}
S^{\bouncy,d}_T&=S^{\bouncy,d}_0+\int_0^T\nabla^2\Psi^d(\xi_t^{\bouncy,d})[(v_t^{\bouncy,d})^{\otimes 2}]\dif t-2\int_{(0,T]\times\mathbb{R}_+} S^{\bouncy,d}_{t-} 1_{\{z\le S^{\bouncy,d}_{t-}\}}N(\dif t,\dif z)\\
&\quad+\rho\int_{(0,T]\times\mathfrak{S}^{d-1}}\left(\langle\nabla\Psi^d( \xi_{t-}^{\bouncy,d}),u\rangle-S^{\bouncy,d}_{t-}\right)R_d(\dif t,\dif u)\nonumber, 
\end{align*}
by It\^{o}'s formula. The first and modified second characteristics are 
\begin{align*}
B'^d_T&:=\int_0^T\nabla^2\Psi^d(\xi_t^{\bouncy,d})[(v_t^{\bouncy,d})^{\otimes 2}]\dif t-2\int_0^T\left\{(S^{\bouncy,d}_t)^+\right\}^2\dif t-\rho\int_0^TS^{\bouncy,d}_t\dif t,\\ 
\tilde{C}'^d_T&:=4\int_0^T\left\{(S^{\bouncy,d}_t)^+\right\}^3\dif t+\rho\int_0^T\left(\frac{\|\nabla\Psi^d(\xi_{t}^{\bouncy,d})\|^2}{d}+(S^{\bouncy,d}_t)^2\right)\dif t, 
\end{align*}
and the third characteristic is 
\begin{align*}
	g(x)*\nu^d_T:=\int g(x)\nu^d_T(\dif x)&:=\int_0^T g(-2S^{\bouncy,d}_t)(S^{\bouncy,d}_t)^+\dif t\\
	&\quad + \rho\int_0^T\int_{\mathfrak{S}^{d-1}} g\left(\langle\nabla\Psi^d(\xi_t^{\bouncy,d}),u\rangle-S^{\bouncy,d}_t\right)\dif t\psi_d(\dif u)
\end{align*}
for a continuous bounded function $g$. 
For the proof of Proposition \ref{prop:general_phi}, we will apply Theorem IX.3.48 \cite{JS} by showing convergences of the characteristics. To show the convergence of $\tilde{C}'^d$, we need the next lemma. 

\begin{lem}\label{lem:d1}
	For $T>0$, we have
	\begin{align}
\label{eq:Uniform_consistency_delta}	
\sup_{0\le t\le T}\left|\frac{\|\nabla\Psi^d(\xi_t^{\bouncy,d})\|^2}{d}-H\right|~\longrightarrow_{d\rightarrow\infty}~0. 
	\end{align}
\end{lem}

\begin{proof}
Let $X^d_T$ be the left-hand side of (\ref{eq:Uniform_consistency_delta}). 
By It\^o's formula, 
\begin{align*}
	\frac{\|\nabla\Psi^d(\xi_T^{\bouncy,d})\|^2}{d}-\frac{\|\nabla\Psi^d(\xi_0^{\bouncy,d})\|^2}{d}
=2d^{-1}\int_0^T\nabla^2\Psi^d(\xi_t^{\bouncy,d})[\nabla\Psi^d(\xi_t^{\bouncy,d}),v_t^{\bouncy,d}]\dif t. 
\end{align*}
Recall that by the Cauchy-Schwarz inequality, we have
\[M[a,b]\le (M[a^{\otimes 2}])^{1/2}(M[b^{\otimes 2}])^{1/2}\]
for vectors $a, b$ and positive definite matrix $M$. We apply this inequality 
for $M_t:=\nabla^2\Psi^d(\xi_t^{\bouncy,d})$, 
$a_t:=\nabla\Psi^d(\xi_t^{\bouncy,d})/\|\nabla\Psi^d(\xi_t^{\bouncy,d})\|$ and $b_t:=v_t^{\bouncy, d}$. 
Note that $M_t[a_t^{\otimes 2}]$ and $M_t[b_t^{\otimes 2}]$ are bounded above by $C$
by (\ref{eq:Uniform_bound_nabla}). 
Also, we have a bound
\[
\frac{\|\nabla\Psi^d(\xi_t^{\bouncy,d})\|^2}{d}\le X_t^d+H
\]
by triangle inequality. 
Therefore we have
\begin{align*}
	X_T^d
	&\le X_0^d+\sup_{0\le t\le T}\left|\frac{\|\nabla\Psi^d(\xi_t^{\bouncy,d})\|^2}{d}-\frac{\|\nabla\Psi^d(\xi_0^{\bouncy,d})\|^2}{d}\right|\\
	&\le X_0^d+2d^{-1}\sup_{0\le t\le T}\left|\int_0^tM_s[a_s,b_s]\times \|\nabla\Psi^d(\xi_t^{\bouncy,d})\|\dif s\right|\\
	&\le X_0^d+
	2CTd^{-1/2}(X^d_T+H)^{1/2}\le 
	X_0^d+
	2CTd^{-1/2}(1+X^d_T+H)
\end{align*}
where we used $a^{1/2}\le 1+a$ for $a>0$. Hence
$$
X_T^d\le (1-2CTd^{-1/2})^{-1}~(X_0^d+2CTd^{-1/2}(1+H))~\longrightarrow_{d\rightarrow\infty}~0
$$ 
in probability since $X_0^d~\longrightarrow~0$ in probability by (\ref{eq:Consistency_delta}). 
\end{proof}

Next we show the following lemma to prove the convergence of $B'^d$. 

\begin{lem}\label{lem:d2}
	For $T>0$, we have
\begin{align}
\label{eq:Uniform_consistency_nabla}
&\sup_{0\le t\le T}\left|[\nabla^2\Psi^d(\xi_t^{\bouncy,d})][(v_t^{\bouncy,d})^{\otimes 2}]\dif s-H\right|~\longrightarrow_{d\rightarrow\infty}~0 
\end{align}
\end{lem}

\begin{proof}
Let $N_R(t)$ and $N_B(t)$ be the number of refreshment jumps and that of bouncy jumps with respectively. 
Since $N_R(T)=R_d((0,T]\times\mathbb{R})$ follows the Poisson distribution with intensity $\rho~T$, it is $\mathbb{P}$-tight. Suppose that the interval $[s,t)$ does not include refreshment jump times. Then, by It\^o's formula, we have
\begin{equation}\label{eq:Uniform_absolute_S}
	\left||S_t^{\bouncy,d}|-|S_s^{\bouncy,d}|\right|=\left|\int_s^t\nabla^2\Psi^d(\xi_u^{\bouncy,d})[(v_u^{\bouncy,d})^{\otimes 2}]~\mathrm{sgn}(S_u^{\bouncy,d})~\dif u\right|\le CT. 
\end{equation}
Therefore, if $0=\sigma_0<\sigma_1<\ldots$ are the refreshment jump times, we have a bound
$$
\sup_{t\in [0,T]}|S_t^{\bouncy,d}|\le CT+\sup_{n=0,\ldots, N_R(T)}|S_{\sigma_n}^{\bouncy,d}|. 
$$
The right-hand side is $\mathbb{P}$-tight since $S_{\sigma_n}^{\bouncy,d}\ (n=1,2,\ldots)$ has the same law as that of $S_0^{\bouncy,d}$, and $N_R(T)$ is $\mathbb{P}$-tight. Thus 
$B^d_T:=\sup_{t\in [0,T]}|S_t^{\bouncy,d}|$ is $\mathbb{P}$-tight. By this fact, 
$$
N_B(T)=\int_{(0,T]\times\mathbb{R}}1_{\{z\le S_{t-}^{\bouncy,d}\}}N(\dif s, \dif z)\le N((0,T]\times[0, B^d_T])
$$
is also $\mathbb{P}$-tight. 

Let $X_t^d$ be the random variable in the absolute value in the left-hand side of (\ref{eq:Uniform_consistency_nabla}). For $\epsilon>0$, let $D_\epsilon=\{0=t_0<\ldots<t_N\}\subset [0,T]$ be a finite set that includes all refreshment jump times and 
$\max|t_i-t_{i-1}|<\epsilon$. 
If the interval $[s,t)$ does not include refreshment jump times, then 
\begin{align*}
X_t^d-X_s^d &=
\int_s^t\nabla^3\Psi^d(\xi_u^{\bouncy,d})[(v_u^{\bouncy,d})^{\otimes 3}]\dif u\\
&\quad-\int_{(s,t]\times\mathbb{R}}\nabla^2\Psi(\xi_{u-}^{\bouncy,d})[(\kappa^d(x_{u-}^{\bouncy,d}))^{\otimes 2}-(v_{u-}^{\bouncy,d})^{\otimes 2}]1_{\{z\le S_{u-}^{\bouncy,d}\}}N(\dif u,\dif z). 
\end{align*}
By (\ref{eq:Uniform_bound_nabla}), we have
\begin{align*}
	|X_t^d-X_s^d|\le |t-s|\left(C+2CN_B(T)\right). 
\end{align*}
Then we have 
\begin{align*}
	\sup_{0\le t\le T}|X_t^d|\le 
	\sup_{t\in D}|X_t^d|+\epsilon~\left(C+2CN_B(T)\right)
\end{align*}
and the first term in the right-hand side converges to $0$ by (\ref{eq:Consistency_nabla}) which proves the claim. 
\end{proof}

\begin{lem}
$S^{\bouncy,d}$ converges to $S_H^\bouncy$. 
\end{lem}

\begin{proof}
We apply Theorem IX.3.48 \cite{JS}. 
The proof follows the same line as that of Theorem \ref{theo:bps_limit} and conditions (i-iv) of Theorem IX.3.48 directly follows from the argument in the proof of Theorem \ref{theo:bps_limit}. The condition (v) follows from \eqref{eq:tv_convergence} with condition (\ref{eq:Consistency_delta}). Conditions [$\delta_{\mathrm{loc}}$-D] and 3.49 can be proved in the same line as that of Theorem \ref{theo:bps_limit}.
Finally, we need to check conditions [Sup-$\beta_{\mathrm{loc}}'$], [$\gamma'_{\mathrm{loc}}$-D] of (vi) which follow from Lemmas \ref{lem:d1} and \ref{lem:d2}. 
\end{proof}

\begin{proof}[Proof of Proposition \ref{prop:general_phi}]
	By stationarity, 
	\begin{align*}
		4\rho\mathbb{E}\left[(\Psi^d(\xi_{\sigma_1}^{\bouncy,d})-\Psi^d(\xi_{\sigma_0}^{\bouncy,d}))^2\right]&=
				4\rho\mathbb{E}\left[\left\{\int_0^{\sigma_1}S_t^{\bouncy,d}\dif t\right\}^2\right]\\
				&=4\rho\int_0^\infty\int_0^\infty\mathbb{E}\left[1_{\{s,t\le\sigma_1\}}S_s^{\bouncy,d}S_t^{\bouncy,d}\right]\dif s\dif t. 
	\end{align*}
	By (\ref{eq:Uniform_bound_nabla}) together with It\^o's formula for $S^{\bouncy,d}$, we have a uniform bound
	$$
	|S_t^{\bouncy,d}|\le |S_0^{\bouncy,d}|+Ct
	$$
	by (\ref{eq:Uniform_absolute_S}). Thus, for $s\le t$, 
	\begin{align*}
		\left|\mathbb{E}\left[1_{\{s,t\le\sigma_1\}}S_s^{\bouncy,d}S_t^{\bouncy,d}\right]\right|
		&\le 
		\mathbb{E}\left[1_{\{t\le\sigma_1\}}(|S_0^{\bouncy,d}|+Ct)^2\right]\\
&= \mathbb{P}(t\le\sigma_1)\mathbb{E}\left[(|S_0^{\bouncy,d}|+Ct)^2\right]\\
&\le  \mathbb{P}(t\le\sigma_1)~2\mathbb{E}\left[|S_0^{\bouncy,d}|^2+(Ct)^2\right]\\
&=e^{-\rho t}2\mathbb{E}\left[\frac{\|\nabla\Psi(\xi_0^{\bouncy,d})\|^2}{d}+(Ct)^2\right]. 
	\end{align*}
Therefore, by (\ref{eq:Uniform_bound_nabla}), this value is bounded above by $\exp(-\rho t)$ times a polynomial of $t$. Thus by the dominated convergence theorem, 
	\begin{align*}
		4\rho\mathbb{E}\left[(\Psi^d(\xi_{\sigma_1}^{\bouncy,d})-\Psi^d(\xi_{\sigma_0}^{\bouncy,d}))^2\right]~\longrightarrow_{d\rightarrow\infty}~4\rho\int_0^\infty\int_0^\infty\mathbb{E}\left[1_{\{s,t\le\sigma_1\}}S_{H,s}^{\bouncy}S_{H,t}^{\bouncy}\right]\dif s\dif t. 
	\end{align*}
Now we are going to substitute $S_{H,t}^{\bouncy}$ in the right hand side by $H^{1/2}S_{H^{1/2}t}^\bouncy$. 
For this substitution, the refreshment jump time $\sigma_1$ is also changed to $H^{-1/2}\sigma_1$. 
	Therefore, the right-hand side of the above equation equals to 
	\begin{align*}
&4\rho\int_0^\infty\int_0^\infty\mathbb{E}\left[1_{\{H^{1/2}~s,\ H^{1/2}~t~\le~\sigma_1\}}(H^{1/2}S_{H^{1/2}s}^{\bouncy})(H^{1/2}S_{H^{1/2}t}^{\bouncy})\right]\dif s\dif t\\
&=
4\rho\int_0^\infty\int_0^\infty\mathbb{E}\left[1_{\{s,t\le \sigma_1\}}S_s^{\bouncy}S_t^{\bouncy}\right]\dif s\dif t\\
&=
4\rho\int_0^\infty\int_0^\infty e^{-H^{-1/2}\rho~\max\{s,t\}}K(s-t,0)\dif s\dif t=H^{1/2}\sigma^2(H^{-1/2}\rho). 
	\end{align*}
\end{proof}

\section{Details for experiments}
\label{sec:experiments-details}

\subsection{Exact values for mean and variance for quantities of interest}
%
\label{sec:exactvalues}

\subsubsection{IID Gaussian}
\label{sec:iidgaussian}
$\Psi^d(\xi) = \sum_{i=1}^d (\xi_i)^2/2$.
In this case, we have
$\E_{\pi} [\Psi^d(\xi)] = d/2$, $\Var_{\pi} [\Psi^d(\xi)] = d/2$, $\E_{\pi} [\xi_i1]= 0$, $\Var_{\pi} [\xi_i1] = 1$.

\subsubsection{Correlated Gaussian}

$\xi \sim \mathcal N(0, \Sigma)$, where 
\[ \Sigma = \begin{pmatrix} 1 & \rho & \hdots & \rho \\
\rho & \ddots & \ddots  & \vdots \\
\vdots & \ddots & \ddots& \rho \\
\rho &  \hdots & \rho & 1
\end{pmatrix},
\]
so that $\Psi^d(\xi) = \tfrac 1 2 x^{\top} V x$,  with $V = \Sigma^{-1}$. Note that $\eta  := V^{1/2} \xi \sim \mathcal N(0, I_d)$, and $\Psi^d(\xi) = \|\eta\|^2/2$. Therefore in this case again
$\E_{\pi} [\Psi^d(\xi)] = d/2$, $\Var_{\pi} [\Psi^d(\xi)] = d/2$, $\E_{\pi} [\xi_i1]= 0$, $\Var_{\pi} [\xi_i1] = 1$.
\subsubsection{IID Student}

$\Psi^d(\xi) = \left(\frac{\nu + 1}{2}\right) \sum_{i=1}^d \log \left(1 + \frac {(\xi_i)^2}{\nu} \right)$ and
\[ \pi(x) = \left(  \frac{\Gamma \left( \frac{\nu+1}{2} \right)}{\sqrt{\nu \pi} \Gamma \left( \frac{\nu}{2} \right)} \right)^d \prod_{i=1}^d  \left( 1 + \frac{(\xi_i)^2}{\nu} \right)^{-\frac{\nu+1}{2}}, \quad x \in \RR^d.\]
We have $\Var_{\pi}[\xi_1] = \frac{\nu}{\nu - 2}$.
Furthermore
\[ \E_{\pi}[\Psi^d(\xi)] =  \left(\frac{\nu + 1}{2}\right)  \frac{\Gamma \left( \frac{\nu+1}{2} \right)}{\sqrt{\nu \pi} \Gamma \left( \frac{\nu}{2} \right)} \sum_{i=1}^d \int_{\RR}   \log \left(1 + \frac {(\xi_i)^2}{\nu} \right)  \left( 1 + \frac{(\xi_i)^2}{\nu} \right)^{-\frac{\nu+1}{2}}\, \dif \xi_i, \] which scales linearly with dimension. There are no simple analytic expressions for $\E_{\pi}[\Psi^d(\xi)]$ and $\Var_{\pi}[\Psi^d(\xi)]$ but we may obtain the following values by numeric integration:

\begin{tabular}{c||c|c}
	$\nu$ & $\E_{\pi}[\Psi^d(\xi)]$ & $\Var_{\pi}[\Psi^d(\xi)]$\\
	\hline
	1 & $1.38629 \times d$ & $3.28987 \times d$\\
	2 & $0.920558 \times d$ & $1.59780 \times d$ \\
	3 & $0.772589 \times d$ & $1.15947 \times d$ \\
	4 & $0.700931 \times d$ & $0.966102 \times d$\\
	5 & $0.658883 \times d$ & $0.858813 \times d$
\end{tabular}

\subsubsection{Spherically symmetric Student}
\label{sec:sphericallysymmetricstudent}
The potential is given by $\Psi^d(\xi) = \left( \frac{\nu+ d}{2}\right) \log \left(1 + \frac{\|x\|^2}{\nu} \right)$, and the probability density function is
\[ \pi(x) = \frac{\Gamma \left(\frac{d + \nu}{2} \right) }{(\pi \nu)^{d/2} \Gamma\left( \frac{\nu}{2} \right)} \left( 1 + \frac{\|x\|^2}{\nu} \right)^{-(\nu + d)/2}, \quad x \in \RR^d.\]
We follow \cite{boisbunon2012class} to obtain
the probability density for $T = \|\xi\|^2/d$ given by
\[ h(t) = \frac 1 {t B \left(\frac{\nu}{2}, \frac{d}{2} \right)} \left(d t\right)^{\frac d 2} \nu^{\frac{\nu}{2}} \left(\nu + d t \right)^{-\frac{(\nu + d)}{2}}, \quad t \geq 0,\]
corresponding to a F-distribution with parameters $(d, \nu)$.
By \cite[Section 2.2.2]{FangKotzNg1990}, $\Var_{\pi}[\xi_1] = \E[T] = \frac{\nu}{\nu - 2}$.
Furthermore
\begin{align*}
\E_{\pi}[\Psi^d(\xi)] & = \int_0^{\infty} \left( \frac{\nu+ d}{2}\right) \log \left(1 + \frac{t d}{\nu} \right) h(t) \, \dif t
\end{align*}
The value of this expression may be expressed in terms of special functions or obtained by numerical integration.
We list a few values in the table below:

\begin{tabular}{c|c||c|c} 
	$d$ & $\nu$ & $\E_{\pi}[\Psi^d(\xi)]$ & $\Var_{\pi}[\Psi^d(\xi)]$ \\
	\hline
	1 & 4  & 0.700931 & 0.966102 \\
	2 & 4 & 1.50000 & 2.25000\\
	4 & 4 & 3.33333 & 5.77778\\
	8 & 4 & 7.70000 & 16.6900\\
	16 & 4 & 18.2897 & 53.9768 \\
	32 & 4 & 43.9119 & 190.449\\
	64 & 4 & 105.019 & 711.039\\
	128 & 4 & 248.112 & 2742.83 \\
	256 & 4 & 577.317 & 10768.9
\end{tabular}

\subsection{Ergodic average evaluation}
\label{sec:ergodicaverage}

As part of our numerical computations, we wish to evaluate
\[ \frac 1 T \int_0^T h(\xi(s)) \, \dif s \]
where $(\xi(t))_{0 \leq t \leq T}$ is a trajectory of a piecewise deterministic process associated with negative log density $\Psi^d$.
Suppose $(\xi(t))_{0 \leq t \leq T}$ is determined by the skeleton points and skeleton times
\[ (\xi_0, \xi_1, \xi_2, \dots, \xi_n), \quad (v_0, v_1, \dots, v_{n-1}), \quad (0=t_0< t_1< \dots< t_n = T),\]
i.e. $\xi(s) = \xi_{i-1} + (s - t_{i-1}) v_{i-1}$ for $t_{i-1} \leq s < t_i$.
We may write
\[ \frac 1 T \int_0^T h(\xi(s)) \, \dif s = \frac 1 T \sum_{i=1}^n \int_0^{t_i - t_{i-1}} h\left (\xi_{i-1} + s v_{i-1} \right) \, \dif s \]
To carry out this computation conveniently, we define functions $F(\xi, v, \Delta)$ such that
\[ F(\xi, v, \Delta) = \int_0^{\Delta} h \left(\xi + s  v \right) \, \dif s\]
for all values of $x$, $v$ and $\Delta$. Once we have access to $F$, we can compute
\[ \frac 1 T \int_0^T h(\xi(s)) \, \dif s =   \frac{1}{T} \sum_{i=1}^n F \left(\xi_{i-1}, v_{i-1}, t_i - t_{i-1}\right).\]

\subsubsection*{Remark} It is tempting to use $v_{i-1} = \frac {\xi_i - \xi_{i-1}}{t_i - t_{i-1}}$ in these computations; however this finite difference operation can be numerically unstable. Therefore we evaluate $F$ in the velocities $(v_i)$ as returned by the piecewise deterministic simulation, and do \emph{not} compute these from the values $(\xi_i)$ and $(t_i)$.

It remains to determine the functions $F$ for the examples above.

\subsubsection{Gaussian distribution}
If $h(\xi) = \tfrac 1 2 \xi^{\top} V \xi$, then we have
\begin{align*}
F(\xi, v, \Delta) & = \int_0^{\Delta} \tfrac 1 2 \left(\xi + sv \right)^{\top} V  \left(\xi + sv \right) \, \dif s = \tfrac {\Delta} 2 \xi^T V \xi + \tfrac {\Delta^2} {2} v^T V \xi + \tfrac{ \Delta^3} 6 v^T V v.
\end{align*}

\subsubsection{IID Student}
We have $h(\xi) = \frac{\nu+1}{2} \sum_{j=1}^d \log \left(1 + \frac{\xi_j^2}{\nu} \right)$. Using that
\begin{equation} \label{eq:some-primitive} G(y) := \int_0^y \log(1+r^2) \, d r = y \log (1 +y^2) - 2 y + 2 \arctan y , \end{equation}
we obtain that
\begin{align*}
F(\xi, v, \Delta) & = \left(\frac {\nu+1}{2}\right) \sum_{j=1}^d \int_0^{\Delta} \log \left(1 + \left(\xi_j + s v_j\right)^2/\nu\right) \, \dif s \\
& =  \left( \frac {\nu+1}{2}\right) \sum_{j=1}^d \frac{\sqrt{\nu}}{v_j} \left[ G((\xi_j + \Delta v_j)/\sqrt{\nu}) - G(\xi_j/\sqrt{\nu}) \right]
\end{align*}

\subsubsection{Spherically symmetric Student}
Here
\[ h(\xi) = \left(\frac{ \nu + d}{2} \right) \log \left(1 + \frac{\|\xi\|^2}{\nu} \right).\]
After some manipulations we obtain that
\begin{align*}
& F(\xi, v, \Delta) \\
& = \left( \frac{ \nu+d}{2 } \right) \int_0^{\Delta}\log \left(1 + \frac{\| \xi + s v\|^2}{\nu} \right) \, \dif s \\
& =  \left(\frac{ \nu+d}{2} \right) \left\{\Delta \log (a - b^2/c) + \frac{\sqrt{ a - b^2/c}}{\sqrt{c}} \left[ G \left( \frac{b + \Delta c}{\sqrt{ c a - b^2}} \right) - G \left(\frac{b}{\sqrt{ c a - b^2}} \right) \right]\right\},
\end{align*}
where $G$ is defined as in~\eqref{eq:some-primitive}, and 
\[ a = 1+\frac{\| \xi\|^2}{\nu}, \quad b = \frac{\langle \xi, v\rangle}{\nu}, \quad \text{and} \quad c = \frac{\|v\|^2}{\nu}.\]
As a sanity check, if $d = 1$, then $a - b^2/c = 1$ and we obtain
\[ F(\xi, v, \Delta) = \left( \frac{\nu+1}{2} \right) \frac{\sqrt{\nu}}{v} \left[ G \left((\xi + \Delta v)/\sqrt{\nu} \right) - G(\xi/\sqrt{\nu}) \right], \]
as in the IID Student case.

%

\bibliographystyle{plainnat}

\begin{thebibliography}{37}
\providecommand{\natexlab}[1]{#1}
\providecommand{\url}[1]{\texttt{#1}}
\expandafter\ifx\csname urlstyle\endcsname\relax
  \providecommand{\doi}[1]{doi: #1}\else
  \providecommand{\doi}{doi: \begingroup \urlstyle{rm}\Url}\fi

\bibitem[{Andrieu} et~al.(2018){Andrieu}, {Durmus}, {N{\"u}sken}, and
  {Roussel}]{2018arXiv180808592A}
Christophe {Andrieu}, Alain {Durmus}, Nikolas {N{\"u}sken}, and Julien
  {Roussel}.
\newblock {Hypocoercivity of Piecewise Deterministic Markov Process-Monte
  Carlo}.
\newblock \emph{arXiv e-prints}, art. arXiv:1808.08592, Aug 2018.

\bibitem[Bierkens and Duncan(2017)]{MR3694318}
Joris Bierkens and Andrew Duncan.
\newblock Limit theorems for the zig-zag process.
\newblock \emph{Adv. in Appl. Probab.}, 49\penalty0 (3):\penalty0 791--825,
  2017.
\newblock ISSN 0001-8678.
\newblock \doi{10.1017/apr.2017.22}.
\newblock URL \url{https://doi.org/10.1017/apr.2017.22}.

\bibitem[Bierkens et~al.(2017)Bierkens, Roberts, and Zitt]{BierkensZitt2017}
Joris Bierkens, Gareth~O. Roberts, and Pierre-Andr{\'{e}} Zitt.
\newblock {Ergodicity of the zigzag process}.
\newblock \emph{arXiv preprint arXiv: 1712.09875}, 2017.
\newblock URL \url{https://arxiv.org/pdf/1712.09875.pdf}.

\bibitem[Bierkens et~al.(2018)Bierkens, Fearnhead, and
  Roberts]{BierkensFearnheadRoberts2016}
Joris Bierkens, Paul Fearnhead, and Gareth~O. Roberts.
\newblock {The Zig-Zag Process and Super-Efficient Sampling for Bayesian
  Analysis of Big Data}.
\newblock \emph{Annals of Statistics}, 2018.
\newblock URL \url{https://arxiv.org/abs/1607.03188}.

\bibitem[Billingsley(1999)]{MR1700749}
Patrick Billingsley.
\newblock \emph{Convergence of probability measures}.
\newblock Wiley Series in Probability and Statistics: Probability and
  Statistics. John Wiley \& Sons Inc., New York, second edition, 1999.
\newblock ISBN 0-471-19745-9.
\newblock \doi{10.1002/9780470316962}.
\newblock A Wiley-Interscience Publication.

\bibitem[Boisbunon(2012)]{boisbunon2012class}
Aur{\'{e}}lie Boisbunon.
\newblock {The class of multivariate spherically symmetric distributions}.
\newblock \emph{Universit{\'{e}} de Rouen, Technical Report,{\#} 2012-005},
  2012.

\bibitem[Bouchard-C\^{o}t\'{e} et~al.(2017)Bouchard-C\^{o}t\'{e}, Vollmer, and
  Doucet]{BouchardCote2017}
Alexandre Bouchard-C\^{o}t\'{e}, Sebastian~J. Vollmer, and Arnaud Doucet.
\newblock The bouncy particle sampler: A non-reversible rejection-free markov
  chain monte carlo method.
\newblock \emph{Journal of the American Statistical Association}, 0\penalty0
  (ja):\penalty0 0--0, 2017.
\newblock \doi{10.1080/01621459.2017.1294075}.
\newblock URL \url{https://doi.org/10.1080/01621459.2017.1294075}.

\bibitem[Bouguet and Cloez(2018)]{BouguetCloez2018}
Florian Bouguet and Bertrand Cloez.
\newblock {Fluctuations of the empirical measure of freezing Markov chains}.
\newblock \emph{Electron. J. Probab.}, 23:\penalty0 Paper No. 2, 31, 2018.
\newblock ISSN 1083-6489.
\newblock \doi{10.1214/17-EJP130}.
\newblock URL \url{https://doi.org/10.1214/17-EJP130}.

\bibitem[Chen et~al.(2011)Chen, Goldstein, and Shao]{Stein}
Louis~H.Y. Chen, Larry Goldstein, and Qi-Man Shao.
\newblock \emph{Normal Approximation by Stein's Method}.
\newblock Probability and Its Applications. Berlin, Heidelberg : Springer
  Berlin Heidelberg, 2011., 2011.
\newblock ISBN 9783642150074.

\bibitem[Christensen et~al.(2005)Christensen, Roberts, and
  Rosenthal]{MR2137324}
Ole~F. Christensen, Gareth~O. Roberts, and Jeffrey~S. Rosenthal.
\newblock Scaling limits for the transient phase of local
  {M}etropolis-{H}astings algorithms.
\newblock \emph{J. R. Stat. Soc. Ser. B Stat. Methodol.}, 67\penalty0
  (2):\penalty0 253--268, 2005.
\newblock ISSN 1369-7412.
\newblock \doi{10.1111/j.1467-9868.2005.00500.x}.

\bibitem[Costa and Dufour(2008)]{MR2385873}
O.~L.~V. Costa and F.~Dufour.
\newblock Stability and ergodicity of piecewise deterministic {M}arkov
  processes.
\newblock \emph{SIAM J. Control Optim.}, 47\penalty0 (2):\penalty0 1053--1077,
  2008.
\newblock ISSN 0363-0129.
\newblock \doi{10.1137/060670109}.
\newblock URL \url{https://doi.org/10.1137/060670109}.

\bibitem[Davis(1984)]{Davis1984}
M.~H.~A. Davis.
\newblock {Piecewise-Deterministic Markov Processes: A General Class of
  Non-Diffusion Stochastic Models}.
\newblock \emph{Journal of the Royal Statistical Society. Series B
  (Methodological)}, 46\penalty0 (3):\penalty0 353--388, 1984.
\newblock ISSN 00359246.
\newblock \doi{10.2307/2345677}.
\newblock URL \url{http://www.jstor.org/stable/2345677}.

\bibitem[Deligiannidis et~al.(2017)Deligiannidis, Bouchard-C{\^{o}}t{\'{e}},
  and Doucet]{Deligiannidis2017}
George Deligiannidis, Alexandre Bouchard-C{\^{o}}t{\'{e}}, and Arnaud Doucet.
\newblock {Exponential Ergodicity of the Bouncy Particle Sampler}.
\newblock \emph{arXiv preprint arXiv: 1705.04579}, 2017.
\newblock URL \url{https://arxiv.org/pdf/1705.04579.pdf}.

\bibitem[{Deligiannidis} et~al.(2018){Deligiannidis}, {Paulin},
  {Bouchard-C{\^o}t{\'e}}, and {Doucet}]{2018arXiv180804299D}
George {Deligiannidis}, Daniel {Paulin}, Alexandre {Bouchard-C{\^o}t{\'e}}, and
  Arnaud {Doucet}.
\newblock {Randomized Hamiltonian Monte Carlo as Scaling Limit of the Bouncy
  Particle Sampler and Dimension-Free Convergence Rates}.
\newblock \emph{arXiv e-prints}, art. arXiv:1808.04299, Aug 2018.

\bibitem[Diaconis and Freedman(1987)]{MR898502}
Persi Diaconis and David Freedman.
\newblock A dozen de {F}inetti-style results in search of a theory.
\newblock \emph{Ann. Inst. H. Poincar\'e Probab. Statist.}, 23\penalty0 (2,
  suppl.):\penalty0 397--423, 1987.
\newblock ISSN 0246-0203.

\bibitem[Doob(1953)]{Doob}
J.~L. Doob.
\newblock \emph{Stochastic processes}.
\newblock John Wiley \& Sons Inc., New York, 1953.

\bibitem[Down et~al.(1995)Down, Meyn, and Tweedie]{MR1379163}
D.~Down, S.~P. Meyn, and R.~L. Tweedie.
\newblock Exponential and uniform ergodicity of {M}arkov processes.
\newblock \emph{Ann. Probab.}, 23\penalty0 (4):\penalty0 1671--1691, 1995.
\newblock ISSN 0091-1798.
\newblock URL
  \url{http://links.jstor.org/sici?sici=0091-1798(199510)23:4<1671:EAUEOM>2.0.CO;2-7&origin=MSN}.

\bibitem[Ethier and Kurtz(1986)]{MR838085}
Stewart~N. Ethier and Thomas~G. Kurtz.
\newblock \emph{Markov processes}.
\newblock Wiley Series in Probability and Mathematical Statistics: Probability
  and Mathematical Statistics. John Wiley \& Sons, Inc., New York, 1986.
\newblock ISBN 0-471-08186-8.
\newblock \doi{10.1002/9780470316658}.
\newblock URL \url{http://dx.doi.org/10.1002/9780470316658}.
\newblock Characterization and convergence.

\bibitem[Fang et~al.(1990)Fang, Kotz, and Ng]{FangKotzNg1990}
Kai~Tai Fang, Samuel Kotz, and Kai~Wang Ng.
\newblock \emph{{Symmetric multivariate and related distributions}}, volume~36
  of \emph{Monographs on Statistics and Applied Probability}.
\newblock Chapman and Hall, Ltd., London, 1990.
\newblock ISBN 0-412-31430-4.
\newblock \doi{10.1007/978-1-4899-2937-2}.
\newblock URL
  \url{http://0-dx.doi.org.pugwash.lib.warwick.ac.uk/10.1007/978-1-4899-2937-2}.

\bibitem[Fontbona et~al.(2015)Fontbona, Gu{\'{e}}rin, and
  Malrieu]{Fontbona2015}
J.~Fontbona, H.~Gu{\'{e}}rin, and F.~Malrieu.
\newblock {Long time behavior of telegraph processes under convex potentials}.
\newblock \emph{Stochastic Processes and their Applications}, 126\penalty0
  (10):\penalty0 1--26, 2015.
\newblock ISSN 03044149.
\newblock \doi{10.1016/j.spa.2016.04.002}.
\newblock URL \url{http://arxiv.org/abs/1507.03503}.

\bibitem[Hahn(1978)]{Hahn1978}
Marjorie~G. Hahn.
\newblock Central limit theorems in d[0, 1].
\newblock \emph{Zeitschrift f{\"u}r Wahrscheinlichkeitstheorie und Verwandte
  Gebiete}, 44\penalty0 (2):\penalty0 89--101, Jun 1978.
\newblock ISSN 1432-2064.
\newblock \doi{10.1007/BF00533047}.
\newblock URL \url{https://doi.org/10.1007/BF00533047}.

\bibitem[Ikeda and Watanabe(1989)]{MR1011252}
Nobuyuki Ikeda and Shinzo Watanabe.
\newblock \emph{Stochastic differential equations and diffusion processes},
  volume~24 of \emph{North-Holland Mathematical Library}.
\newblock North-Holland Publishing Co., Amsterdam; Kodansha, Ltd., Tokyo,
  second edition, 1989.
\newblock ISBN 0-444-87378-3.

\bibitem[Jacod and Shiryaev(2003)]{JS}
Jean Jacod and Albert~N. Shiryaev.
\newblock \emph{Limit theorems for stochastic processes.}
\newblock Grundlehren der Mathematischen Wissenschaften. Springer-Verlag,
  Berlin, 2nd edition, 2003.

\bibitem[Jourdain et~al.(2015)Jourdain, Leli\`evre, and Miasojedow]{MR3349007}
Benjamin Jourdain, Tony Leli\`evre, and B\l~a\.zej Miasojedow.
\newblock Optimal scaling for the transient phase of the random walk
  {M}etropolis algorithm: the mean-field limit.
\newblock \emph{Ann. Appl. Probab.}, 25\penalty0 (4):\penalty0 2263--2300,
  2015.
\newblock ISSN 1050-5164.
\newblock URL \url{https://doi.org/10.1214/14-AAP1048}.

\bibitem[Kamatani(2018)]{arXiv:1412.6231}
Kengo Kamatani.
\newblock Efficient strategy for the markov chain monte carlo in high-dimension
  with heavy-tailed target probability distribution.
\newblock \emph{Bernoulli}, 24\penalty0 (4B):\penalty0 3711--3750, 2018.
\newblock ISSN 1350-7265.
\newblock \doi{10.3150/17-BEJ976}.

\bibitem[Karatzas and Shreve(1991)]{Karatzas1991}
Ioannis Karatzas and Steven~E. Shreve.
\newblock \emph{Brownian motion and stochastic calculus}.
\newblock Number 113 in Graduate texts in mathematics. Springer-Verlag, 2nd ed
  edition, 1991.

\bibitem[Kurtz(2011)]{MR2789081}
Thomas~G. Kurtz.
\newblock Equivalence of stochastic equations and martingale problems.
\newblock In \emph{Stochastic analysis 2010}, pages 113--130. Springer,
  Heidelberg, 2011.
\newblock \doi{10.1007/978-3-642-15358-7_6}.
\newblock URL \url{https://doi.org/10.1007/978-3-642-15358-7_6}.

\bibitem[Marcus and Rosen(2006)]{MR2250510}
Michael~B. Marcus and Jay Rosen.
\newblock \emph{Markov processes, {G}aussian processes, and local times},
  volume 100 of \emph{Cambridge Studies in Advanced Mathematics}.
\newblock Cambridge University Press, Cambridge, 2006.
\newblock ISBN 978-0-521-86300-1; 0-521-86300-7.
\newblock \doi{10.1017/CBO9780511617997}.
\newblock URL \url{https://doi.org/10.1017/CBO9780511617997}.

\bibitem[Michel et~al.(2014)Michel, Kapfer, and Krauth]{michel2014generalized}
Manon Michel, Sebastian~C Kapfer, and Werner Krauth.
\newblock Generalized event-chain monte carlo: Constructing rejection-free
  global-balance algorithms from infinitesimal steps.
\newblock \emph{The Journal of chemical physics}, 140\penalty0 (5):\penalty0
  054116, 2014.

\bibitem[Nourdin and Peccati(2012)]{MR2962301}
Ivan Nourdin and Giovanni Peccati.
\newblock \emph{Normal approximations with {M}alliavin calculus. From Stein's
  method to universality}, volume 192 of \emph{Cambridge Tracts in
  Mathematics}.
\newblock Cambridge University Press, Cambridge, 2012.
\newblock ISBN 978-1-107-01777-1.
\newblock \doi{10.1017/CBO9781139084659}.

\bibitem[Pakman et~al.(2016)Pakman, Gilboa, Carlson, and
  Paninski]{pakman2016stochastic}
Ari Pakman, Dar Gilboa, David Carlson, and Liam Paninski.
\newblock Stochastic bouncy particle sampler.
\newblock \emph{arXiv preprint arXiv:1609.00770}, 2016.

\bibitem[Peters and de~With(2012)]{peters2012rejection}
Elias~AJF Peters and G.~de~With.
\newblock Rejection-free monte carlo sampling for general potentials.
\newblock \emph{Physical Review E}, 85\penalty0 (2):\penalty0 026703, 2012.

\bibitem[Revuz and Yor(1999)]{MR1725357}
Daniel Revuz and Marc Yor.
\newblock \emph{Continuous martingales and {B}rownian motion}, volume 293 of
  \emph{Grundlehren der Mathematischen Wissenschaften [Fundamental Principles
  of Mathematical Sciences]}.
\newblock Springer-Verlag, Berlin, third edition, 1999.
\newblock ISBN 3-540-64325-7.
\newblock URL
  \url{https://doi-org.remote.library.osaka-u.ac.jp:8443/10.1007/978-3-662-06400-9}.

\bibitem[Roberts and Rosenthal(2001)]{MR1888450}
Gareth~O. Roberts and Jeffrey~S. Rosenthal.
\newblock Optimal scaling for various {M}etropolis-{H}astings algorithms.
\newblock \emph{Statist. Sci.}, 16\penalty0 (4):\penalty0 351--367, 2001.
\newblock ISSN 0883-4237.
\newblock \doi{10.1214/ss/1015346320}.

\bibitem[Roberts and Rosenthal(2016)]{roberts2014complexity}
Gareth~O. Roberts and Jeffrey~S. Rosenthal.
\newblock Complexity bounds for mcmc via diffusion limits.
\newblock \emph{Journal of Applied Probability}, 53:\penalty0 410--420, 2016.

\bibitem[Roberts et~al.(1997)Roberts, Gelman, and Gilks]{RGG}
Gareth~O. Roberts, Andrew Gelman, and Walter~R. Gilks.
\newblock Weak convergence and optimal scaling of random walk {M}etropolis
  algorithms.
\newblock \emph{Ann. Appl. Probab.}, 7\penalty0 (1):\penalty0 110--120, 1997.
\newblock ISSN 1050-5164.
\newblock \doi{10.1214/aoap/1034625254}.

\bibitem[Vanetti et~al.(2017)Vanetti, Bouchard-C{\^o}t{\'e}, Deligiannidis, and
  Doucet]{vanetti2017piecewise}
Paul Vanetti, Alexandre Bouchard-C{\^o}t{\'e}, George Deligiannidis, and Arnaud
  Doucet.
\newblock Piecewise deterministic markov chain monte carlo.
\newblock \emph{arXiv preprint arXiv:1707.05296}, 2017.

\end{thebibliography}
%

\end{document}